\documentclass[a4paper,leqno,11pt]{amsart}
\usepackage[utf8]{inputenc}
\usepackage[margin = 17mm]{geometry}
\usepackage[T1]{fontenc}
\usepackage{graphicx,amsmath,mathrsfs,amssymb,amsthm,amsfonts,float, tikz,tikz-cd,mathtools,bm,csquotes}

\usepackage[style=alphabetic]{biblatex}
\usepackage[pdfencoding=auto]{hyperref}
\hypersetup{colorlinks=true, allcolors=blue}        

\makeatletter
\def\@tocline#1#2#3#4#5#6#7{\relax
  \ifnum #1>\c@tocdepth 
  \else
    \par \addpenalty\@secpenalty\addvspace{#2}%
    \begingroup \hyphenpenalty\@M
    \@ifempty{#4}{%
      \@tempdima\csname r@tocindent\number#1\endcsname\relax
    }{%
      \@tempdima#4\relax
    }%
    \parindent\z@ \leftskip#3\relax \advance\leftskip\@tempdima\relax
    \rightskip\@pnumwidth plus4em \parfillskip-\@pnumwidth
    #5\leavevmode\hskip-\@tempdima
      \ifcase #1
       \or\or \hskip 1em \or \hskip 2em \else \hskip 3em \fi%
      #6\nobreak\relax
    \hfill\hbox to\@pnumwidth{\@tocpagenum{#7}}\par
    \nobreak
    \endgroup
  \fi}
\makeatother

\title{Quiver varieties for affine orthogonal quivers}

\author{Victor Pinot}
\address{Université de Lille, CNRS, UMR 8524 -- Laboratoire Paul Painlevé}
\email{victor.pinot@univ-lille.fr}
\date{\today}

\DeclareMathOperator{\NN}{\mathbb{N}}
\DeclareMathOperator{\ZZ}{\mathbb{Z}}
\DeclareMathOperator{\CC}{\mathbb{C}}
\DeclareMathOperator{\PP}{\mathbb{P}}
\DeclareMathOperator{\OO}{\mathrm{O}}
\DeclareMathOperator{\Sp}{\mathrm{Sp}}
\DeclareMathOperator{\GL}{\mathrm{GL}}
\DeclareMathOperator{\ggot}{\mathfrak{g}}

\DeclareMathOperator{\Rep}{\mathrm{Rep}}
\DeclareMathOperator{\Srep}{\mathrm{SRep}}
\DeclareMathOperator{\repinvolution}{\sigma}
\DeclareMathOperator{\lieinvolution}{\varsigma}
\DeclareMathOperator{\Sym}{S^2}
\DeclareMathOperator{\Hom}{\mathrm{Hom}_{\CC}}
\DeclareMathOperator{\tr}{tr}
\DeclareMathOperator{\transpose}{\mathsf t}
\DeclareMathOperator{\dd}{\mathrm{d}}

\DeclareMathOperator{\Mgot}{\mathfrak{M}}
\DeclareMathOperator{\Mgotgras}{\boldsymbol{\Mgot}}
\DeclareMathOperator{\mugras}{\boldsymbol{\mu}}

\DeclareMathOperator{\inj}{\hookrightarrow}
\DeclareMathOperator{\surj}{\twoheadrightarrow}
\DeclareMathOperator{\action}{\curvearrowright}
\DeclareMathOperator{\opp}{\mathrm{opp}}

\DeclareMathOperator{\Mor}{\mathrm{Mor}}
\DeclareMathOperator{\End}{\mathrm{End}}
\DeclareMathOperator{\VV}{\mathbb{V}}
\DeclareMathOperator{\Spec}{\mathrm{Spec}}
\DeclareMathOperator{\pt}{\mathrm{reduced \ point}}

\renewcommand*{\multicitedelim}{\addcomma\space}
\newcommand{\cycleCone}[1][1]{%
\mathord{%
\hspace{+2mm}%
\raisebox{-0.4\height}{%
\scalebox{#1}{%
\begin{tikzpicture}[x=0.3pt,y=0.3pt,yscale=-1,xscale=1]

\draw  [color={rgb, 255:red, 0; green, 0; blue, 0 }  ,draw opacity=1 ][fill={rgb, 255:red, 0; green, 0; blue, 0 }  ,fill opacity=1 ][line width=0.75]  (272,151) .. controls (272,152.66) and (273.34,154) .. (275,154) .. controls (276.66,154) and (278,152.66) .. (278,151) .. controls (278,149.34) and (276.66,148) .. (275,148) .. controls (273.34,148) and (272,149.34) .. (272,151) -- cycle ;
\draw  [color={rgb, 255:red, 0; green, 0; blue, 0 }  ,draw opacity=1 ][fill={rgb, 255:red, 0; green, 0; blue, 0 }  ,fill opacity=1 ][line width=0.75]  (422,151) .. controls (422,152.66) and (423.34,154) .. (425,154) .. controls (426.66,154) and (428,152.66) .. (428,151) .. controls (428,149.34) and (426.66,148) .. (425,148) .. controls (423.34,148) and (422,149.34) .. (422,151) -- cycle ;
\draw  [color={rgb, 255:red, 0; green, 0; blue, 0 }  ,draw opacity=1 ][fill={rgb, 255:red, 0; green, 0; blue, 0 }  ,fill opacity=1 ][line width=0.75]  (307.5,193) .. controls (307.5,194.66) and (308.84,196) .. (310.5,196) .. controls (312.16,196) and (313.5,194.66) .. (313.5,193) .. controls (313.5,191.34) and (312.16,190) .. (310.5,190) .. controls (308.84,190) and (307.5,191.34) .. (307.5,193) -- cycle ;
\draw  [color={rgb, 255:red, 0; green, 0; blue, 0 }  ,draw opacity=1 ][fill={rgb, 255:red, 0; green, 0; blue, 0 }  ,fill opacity=1 ][line width=0.75]  (386.5,193) .. controls (386.5,194.66) and (387.84,196) .. (389.5,196) .. controls (391.16,196) and (392.5,194.66) .. (392.5,193) .. controls (392.5,191.34) and (391.16,190) .. (389.5,190) .. controls (387.84,190) and (386.5,191.34) .. (386.5,193) -- cycle ;
\draw  [color={rgb, 255:red, 0; green, 0; blue, 0 }  ,draw opacity=1 ][fill={rgb, 255:red, 0; green, 0; blue, 0 }  ,fill opacity=1 ][line width=0.75]  (307.5,108) .. controls (307.5,109.66) and (308.84,111) .. (310.5,111) .. controls (312.16,111) and (313.5,109.66) .. (313.5,108) .. controls (313.5,106.34) and (312.16,105) .. (310.5,105) .. controls (308.84,105) and (307.5,106.34) .. (307.5,108) -- cycle ;
\draw  [color={rgb, 255:red, 0; green, 0; blue, 0 }  ,draw opacity=1 ][fill={rgb, 255:red, 0; green, 0; blue, 0 }  ,fill opacity=1 ][line width=0.75]  (386.5,108) .. controls (386.5,109.66) and (387.84,111) .. (389.5,111) .. controls (391.16,111) and (392.5,109.66) .. (392.5,108) .. controls (392.5,106.34) and (391.16,105) .. (389.5,105) .. controls (387.84,105) and (386.5,106.34) .. (386.5,108) -- cycle ;
\draw [line width=0.75]    (277.29,157.07) .. controls (284.18,170.4) and (292.13,182.42) .. (312.36,187.98) .. controls (332.58,193.53) and (368.58,193.98) .. (387.69,187.76) .. controls (406.8,181.53) and (418.8,165.31) .. (418.58,151.09) .. controls (418.36,136.87) and (403.87,118.67) .. (385.47,112.64) .. controls (367.07,106.62) and (331.64,108.89) .. (311.91,113.31) .. controls (293.26,117.49) and (284.89,128.24) .. (278.03,142.35) ;
\draw [shift={(276.84,144.84)}, rotate = 295.95] [fill={rgb, 255:red, 0; green, 0; blue, 0 }  ][line width=0.08]  [draw opacity=0] (10.72,-5.15) -- (0,0) -- (10.72,5.15) -- (7.12,0) -- cycle    ;
\end{tikzpicture}
}%
}%
}%
}

\newcommand{\cycleConehaut}[1][1]{%
\mathord{%
\hspace{+2mm}%
\raisebox{-0.4\height}{%
\scalebox{#1}{%
\begin{tikzpicture}[x=0.3pt,y=0.3pt,yscale=-1,xscale=1]

\draw  [color={rgb, 255:red, 0; green, 0; blue, 0 }  ,draw opacity=1 ][fill={rgb, 255:red, 0; green, 0; blue, 0 }  ,fill opacity=1 ][line width=0.75]  (272,151) .. controls (272,152.66) and (273.34,154) .. (275,154) .. controls (276.66,154) and (278,152.66) .. (278,151) .. controls (278,149.34) and (276.66,148) .. (275,148) .. controls (273.34,148) and (272,149.34) .. (272,151) -- cycle ;
\draw  [color={rgb, 255:red, 0; green, 0; blue, 0 }  ,draw opacity=1 ][fill={rgb, 255:red, 0; green, 0; blue, 0 }  ,fill opacity=1 ][line width=0.75]  (422,151) .. controls (422,152.66) and (423.34,154) .. (425,154) .. controls (426.66,154) and (428,152.66) .. (428,151) .. controls (428,149.34) and (426.66,148) .. (425,148) .. controls (423.34,148) and (422,149.34) .. (422,151) -- cycle ;
\draw  [color={rgb, 255:red, 0; green, 0; blue, 0 }  ,draw opacity=1 ][fill={rgb, 255:red, 0; green, 0; blue, 0 }  ,fill opacity=1 ][line width=0.75]  (307.5,193) .. controls (307.5,194.66) and (308.84,196) .. (310.5,196) .. controls (312.16,196) and (313.5,194.66) .. (313.5,193) .. controls (313.5,191.34) and (312.16,190) .. (310.5,190) .. controls (308.84,190) and (307.5,191.34) .. (307.5,193) -- cycle ;
\draw  [color={rgb, 255:red, 0; green, 0; blue, 0 }  ,draw opacity=1 ][fill={rgb, 255:red, 0; green, 0; blue, 0 }  ,fill opacity=1 ][line width=0.75]  (386.5,193) .. controls (386.5,194.66) and (387.84,196) .. (389.5,196) .. controls (391.16,196) and (392.5,194.66) .. (392.5,193) .. controls (392.5,191.34) and (391.16,190) .. (389.5,190) .. controls (387.84,190) and (386.5,191.34) .. (386.5,193) -- cycle ;
\draw  [color={rgb, 255:red, 0; green, 0; blue, 0 }  ,draw opacity=1 ][fill={rgb, 255:red, 0; green, 0; blue, 0 }  ,fill opacity=1 ][line width=0.75]  (307.5,108) .. controls (307.5,109.66) and (308.84,111) .. (310.5,111) .. controls (312.16,111) and (313.5,109.66) .. (313.5,108) .. controls (313.5,106.34) and (312.16,105) .. (310.5,105) .. controls (308.84,105) and (307.5,106.34) .. (307.5,108) -- cycle ;
\draw  [color={rgb, 255:red, 0; green, 0; blue, 0 }  ,draw opacity=1 ][fill={rgb, 255:red, 0; green, 0; blue, 0 }  ,fill opacity=1 ][line width=0.75]  (386.5,108) .. controls (386.5,109.66) and (387.84,111) .. (389.5,111) .. controls (391.16,111) and (392.5,109.66) .. (392.5,108) .. controls (392.5,106.34) and (391.16,105) .. (389.5,105) .. controls (387.84,105) and (386.5,106.34) .. (386.5,108) -- cycle ;
\draw [line width=0.75]    (418,148.1) .. controls (411,135.1) and (403.87,118.67) .. (385.47,112.64) .. controls (367.07,106.62) and (331.64,108.89) .. (311.91,113.31) .. controls (293.26,117.49) and (290.67,131.89) .. (282.56,145.28) ;
\draw [shift={(281.09,147.6)}, rotate = 301.2] [fill={rgb, 255:red, 0; green, 0; blue, 0 }  ][line width=0.08]  [draw opacity=0] (10.72,-5.15) -- (0,0) -- (10.72,5.15) -- (7.12,0) -- cycle    ;
\end{tikzpicture}
}%
}%
}%
}

\newcommand{\cycleConebas}[1][1]{%
\mathord{%
\hspace{+2mm}%
\raisebox{-0.4\height}{%
\scalebox{#1}{%
\begin{tikzpicture}[x=0.3pt,y=0.3pt,yscale=-1,xscale=1]

\draw  [color={rgb, 255:red, 0; green, 0; blue, 0 }  ,draw opacity=1 ][fill={rgb, 255:red, 0; green, 0; blue, 0 }  ,fill opacity=1 ][line width=0.75]  (246,142) .. controls (246,143.66) and (247.34,145) .. (249,145) .. controls (250.66,145) and (252,143.66) .. (252,142) .. controls (252,140.34) and (250.66,139) .. (249,139) .. controls (247.34,139) and (246,140.34) .. (246,142) -- cycle ;
\draw  [color={rgb, 255:red, 0; green, 0; blue, 0 }  ,draw opacity=1 ][fill={rgb, 255:red, 0; green, 0; blue, 0 }  ,fill opacity=1 ][line width=0.75]  (396,142) .. controls (396,143.66) and (397.34,145) .. (399,145) .. controls (400.66,145) and (402,143.66) .. (402,142) .. controls (402,140.34) and (400.66,139) .. (399,139) .. controls (397.34,139) and (396,140.34) .. (396,142) -- cycle ;
\draw  [color={rgb, 255:red, 0; green, 0; blue, 0 }  ,draw opacity=1 ][fill={rgb, 255:red, 0; green, 0; blue, 0 }  ,fill opacity=1 ][line width=0.75]  (281.5,184) .. controls (281.5,185.66) and (282.84,187) .. (284.5,187) .. controls (286.16,187) and (287.5,185.66) .. (287.5,184) .. controls (287.5,182.34) and (286.16,181) .. (284.5,181) .. controls (282.84,181) and (281.5,182.34) .. (281.5,184) -- cycle ;
\draw  [color={rgb, 255:red, 0; green, 0; blue, 0 }  ,draw opacity=1 ][fill={rgb, 255:red, 0; green, 0; blue, 0 }  ,fill opacity=1 ][line width=0.75]  (360.5,184) .. controls (360.5,185.66) and (361.84,187) .. (363.5,187) .. controls (365.16,187) and (366.5,185.66) .. (366.5,184) .. controls (366.5,182.34) and (365.16,181) .. (363.5,181) .. controls (361.84,181) and (360.5,182.34) .. (360.5,184) -- cycle ;
\draw  [color={rgb, 255:red, 0; green, 0; blue, 0 }  ,draw opacity=1 ][fill={rgb, 255:red, 0; green, 0; blue, 0 }  ,fill opacity=1 ][line width=0.75]  (281.5,99) .. controls (281.5,100.66) and (282.84,102) .. (284.5,102) .. controls (286.16,102) and (287.5,100.66) .. (287.5,99) .. controls (287.5,97.34) and (286.16,96) .. (284.5,96) .. controls (282.84,96) and (281.5,97.34) .. (281.5,99) -- cycle ;
\draw  [color={rgb, 255:red, 0; green, 0; blue, 0 }  ,draw opacity=1 ][fill={rgb, 255:red, 0; green, 0; blue, 0 }  ,fill opacity=1 ][line width=0.75]  (360.5,99) .. controls (360.5,100.66) and (361.84,102) .. (363.5,102) .. controls (365.16,102) and (366.5,100.66) .. (366.5,99) .. controls (366.5,97.34) and (365.16,96) .. (363.5,96) .. controls (361.84,96) and (360.5,97.34) .. (360.5,99) -- cycle ;
\draw [line width=0.75]    (251.29,148.07) .. controls (258.18,161.4) and (266.13,173.42) .. (286.36,178.98) .. controls (306.58,184.53) and (342.58,184.98) .. (361.69,178.76) .. controls (379.65,172.91) and (385.62,163.93) .. (393.47,151.52) ;
\draw [shift={(395,149.1)}, rotate = 122.31] [fill={rgb, 255:red, 0; green, 0; blue, 0 }  ][line width=0.08]  [draw opacity=0] (10.72,-5.15) -- (0,0) -- (10.72,5.15) -- (7.12,0) -- cycle    ;
\end{tikzpicture}
}%
}%
}%
}

\newcommand{\cycleConebashaut}[1][1]{%
\mathord{%
\hspace{+2mm}%
\raisebox{-0.4\height}{%
\scalebox{#1}{%
\begin{tikzpicture}[x=0.3pt,y=0.3pt,yscale=-1,xscale=1]

\draw  [color={rgb, 255:red, 0; green, 0; blue, 0 }  ,draw opacity=1 ][fill={rgb, 255:red, 0; green, 0; blue, 0 }  ,fill opacity=1 ][line width=0.75]  (253,154) .. controls (253,152.34) and (254.34,151) .. (256,151) .. controls (257.66,151) and (259,152.34) .. (259,154) .. controls (259,155.66) and (257.66,157) .. (256,157) .. controls (254.34,157) and (253,155.66) .. (253,154) -- cycle ;
\draw  [color={rgb, 255:red, 0; green, 0; blue, 0 }  ,draw opacity=1 ][fill={rgb, 255:red, 0; green, 0; blue, 0 }  ,fill opacity=1 ][line width=0.75]  (403,154) .. controls (403,152.34) and (404.34,151) .. (406,151) .. controls (407.66,151) and (409,152.34) .. (409,154) .. controls (409,155.66) and (407.66,157) .. (406,157) .. controls (404.34,157) and (403,155.66) .. (403,154) -- cycle ;
\draw  [color={rgb, 255:red, 0; green, 0; blue, 0 }  ,draw opacity=1 ][fill={rgb, 255:red, 0; green, 0; blue, 0 }  ,fill opacity=1 ][line width=0.75]  (288.5,112) .. controls (288.5,110.34) and (289.84,109) .. (291.5,109) .. controls (293.16,109) and (294.5,110.34) .. (294.5,112) .. controls (294.5,113.66) and (293.16,115) .. (291.5,115) .. controls (289.84,115) and (288.5,113.66) .. (288.5,112) -- cycle ;
\draw  [color={rgb, 255:red, 0; green, 0; blue, 0 }  ,draw opacity=1 ][fill={rgb, 255:red, 0; green, 0; blue, 0 }  ,fill opacity=1 ][line width=0.75]  (367.5,112) .. controls (367.5,110.34) and (368.84,109) .. (370.5,109) .. controls (372.16,109) and (373.5,110.34) .. (373.5,112) .. controls (373.5,113.66) and (372.16,115) .. (370.5,115) .. controls (368.84,115) and (367.5,113.66) .. (367.5,112) -- cycle ;
\draw  [color={rgb, 255:red, 0; green, 0; blue, 0 }  ,draw opacity=1 ][fill={rgb, 255:red, 0; green, 0; blue, 0 }  ,fill opacity=1 ][line width=0.75]  (288.5,197) .. controls (288.5,195.34) and (289.84,194) .. (291.5,194) .. controls (293.16,194) and (294.5,195.34) .. (294.5,197) .. controls (294.5,198.66) and (293.16,200) .. (291.5,200) .. controls (289.84,200) and (288.5,198.66) .. (288.5,197) -- cycle ;
\draw  [color={rgb, 255:red, 0; green, 0; blue, 0 }  ,draw opacity=1 ][fill={rgb, 255:red, 0; green, 0; blue, 0 }  ,fill opacity=1 ][line width=0.75]  (367.5,197) .. controls (367.5,195.34) and (368.84,194) .. (370.5,194) .. controls (372.16,194) and (373.5,195.34) .. (373.5,197) .. controls (373.5,198.66) and (372.16,200) .. (370.5,200) .. controls (368.84,200) and (367.5,198.66) .. (367.5,197) -- cycle ;
\draw [line width=0.75]    (258.29,147.93) .. controls (265.18,134.6) and (273.13,122.58) .. (293.36,117.02) .. controls (313.58,111.47) and (349.58,111.02) .. (368.69,117.24) .. controls (386.65,123.09) and (392.62,132.07) .. (400.47,144.48) ;
\draw [shift={(402,146.9)}, rotate = 237.69] [fill={rgb, 255:red, 0; green, 0; blue, 0 }  ][line width=0.08]  [draw opacity=0] (10.72,-5.15) -- (0,0) -- (10.72,5.15) -- (7.12,0) -- cycle    ;
\end{tikzpicture}
}%
}%
}%
}

\newcommand{\cycleConehomotopietrivial}[1][1]{%
\mathord{%
\hspace{+2mm}%
\raisebox{-0.4\height}{%
\scalebox{#1}{%
\begin{tikzpicture}[x=0.3pt,y=0.3pt,yscale=-1,xscale=1]

\draw  [color={rgb, 255:red, 0; green, 0; blue, 0 }  ,draw opacity=1 ][fill={rgb, 255:red, 0; green, 0; blue, 0 }  ,fill opacity=1 ][line width=0.75]  (272,151) .. controls (272,152.66) and (273.34,154) .. (275,154) .. controls (276.66,154) and (278,152.66) .. (278,151) .. controls (278,149.34) and (276.66,148) .. (275,148) .. controls (273.34,148) and (272,149.34) .. (272,151) -- cycle ;
\draw  [color={rgb, 255:red, 0; green, 0; blue, 0 }  ,draw opacity=1 ][fill={rgb, 255:red, 0; green, 0; blue, 0 }  ,fill opacity=1 ][line width=0.75]  (422,151) .. controls (422,152.66) and (423.34,154) .. (425,154) .. controls (426.66,154) and (428,152.66) .. (428,151) .. controls (428,149.34) and (426.66,148) .. (425,148) .. controls (423.34,148) and (422,149.34) .. (422,151) -- cycle ;
\draw  [color={rgb, 255:red, 0; green, 0; blue, 0 }  ,draw opacity=1 ][fill={rgb, 255:red, 0; green, 0; blue, 0 }  ,fill opacity=1 ][line width=0.75]  (307.5,193) .. controls (307.5,194.66) and (308.84,196) .. (310.5,196) .. controls (312.16,196) and (313.5,194.66) .. (313.5,193) .. controls (313.5,191.34) and (312.16,190) .. (310.5,190) .. controls (308.84,190) and (307.5,191.34) .. (307.5,193) -- cycle ;
\draw  [color={rgb, 255:red, 0; green, 0; blue, 0 }  ,draw opacity=1 ][fill={rgb, 255:red, 0; green, 0; blue, 0 }  ,fill opacity=1 ][line width=0.75]  (386.5,193) .. controls (386.5,194.66) and (387.84,196) .. (389.5,196) .. controls (391.16,196) and (392.5,194.66) .. (392.5,193) .. controls (392.5,191.34) and (391.16,190) .. (389.5,190) .. controls (387.84,190) and (386.5,191.34) .. (386.5,193) -- cycle ;
\draw  [color={rgb, 255:red, 0; green, 0; blue, 0 }  ,draw opacity=1 ][fill={rgb, 255:red, 0; green, 0; blue, 0 }  ,fill opacity=1 ][line width=0.75]  (307.5,108) .. controls (307.5,109.66) and (308.84,111) .. (310.5,111) .. controls (312.16,111) and (313.5,109.66) .. (313.5,108) .. controls (313.5,106.34) and (312.16,105) .. (310.5,105) .. controls (308.84,105) and (307.5,106.34) .. (307.5,108) -- cycle ;
\draw  [color={rgb, 255:red, 0; green, 0; blue, 0 }  ,draw opacity=1 ][fill={rgb, 255:red, 0; green, 0; blue, 0 }  ,fill opacity=1 ][line width=0.75]  (386.5,108) .. controls (386.5,109.66) and (387.84,111) .. (389.5,111) .. controls (391.16,111) and (392.5,109.66) .. (392.5,108) .. controls (392.5,106.34) and (391.16,105) .. (389.5,105) .. controls (387.84,105) and (386.5,106.34) .. (386.5,108) -- cycle ;
\draw [line width=0.75]    (274.33,143.43) .. controls (279.67,131.77) and (292.67,106.43) .. (308.67,102.43) .. controls (324.67,98.43) and (380.33,96.77) .. (393.33,101.1) .. controls (406.33,105.43) and (440,149.77) .. (428.67,157.1) .. controls (417.33,164.43) and (403.87,118.67) .. (385.47,112.64) .. controls (367.07,106.62) and (331.64,108.89) .. (311.91,113.31) .. controls (293.26,117.49) and (290.67,131.89) .. (282.56,145.28) ;
\draw [shift={(281.09,147.6)}, rotate = 301.2] [fill={rgb, 255:red, 0; green, 0; blue, 0 }  ][line width=0.08]  [draw opacity=0] (10.72,-5.15) -- (0,0) -- (10.72,5.15) -- (7.12,0) -- cycle    ;
\end{tikzpicture}
}%
}%
}%
}

\newcommand{\phasmeCtwo}[1][1]{%
\mathord{%
\hspace{+2mm}%
\raisebox{-0.4\height}{%
\scalebox{#1}{%
\begin{tikzpicture}[x=0.3pt,y=0.3pt,yscale=-1,xscale=1]

\draw  [color={rgb, 255:red, 0; green, 0; blue, 0 }  ,draw opacity=1 ][fill={rgb, 255:red, 0; green, 0; blue, 0 }  ,fill opacity=1 ][line width=0.75]  (270,153) .. controls (270,151.34) and (271.34,150) .. (273,150) .. controls (274.66,150) and (276,151.34) .. (276,153) .. controls (276,154.66) and (274.66,156) .. (273,156) .. controls (271.34,156) and (270,154.66) .. (270,153) -- cycle ;
\draw  [color={rgb, 255:red, 0; green, 0; blue, 0 }  ,draw opacity=1 ][fill={rgb, 255:red, 0; green, 0; blue, 0 }  ,fill opacity=1 ][line width=0.75]  (325,153) .. controls (325,151.34) and (326.34,150) .. (328,150) .. controls (329.66,150) and (331,151.34) .. (331,153) .. controls (331,154.66) and (329.66,156) .. (328,156) .. controls (326.34,156) and (325,154.66) .. (325,153) -- cycle ;
\draw  [color={rgb, 255:red, 0; green, 0; blue, 0 }  ,draw opacity=1 ][fill={rgb, 255:red, 0; green, 0; blue, 0 }  ,fill opacity=1 ][line width=0.75]  (430,153) .. controls (430,151.34) and (431.34,150) .. (433,150) .. controls (434.66,150) and (436,151.34) .. (436,153) .. controls (436,154.66) and (434.66,156) .. (433,156) .. controls (431.34,156) and (430,154.66) .. (430,153) -- cycle ;
\draw  [color={rgb, 255:red, 0; green, 0; blue, 0 }  ,draw opacity=1 ][fill={rgb, 255:red, 0; green, 0; blue, 0 }  ,fill opacity=1 ] (240,203) .. controls (240,201.34) and (241.34,200) .. (243,200) .. controls (244.66,200) and (246,201.34) .. (246,203) .. controls (246,204.66) and (244.66,206) .. (243,206) .. controls (241.34,206) and (240,204.66) .. (240,203) -- cycle ;
\draw  [color={rgb, 255:red, 0; green, 0; blue, 0 }  ,draw opacity=1 ][fill={rgb, 255:red, 0; green, 0; blue, 0 }  ,fill opacity=1 ] (240,103) .. controls (240,101.34) and (241.34,100) .. (243,100) .. controls (244.66,100) and (246,101.34) .. (246,103) .. controls (246,104.66) and (244.66,106) .. (243,106) .. controls (241.34,106) and (240,104.66) .. (240,103) -- cycle ;
\draw  [color={rgb, 255:red, 0; green, 0; blue, 0 }  ,draw opacity=1 ][fill={rgb, 255:red, 0; green, 0; blue, 0 }  ,fill opacity=1 ] (460,203) .. controls (460,201.34) and (461.34,200) .. (463,200) .. controls (464.66,200) and (466,201.34) .. (466,203) .. controls (466,204.66) and (464.66,206) .. (463,206) .. controls (461.34,206) and (460,204.66) .. (460,203) -- cycle ;
\draw  [color={rgb, 255:red, 0; green, 0; blue, 0 }  ,draw opacity=1 ][fill={rgb, 255:red, 0; green, 0; blue, 0 }  ,fill opacity=1 ][line width=0.75]  (460,103) .. controls (460,101.34) and (461.34,100) .. (463,100) .. controls (464.66,100) and (466,101.34) .. (466,103) .. controls (466,104.66) and (464.66,106) .. (463,106) .. controls (461.34,106) and (460,104.66) .. (460,103) -- cycle ;
\draw  [color={rgb, 255:red, 0; green, 0; blue, 0 }  ,draw opacity=1 ][fill={rgb, 255:red, 0; green, 0; blue, 0 }  ,fill opacity=1 ][line width=0.75]  (375,153) .. controls (375,151.34) and (376.34,150) .. (378,150) .. controls (379.66,150) and (381,151.34) .. (381,153) .. controls (381,154.66) and (379.66,156) .. (378,156) .. controls (376.34,156) and (375,154.66) .. (375,153) -- cycle ;
\draw [line width=0.75]    (278.73,146.43) .. controls (296.4,146.1) and (319.07,146.1) .. (328.73,145.77) .. controls (338.4,145.43) and (366.4,146.43) .. (377.07,146.43) .. controls (387.73,146.43) and (419.07,152.43) .. (430.07,144.77) .. controls (441.07,137.1) and (452.07,88.43) .. (466.07,95.1) .. controls (480.07,101.77) and (452.07,154.77) .. (438.4,158.77) .. controls (424.73,162.77) and (364.73,159.43) .. (349.73,159.43) .. controls (335.48,159.43) and (300.78,163.65) .. (283.34,158.06) ;
\draw [shift={(280.73,157.1)}, rotate = 16.35] [fill={rgb, 255:red, 0; green, 0; blue, 0 }  ][line width=0.08]  [draw opacity=0] (10.72,-5.15) -- (0,0) -- (10.72,5.15) -- (7.12,0) -- cycle    ;
\end{tikzpicture}
}%
}%
}%
}

\newcommand{\riri}[1][1]{%
\mathord{%
\hspace{+2mm}%
\raisebox{-0.4\height}{%
\scalebox{#1}{%
\begin{tikzpicture}[x=0.3pt,y=0.3pt,yscale=-1,xscale=1]

\draw  [color={rgb, 255:red, 0; green, 0; blue, 0 }  ,draw opacity=1 ][fill={rgb, 255:red, 0; green, 0; blue, 0 }  ,fill opacity=1 ][line width=0.75]  (239,140) .. controls (239,138.34) and (240.34,137) .. (242,137) .. controls (243.66,137) and (245,138.34) .. (245,140) .. controls (245,141.66) and (243.66,143) .. (242,143) .. controls (240.34,143) and (239,141.66) .. (239,140) -- cycle ;
\draw  [color={rgb, 255:red, 0; green, 0; blue, 0 }  ,draw opacity=1 ][fill={rgb, 255:red, 0; green, 0; blue, 0 }  ,fill opacity=1 ][line width=0.75]  (294,140) .. controls (294,138.34) and (295.34,137) .. (297,137) .. controls (298.66,137) and (300,138.34) .. (300,140) .. controls (300,141.66) and (298.66,143) .. (297,143) .. controls (295.34,143) and (294,141.66) .. (294,140) -- cycle ;
\draw  [color={rgb, 255:red, 0; green, 0; blue, 0 }  ,draw opacity=1 ][fill={rgb, 255:red, 0; green, 0; blue, 0 }  ,fill opacity=1 ][line width=0.75]  (399,140) .. controls (399,138.34) and (400.34,137) .. (402,137) .. controls (403.66,137) and (405,138.34) .. (405,140) .. controls (405,141.66) and (403.66,143) .. (402,143) .. controls (400.34,143) and (399,141.66) .. (399,140) -- cycle ;
\draw  [color={rgb, 255:red, 0; green, 0; blue, 0 }  ,draw opacity=1 ][fill={rgb, 255:red, 0; green, 0; blue, 0 }  ,fill opacity=1 ] (209,190) .. controls (209,188.34) and (210.34,187) .. (212,187) .. controls (213.66,187) and (215,188.34) .. (215,190) .. controls (215,191.66) and (213.66,193) .. (212,193) .. controls (210.34,193) and (209,191.66) .. (209,190) -- cycle ;
\draw  [color={rgb, 255:red, 0; green, 0; blue, 0 }  ,draw opacity=1 ][fill={rgb, 255:red, 0; green, 0; blue, 0 }  ,fill opacity=1 ] (209,90) .. controls (209,88.34) and (210.34,87) .. (212,87) .. controls (213.66,87) and (215,88.34) .. (215,90) .. controls (215,91.66) and (213.66,93) .. (212,93) .. controls (210.34,93) and (209,91.66) .. (209,90) -- cycle ;
\draw  [color={rgb, 255:red, 0; green, 0; blue, 0 }  ,draw opacity=1 ][fill={rgb, 255:red, 0; green, 0; blue, 0 }  ,fill opacity=1 ] (429,190) .. controls (429,188.34) and (430.34,187) .. (432,187) .. controls (433.66,187) and (435,188.34) .. (435,190) .. controls (435,191.66) and (433.66,193) .. (432,193) .. controls (430.34,193) and (429,191.66) .. (429,190) -- cycle ;
\draw  [color={rgb, 255:red, 0; green, 0; blue, 0 }  ,draw opacity=1 ][fill={rgb, 255:red, 0; green, 0; blue, 0 }  ,fill opacity=1 ][line width=0.75]  (429,90) .. controls (429,88.34) and (430.34,87) .. (432,87) .. controls (433.66,87) and (435,88.34) .. (435,90) .. controls (435,91.66) and (433.66,93) .. (432,93) .. controls (430.34,93) and (429,91.66) .. (429,90) -- cycle ;
\draw  [color={rgb, 255:red, 0; green, 0; blue, 0 }  ,draw opacity=1 ][fill={rgb, 255:red, 0; green, 0; blue, 0 }  ,fill opacity=1 ][line width=0.75]  (344,140) .. controls (344,138.34) and (345.34,137) .. (347,137) .. controls (348.66,137) and (350,138.34) .. (350,140) .. controls (350,141.66) and (348.66,143) .. (347,143) .. controls (345.34,143) and (344,141.66) .. (344,140) -- cycle ;
\draw [line width=0.75]    (399.07,131.77) .. controls (410.07,124.1) and (421.07,75.43) .. (435.07,82.1) .. controls (449.07,88.77) and (421.07,141.77) .. (407.4,145.77) .. controls (393.73,149.77) and (333.73,146.43) .. (318.73,146.43) .. controls (304.48,146.43) and (269.78,150.65) .. (252.34,145.06) ;
\draw [shift={(249.73,144.1)}, rotate = 16.35] [fill={rgb, 255:red, 0; green, 0; blue, 0 }  ][line width=0.08]  [draw opacity=0] (10.72,-5.15) -- (0,0) -- (10.72,5.15) -- (7.12,0) -- cycle    ;
\end{tikzpicture}
}%
}%
}%
}

\newcommand{\fifi}[1][1]{%
\mathord{%
\hspace{+2mm}%
\raisebox{-0.4\height}{%
\scalebox{#1}{%
\begin{tikzpicture}[x=0.3pt,y=0.3pt,yscale=-1,xscale=1]

\draw  [color={rgb, 255:red, 0; green, 0; blue, 0 }  ,draw opacity=1 ][fill={rgb, 255:red, 0; green, 0; blue, 0 }  ,fill opacity=1 ][line width=0.75]  (241.73,137.68) .. controls (241.73,136.02) and (243.08,134.68) .. (244.73,134.68) .. controls (246.39,134.68) and (247.73,136.02) .. (247.73,137.68) .. controls (247.73,139.34) and (246.39,140.68) .. (244.73,140.68) .. controls (243.08,140.68) and (241.73,139.34) .. (241.73,137.68) -- cycle ;
\draw  [color={rgb, 255:red, 0; green, 0; blue, 0 }  ,draw opacity=1 ][fill={rgb, 255:red, 0; green, 0; blue, 0 }  ,fill opacity=1 ][line width=0.75]  (296.73,137.68) .. controls (296.73,136.02) and (298.08,134.68) .. (299.73,134.68) .. controls (301.39,134.68) and (302.73,136.02) .. (302.73,137.68) .. controls (302.73,139.34) and (301.39,140.68) .. (299.73,140.68) .. controls (298.08,140.68) and (296.73,139.34) .. (296.73,137.68) -- cycle ;
\draw  [color={rgb, 255:red, 0; green, 0; blue, 0 }  ,draw opacity=1 ][fill={rgb, 255:red, 0; green, 0; blue, 0 }  ,fill opacity=1 ][line width=0.75]  (401.73,137.68) .. controls (401.73,136.02) and (403.08,134.68) .. (404.73,134.68) .. controls (406.39,134.68) and (407.73,136.02) .. (407.73,137.68) .. controls (407.73,139.34) and (406.39,140.68) .. (404.73,140.68) .. controls (403.08,140.68) and (401.73,139.34) .. (401.73,137.68) -- cycle ;
\draw  [color={rgb, 255:red, 0; green, 0; blue, 0 }  ,draw opacity=1 ][fill={rgb, 255:red, 0; green, 0; blue, 0 }  ,fill opacity=1 ] (211.73,187.68) .. controls (211.73,186.02) and (213.08,184.68) .. (214.73,184.68) .. controls (216.39,184.68) and (217.73,186.02) .. (217.73,187.68) .. controls (217.73,189.34) and (216.39,190.68) .. (214.73,190.68) .. controls (213.08,190.68) and (211.73,189.34) .. (211.73,187.68) -- cycle ;
\draw  [color={rgb, 255:red, 0; green, 0; blue, 0 }  ,draw opacity=1 ][fill={rgb, 255:red, 0; green, 0; blue, 0 }  ,fill opacity=1 ] (211.73,87.68) .. controls (211.73,86.02) and (213.08,84.68) .. (214.73,84.68) .. controls (216.39,84.68) and (217.73,86.02) .. (217.73,87.68) .. controls (217.73,89.34) and (216.39,90.68) .. (214.73,90.68) .. controls (213.08,90.68) and (211.73,89.34) .. (211.73,87.68) -- cycle ;
\draw  [color={rgb, 255:red, 0; green, 0; blue, 0 }  ,draw opacity=1 ][fill={rgb, 255:red, 0; green, 0; blue, 0 }  ,fill opacity=1 ] (431.73,187.68) .. controls (431.73,186.02) and (433.08,184.68) .. (434.73,184.68) .. controls (436.39,184.68) and (437.73,186.02) .. (437.73,187.68) .. controls (437.73,189.34) and (436.39,190.68) .. (434.73,190.68) .. controls (433.08,190.68) and (431.73,189.34) .. (431.73,187.68) -- cycle ;
\draw  [color={rgb, 255:red, 0; green, 0; blue, 0 }  ,draw opacity=1 ][fill={rgb, 255:red, 0; green, 0; blue, 0 }  ,fill opacity=1 ][line width=0.75]  (431.73,87.68) .. controls (431.73,86.02) and (433.08,84.68) .. (434.73,84.68) .. controls (436.39,84.68) and (437.73,86.02) .. (437.73,87.68) .. controls (437.73,89.34) and (436.39,90.68) .. (434.73,90.68) .. controls (433.08,90.68) and (431.73,89.34) .. (431.73,87.68) -- cycle ;
\draw  [color={rgb, 255:red, 0; green, 0; blue, 0 }  ,draw opacity=1 ][fill={rgb, 255:red, 0; green, 0; blue, 0 }  ,fill opacity=1 ][line width=0.75]  (346.73,137.68) .. controls (346.73,136.02) and (348.08,134.68) .. (349.73,134.68) .. controls (351.39,134.68) and (352.73,136.02) .. (352.73,137.68) .. controls (352.73,139.34) and (351.39,140.68) .. (349.73,140.68) .. controls (348.08,140.68) and (346.73,139.34) .. (346.73,137.68) -- cycle ;
\draw [line width=0.75]    (250.47,131.11) .. controls (268.13,130.78) and (290.8,130.78) .. (300.47,130.45) .. controls (310.13,130.11) and (338.13,131.11) .. (348.8,131.11) .. controls (359.47,131.11) and (390.8,137.11) .. (401.8,129.45) .. controls (412.8,121.78) and (427.32,74.76) .. (437.8,79.78) .. controls (447.65,84.5) and (428.16,123.31) .. (414.43,131.27) ;
\draw [shift={(411.88,132.4)}, rotate = 324.94] [fill={rgb, 255:red, 0; green, 0; blue, 0 }  ][line width=0.08]  [draw opacity=0] (10.72,-5.15) -- (0,0) -- (10.72,5.15) -- (7.12,0) -- cycle    ;
\end{tikzpicture}
}%
}%
}%
}

\newcommand{\loulou}[1][1]{%
\mathord{%
\hspace{+2mm}%
\raisebox{-0.4\height}{%
\scalebox{#1}{%
\begin{tikzpicture}[x=0.3pt,y=0.3pt,yscale=-1,xscale=1]

\draw  [color={rgb, 255:red, 0; green, 0; blue, 0 }  ,draw opacity=1 ][fill={rgb, 255:red, 0; green, 0; blue, 0 }  ,fill opacity=1 ][line width=0.75]  (200.93,104.18) .. controls (200.93,102.52) and (202.28,101.18) .. (203.93,101.18) .. controls (205.59,101.18) and (206.93,102.52) .. (206.93,104.18) .. controls (206.93,105.84) and (205.59,107.18) .. (203.93,107.18) .. controls (202.28,107.18) and (200.93,105.84) .. (200.93,104.18) -- cycle ;
\draw  [color={rgb, 255:red, 0; green, 0; blue, 0 }  ,draw opacity=1 ][fill={rgb, 255:red, 0; green, 0; blue, 0 }  ,fill opacity=1 ][line width=0.75]  (255.93,104.18) .. controls (255.93,102.52) and (257.28,101.18) .. (258.93,101.18) .. controls (260.59,101.18) and (261.93,102.52) .. (261.93,104.18) .. controls (261.93,105.84) and (260.59,107.18) .. (258.93,107.18) .. controls (257.28,107.18) and (255.93,105.84) .. (255.93,104.18) -- cycle ;
\draw  [color={rgb, 255:red, 0; green, 0; blue, 0 }  ,draw opacity=1 ][fill={rgb, 255:red, 0; green, 0; blue, 0 }  ,fill opacity=1 ][line width=0.75]  (360.93,104.18) .. controls (360.93,102.52) and (362.28,101.18) .. (363.93,101.18) .. controls (365.59,101.18) and (366.93,102.52) .. (366.93,104.18) .. controls (366.93,105.84) and (365.59,107.18) .. (363.93,107.18) .. controls (362.28,107.18) and (360.93,105.84) .. (360.93,104.18) -- cycle ;
\draw  [color={rgb, 255:red, 0; green, 0; blue, 0 }  ,draw opacity=1 ][fill={rgb, 255:red, 0; green, 0; blue, 0 }  ,fill opacity=1 ] (170.93,154.18) .. controls (170.93,152.52) and (172.28,151.18) .. (173.93,151.18) .. controls (175.59,151.18) and (176.93,152.52) .. (176.93,154.18) .. controls (176.93,155.84) and (175.59,157.18) .. (173.93,157.18) .. controls (172.28,157.18) and (170.93,155.84) .. (170.93,154.18) -- cycle ;
\draw  [color={rgb, 255:red, 0; green, 0; blue, 0 }  ,draw opacity=1 ][fill={rgb, 255:red, 0; green, 0; blue, 0 }  ,fill opacity=1 ] (170.93,54.18) .. controls (170.93,52.52) and (172.28,51.18) .. (173.93,51.18) .. controls (175.59,51.18) and (176.93,52.52) .. (176.93,54.18) .. controls (176.93,55.84) and (175.59,57.18) .. (173.93,57.18) .. controls (172.28,57.18) and (170.93,55.84) .. (170.93,54.18) -- cycle ;
\draw  [color={rgb, 255:red, 0; green, 0; blue, 0 }  ,draw opacity=1 ][fill={rgb, 255:red, 0; green, 0; blue, 0 }  ,fill opacity=1 ] (390.93,154.18) .. controls (390.93,152.52) and (392.28,151.18) .. (393.93,151.18) .. controls (395.59,151.18) and (396.93,152.52) .. (396.93,154.18) .. controls (396.93,155.84) and (395.59,157.18) .. (393.93,157.18) .. controls (392.28,157.18) and (390.93,155.84) .. (390.93,154.18) -- cycle ;
\draw  [color={rgb, 255:red, 0; green, 0; blue, 0 }  ,draw opacity=1 ][fill={rgb, 255:red, 0; green, 0; blue, 0 }  ,fill opacity=1 ][line width=0.75]  (390.93,54.18) .. controls (390.93,52.52) and (392.28,51.18) .. (393.93,51.18) .. controls (395.59,51.18) and (396.93,52.52) .. (396.93,54.18) .. controls (396.93,55.84) and (395.59,57.18) .. (393.93,57.18) .. controls (392.28,57.18) and (390.93,55.84) .. (390.93,54.18) -- cycle ;
\draw  [color={rgb, 255:red, 0; green, 0; blue, 0 }  ,draw opacity=1 ][fill={rgb, 255:red, 0; green, 0; blue, 0 }  ,fill opacity=1 ][line width=0.75]  (305.93,104.18) .. controls (305.93,102.52) and (307.28,101.18) .. (308.93,101.18) .. controls (310.59,101.18) and (311.93,102.52) .. (311.93,104.18) .. controls (311.93,105.84) and (310.59,107.18) .. (308.93,107.18) .. controls (307.28,107.18) and (305.93,105.84) .. (305.93,104.18) -- cycle ;
\draw [line width=0.75]    (211,97.85) .. controls (232.5,94.35) and (250.83,97.68) .. (260.5,97.35) .. controls (270.17,97.02) and (297.5,98.38) .. (310,98.61) .. controls (322.5,98.85) and (358.5,95.35) .. (369.5,99.35) .. controls (380.5,103.35) and (407,154.35) .. (399.5,162.35) .. controls (392.26,170.07) and (376.18,139.61) .. (367.42,115) ;
\draw [shift={(366.5,112.35)}, rotate = 68.9] [fill={rgb, 255:red, 0; green, 0; blue, 0 }  ][line width=0.08]  [draw opacity=0] (10.72,-5.15) -- (0,0) -- (10.72,5.15) -- (7.12,0) -- cycle    ;
\end{tikzpicture}
}%
}%
}%
}

\newcommand{\donald}[1][1]{%
\mathord{%
\hspace{+2mm}%
\raisebox{-0.4\height}{%
\scalebox{#1}{%
\begin{tikzpicture}[x=0.3pt,y=0.3pt,yscale=-1,xscale=1]

\draw  [color={rgb, 255:red, 0; green, 0; blue, 0 }  ,draw opacity=1 ][fill={rgb, 255:red, 0; green, 0; blue, 0 }  ,fill opacity=1 ][line width=0.75]  (243.2,138.28) .. controls (243.2,136.62) and (244.54,135.28) .. (246.2,135.28) .. controls (247.86,135.28) and (249.2,136.62) .. (249.2,138.28) .. controls (249.2,139.94) and (247.86,141.28) .. (246.2,141.28) .. controls (244.54,141.28) and (243.2,139.94) .. (243.2,138.28) -- cycle ;
\draw  [color={rgb, 255:red, 0; green, 0; blue, 0 }  ,draw opacity=1 ][fill={rgb, 255:red, 0; green, 0; blue, 0 }  ,fill opacity=1 ][line width=0.75]  (298.2,138.28) .. controls (298.2,136.62) and (299.54,135.28) .. (301.2,135.28) .. controls (302.86,135.28) and (304.2,136.62) .. (304.2,138.28) .. controls (304.2,139.94) and (302.86,141.28) .. (301.2,141.28) .. controls (299.54,141.28) and (298.2,139.94) .. (298.2,138.28) -- cycle ;
\draw  [color={rgb, 255:red, 0; green, 0; blue, 0 }  ,draw opacity=1 ][fill={rgb, 255:red, 0; green, 0; blue, 0 }  ,fill opacity=1 ][line width=0.75]  (403.2,138.28) .. controls (403.2,136.62) and (404.54,135.28) .. (406.2,135.28) .. controls (407.86,135.28) and (409.2,136.62) .. (409.2,138.28) .. controls (409.2,139.94) and (407.86,141.28) .. (406.2,141.28) .. controls (404.54,141.28) and (403.2,139.94) .. (403.2,138.28) -- cycle ;
\draw  [color={rgb, 255:red, 0; green, 0; blue, 0 }  ,draw opacity=1 ][fill={rgb, 255:red, 0; green, 0; blue, 0 }  ,fill opacity=1 ] (213.2,188.28) .. controls (213.2,186.62) and (214.54,185.28) .. (216.2,185.28) .. controls (217.86,185.28) and (219.2,186.62) .. (219.2,188.28) .. controls (219.2,189.94) and (217.86,191.28) .. (216.2,191.28) .. controls (214.54,191.28) and (213.2,189.94) .. (213.2,188.28) -- cycle ;
\draw  [color={rgb, 255:red, 0; green, 0; blue, 0 }  ,draw opacity=1 ][fill={rgb, 255:red, 0; green, 0; blue, 0 }  ,fill opacity=1 ] (213.2,88.28) .. controls (213.2,86.62) and (214.54,85.28) .. (216.2,85.28) .. controls (217.86,85.28) and (219.2,86.62) .. (219.2,88.28) .. controls (219.2,89.94) and (217.86,91.28) .. (216.2,91.28) .. controls (214.54,91.28) and (213.2,89.94) .. (213.2,88.28) -- cycle ;
\draw  [color={rgb, 255:red, 0; green, 0; blue, 0 }  ,draw opacity=1 ][fill={rgb, 255:red, 0; green, 0; blue, 0 }  ,fill opacity=1 ] (433.2,188.28) .. controls (433.2,186.62) and (434.54,185.28) .. (436.2,185.28) .. controls (437.86,185.28) and (439.2,186.62) .. (439.2,188.28) .. controls (439.2,189.94) and (437.86,191.28) .. (436.2,191.28) .. controls (434.54,191.28) and (433.2,189.94) .. (433.2,188.28) -- cycle ;
\draw  [color={rgb, 255:red, 0; green, 0; blue, 0 }  ,draw opacity=1 ][fill={rgb, 255:red, 0; green, 0; blue, 0 }  ,fill opacity=1 ][line width=0.75]  (433.2,88.28) .. controls (433.2,86.62) and (434.54,85.28) .. (436.2,85.28) .. controls (437.86,85.28) and (439.2,86.62) .. (439.2,88.28) .. controls (439.2,89.94) and (437.86,91.28) .. (436.2,91.28) .. controls (434.54,91.28) and (433.2,89.94) .. (433.2,88.28) -- cycle ;
\draw  [color={rgb, 255:red, 0; green, 0; blue, 0 }  ,draw opacity=1 ][fill={rgb, 255:red, 0; green, 0; blue, 0 }  ,fill opacity=1 ][line width=0.75]  (348.2,138.28) .. controls (348.2,136.62) and (349.54,135.28) .. (351.2,135.28) .. controls (352.86,135.28) and (354.2,136.62) .. (354.2,138.28) .. controls (354.2,139.94) and (352.86,141.28) .. (351.2,141.28) .. controls (349.54,141.28) and (348.2,139.94) .. (348.2,138.28) -- cycle ;
\draw [line width=0.75]    (398.64,143) .. controls (381.84,147) and (313.44,144.6) .. (299.04,144.6) .. controls (284.64,144.6) and (239.44,150.6) .. (239.44,138.6) .. controls (239.44,126.6) and (279.44,131.4) .. (298.24,132.2) .. controls (316.1,132.96) and (374.75,131.19) .. (395.32,133.42) ;
\draw [shift={(398.24,133.8)}, rotate = 186.17] [fill={rgb, 255:red, 0; green, 0; blue, 0 }  ][line width=0.08]  [draw opacity=0] (10.72,-5.15) -- (0,0) -- (10.72,5.15) -- (7.12,0) -- cycle    ;
\end{tikzpicture}
}%
}%
}%
}

\newcommand{\picsou}[1][1]{%
\mathord{%
\hspace{+2mm}%
\raisebox{-0.4\height}{%
\scalebox{#1}{%
\begin{tikzpicture}[x=0.3pt,y=0.3pt,yscale=-1,xscale=1]

\draw  [color={rgb, 255:red, 0; green, 0; blue, 0 }  ,draw opacity=1 ][fill={rgb, 255:red, 0; green, 0; blue, 0 }  ,fill opacity=1 ][line width=0.75]  (270,153) .. controls (270,151.34) and (271.34,150) .. (273,150) .. controls (274.66,150) and (276,151.34) .. (276,153) .. controls (276,154.66) and (274.66,156) .. (273,156) .. controls (271.34,156) and (270,154.66) .. (270,153) -- cycle ;
\draw  [color={rgb, 255:red, 0; green, 0; blue, 0 }  ,draw opacity=1 ][fill={rgb, 255:red, 0; green, 0; blue, 0 }  ,fill opacity=1 ][line width=0.75]  (325,153) .. controls (325,151.34) and (326.34,150) .. (328,150) .. controls (329.66,150) and (331,151.34) .. (331,153) .. controls (331,154.66) and (329.66,156) .. (328,156) .. controls (326.34,156) and (325,154.66) .. (325,153) -- cycle ;
\draw  [color={rgb, 255:red, 0; green, 0; blue, 0 }  ,draw opacity=1 ][fill={rgb, 255:red, 0; green, 0; blue, 0 }  ,fill opacity=1 ][line width=0.75]  (430,153) .. controls (430,151.34) and (431.34,150) .. (433,150) .. controls (434.66,150) and (436,151.34) .. (436,153) .. controls (436,154.66) and (434.66,156) .. (433,156) .. controls (431.34,156) and (430,154.66) .. (430,153) -- cycle ;
\draw  [color={rgb, 255:red, 0; green, 0; blue, 0 }  ,draw opacity=1 ][fill={rgb, 255:red, 0; green, 0; blue, 0 }  ,fill opacity=1 ] (240,203) .. controls (240,201.34) and (241.34,200) .. (243,200) .. controls (244.66,200) and (246,201.34) .. (246,203) .. controls (246,204.66) and (244.66,206) .. (243,206) .. controls (241.34,206) and (240,204.66) .. (240,203) -- cycle ;
\draw  [color={rgb, 255:red, 0; green, 0; blue, 0 }  ,draw opacity=1 ][fill={rgb, 255:red, 0; green, 0; blue, 0 }  ,fill opacity=1 ] (240,103) .. controls (240,101.34) and (241.34,100) .. (243,100) .. controls (244.66,100) and (246,101.34) .. (246,103) .. controls (246,104.66) and (244.66,106) .. (243,106) .. controls (241.34,106) and (240,104.66) .. (240,103) -- cycle ;
\draw  [color={rgb, 255:red, 0; green, 0; blue, 0 }  ,draw opacity=1 ][fill={rgb, 255:red, 0; green, 0; blue, 0 }  ,fill opacity=1 ] (460,203) .. controls (460,201.34) and (461.34,200) .. (463,200) .. controls (464.66,200) and (466,201.34) .. (466,203) .. controls (466,204.66) and (464.66,206) .. (463,206) .. controls (461.34,206) and (460,204.66) .. (460,203) -- cycle ;
\draw  [color={rgb, 255:red, 0; green, 0; blue, 0 }  ,draw opacity=1 ][fill={rgb, 255:red, 0; green, 0; blue, 0 }  ,fill opacity=1 ][line width=0.75]  (460,103) .. controls (460,101.34) and (461.34,100) .. (463,100) .. controls (464.66,100) and (466,101.34) .. (466,103) .. controls (466,104.66) and (464.66,106) .. (463,106) .. controls (461.34,106) and (460,104.66) .. (460,103) -- cycle ;
\draw  [color={rgb, 255:red, 0; green, 0; blue, 0 }  ,draw opacity=1 ][fill={rgb, 255:red, 0; green, 0; blue, 0 }  ,fill opacity=1 ][line width=0.75]  (375,153) .. controls (375,151.34) and (376.34,150) .. (378,150) .. controls (379.66,150) and (381,151.34) .. (381,153) .. controls (381,154.66) and (379.66,156) .. (378,156) .. controls (376.34,156) and (375,154.66) .. (375,153) -- cycle ;
\draw [line width=0.75]    (284.41,146.54) .. controls (307.37,144.57) and (335.33,146.83) .. (348.7,146.35) .. controls (362.7,145.85) and (409.2,151.35) .. (424.7,144.85) .. controls (440.2,138.35) and (452.7,90.18) .. (466.7,96.85) .. controls (480.7,103.52) and (441.7,140.85) .. (441.7,151.85) .. controls (441.7,162.85) and (478.7,198.35) .. (468.7,209.35) .. controls (458.7,220.35) and (438.7,167.85) .. (425.7,162.35) .. controls (412.7,156.85) and (364.7,160.35) .. (349.2,160.35) .. controls (333.7,160.35) and (305.7,162.85) .. (279.2,158.85) ;
\draw [shift={(281.2,146.85)}, rotate = 355.09] [fill={rgb, 255:red, 0; green, 0; blue, 0 }  ][line width=0.08]  [draw opacity=0] (10.72,-5.15) -- (0,0) -- (10.72,5.15) -- (7.12,0) -- cycle    ;
\end{tikzpicture}
}%
}%
}%
}

\newcommand{\zaza}[1][1]{%
\mathord{%
\hspace{+2mm}%
\raisebox{-0.4\height}{%
\scalebox{#1}{%
\begin{tikzpicture}[x=0.3pt,y=0.3pt,yscale=-1,xscale=1]

\draw  [color={rgb, 255:red, 0; green, 0; blue, 0 }  ,draw opacity=1 ][fill={rgb, 255:red, 0; green, 0; blue, 0 }  ,fill opacity=1 ][line width=0.75]  (234,125) .. controls (234,123.34) and (235.34,122) .. (237,122) .. controls (238.66,122) and (240,123.34) .. (240,125) .. controls (240,126.66) and (238.66,128) .. (237,128) .. controls (235.34,128) and (234,126.66) .. (234,125) -- cycle ;
\draw  [color={rgb, 255:red, 0; green, 0; blue, 0 }  ,draw opacity=1 ][fill={rgb, 255:red, 0; green, 0; blue, 0 }  ,fill opacity=1 ][line width=0.75]  (289,125) .. controls (289,123.34) and (290.34,122) .. (292,122) .. controls (293.66,122) and (295,123.34) .. (295,125) .. controls (295,126.66) and (293.66,128) .. (292,128) .. controls (290.34,128) and (289,126.66) .. (289,125) -- cycle ;
\draw  [color={rgb, 255:red, 0; green, 0; blue, 0 }  ,draw opacity=1 ][fill={rgb, 255:red, 0; green, 0; blue, 0 }  ,fill opacity=1 ][line width=0.75]  (394,125) .. controls (394,123.34) and (395.34,122) .. (397,122) .. controls (398.66,122) and (400,123.34) .. (400,125) .. controls (400,126.66) and (398.66,128) .. (397,128) .. controls (395.34,128) and (394,126.66) .. (394,125) -- cycle ;
\draw  [color={rgb, 255:red, 0; green, 0; blue, 0 }  ,draw opacity=1 ][fill={rgb, 255:red, 0; green, 0; blue, 0 }  ,fill opacity=1 ] (204,175) .. controls (204,173.34) and (205.34,172) .. (207,172) .. controls (208.66,172) and (210,173.34) .. (210,175) .. controls (210,176.66) and (208.66,178) .. (207,178) .. controls (205.34,178) and (204,176.66) .. (204,175) -- cycle ;
\draw  [color={rgb, 255:red, 0; green, 0; blue, 0 }  ,draw opacity=1 ][fill={rgb, 255:red, 0; green, 0; blue, 0 }  ,fill opacity=1 ] (204,75) .. controls (204,73.34) and (205.34,72) .. (207,72) .. controls (208.66,72) and (210,73.34) .. (210,75) .. controls (210,76.66) and (208.66,78) .. (207,78) .. controls (205.34,78) and (204,76.66) .. (204,75) -- cycle ;
\draw  [color={rgb, 255:red, 0; green, 0; blue, 0 }  ,draw opacity=1 ][fill={rgb, 255:red, 0; green, 0; blue, 0 }  ,fill opacity=1 ] (424,175) .. controls (424,173.34) and (425.34,172) .. (427,172) .. controls (428.66,172) and (430,173.34) .. (430,175) .. controls (430,176.66) and (428.66,178) .. (427,178) .. controls (425.34,178) and (424,176.66) .. (424,175) -- cycle ;
\draw  [color={rgb, 255:red, 0; green, 0; blue, 0 }  ,draw opacity=1 ][fill={rgb, 255:red, 0; green, 0; blue, 0 }  ,fill opacity=1 ][line width=0.75]  (424,75) .. controls (424,73.34) and (425.34,72) .. (427,72) .. controls (428.66,72) and (430,73.34) .. (430,75) .. controls (430,76.66) and (428.66,78) .. (427,78) .. controls (425.34,78) and (424,76.66) .. (424,75) -- cycle ;
\draw  [color={rgb, 255:red, 0; green, 0; blue, 0 }  ,draw opacity=1 ][fill={rgb, 255:red, 0; green, 0; blue, 0 }  ,fill opacity=1 ][line width=0.75]  (339,125) .. controls (339,123.34) and (340.34,122) .. (342,122) .. controls (343.66,122) and (345,123.34) .. (345,125) .. controls (345,126.66) and (343.66,128) .. (342,128) .. controls (340.34,128) and (339,126.66) .. (339,125) -- cycle ;
\draw [line width=0.75]    (248.41,118.54) .. controls (271.37,116.57) and (299.33,118.83) .. (312.7,118.35) .. controls (326.7,117.85) and (373.2,123.35) .. (388.7,116.85) .. controls (404.2,110.35) and (416.7,62.18) .. (430.7,68.85) .. controls (444.7,75.52) and (412.67,126.43) .. (400.33,133.1) .. controls (388,139.77) and (345.33,139.1) .. (345.33,132.1) .. controls (345.33,125.1) and (388.67,122.43) .. (389,128.43) .. controls (389.33,134.43) and (328.7,132.35) .. (313.2,132.35) .. controls (297.7,132.35) and (272.17,133.1) .. (245.67,129.1) ;
\draw [shift={(245.2,118.85)}, rotate = 355.09] [fill={rgb, 255:red, 0; green, 0; blue, 0 }  ][line width=0.08]  [draw opacity=0] (10.72,-5.15) -- (0,0) -- (10.72,5.15) -- (7.12,0) -- cycle    ;
\end{tikzpicture}
}%
}%
}%
}

\usepackage{aliascnt}
\usepackage[nameinlink]{cleveref}

\newtheorem{thm}{Theorem}[section]

\newaliascnt{lem}{thm}
\newtheorem{lem}[lem]{Lemma}
\aliascntresetthe{lem}

\newaliascnt{prop}{thm}
\newtheorem{prop}[prop]{Proposition}
\aliascntresetthe{prop}

\newaliascnt{cor}{thm}
\newtheorem{cor}[cor]{Corollary}
\aliascntresetthe{cor}

\theoremstyle{definition}

\newaliascnt{defn}{thm}
\newtheorem{defn}[defn]{Definition}
\aliascntresetthe{defn}

\newaliascnt{notation}{thm}
\newtheorem{notation}[notation]{Notation}
\aliascntresetthe{notation}

\theoremstyle{remark}

\newaliascnt{rmq}{thm}
\newtheorem{rmq}[rmq]{Remark}
\aliascntresetthe{rmq}

\newaliascnt{ex}{thm}
\newtheorem{ex}[ex]{Example}
\aliascntresetthe{ex}

\crefname{thm}{theorem}{theorems}
\Crefname{thm}{Theorem}{Theorems}

\crefname{lem}{lemma}{lemmas}
\Crefname{lem}{Lemma}{Lemmas}

\crefname{prop}{proposition}{propositions}
\Crefname{prop}{Proposition}{Propositions}

\crefname{cor}{corollary}{corollaries}
\Crefname{cor}{Corollary}{Corollaries}

\crefname{defn}{definition}{definitions}
\Crefname{defn}{Definition}{Definitions}

\crefname{notation}{notation}{notations}
\Crefname{notation}{Notation}{Notations}

\crefname{rmq}{remark}{remarks}
\Crefname{rmq}{Remark}{Remarks}

\crefname{ex}{example}{examples}
\Crefname{ex}{Example}{Examples}

\usepackage{stmaryrd}

\begin{document}

\begin{abstract}
    We study quiver varieties associated with orthogonal quivers and determine the Kleinian surface singularities arising from affine Dynkin types.
    In the classical setting, affine Dynkin quivers play a distinguished role through Kronheimer's construction of Kleinian singularities and its reformulation by Nakajima in terms of quiver varieties.
    Using a graphical calculus in the preprojective algebra of the quiver, we determine which Kleinian surface singularities arise from orthogonal quivers whose underlying graph is an affine Dynkin diagram.
\end{abstract}

\maketitle

\tableofcontents

\section{Introduction}

Quivers are directed graphs whose representation theory was initiated by Gabriel \cite{Gab72}. 
Representations of quivers provide a natural framework for studying representations of products of general linear groups and their invariant theory.

Quiver varieties, introduced by Nakajima in \cite{Nak94}, arise as moduli spaces obtained as symplectic (or hyperkähler) quotients of representation spaces of quivers.
In physics, appropriately chosen quivers and quiver varieties describe Higgs branches of certain supersymmetric gauge theories \cite[Introduction~1(vii)]{Nak2016}.
The corresponding gauge group is the compact Lie group $\prod_{i \in I} \mathrm{U}(\dd_i)$, whose complexification yields the base-change group $G_{\dd} = \prod_{i \in I}\GL_{\dd_i}(\CC)$ in the construction of quiver varieties.
For an introduction to quiver varieties, we refer to \cite{Nakintro}.

In 2002, Derksen and Weyman introduced the theory of symmetric quiver representations (see \cite{DW02}), providing a framework for studying more general representations, including those involving classical groups such as orthogonal and symplectic groups.
More recently, quiver varieties associated with classical groups have been studied in the work of Li and Nakajima \cite{li2018,Nak25classical}. 
A suitable generalization of quiver varieties is expected to describe Higgs branches for orientifolds.

This work studies analogues of Nakajima quiver varieties in this more general symmetric setting.
Various notions of quivers with involution, symmetric quivers, and supermixed quivers have already been studied in the literature, although their associated quiver varieties have not been systematically investigated (see \cite{DW02,Sam12,Shm03,Zub03}).
In this work, we restrict our attention to a class of quivers that we call orthogonal quivers, namely quivers whose base-change groups may contain orthogonal factors.

A central object in the construction of quiver varieties is the moment map $\mugras$, defined on the representation space of the doubled quiver as shown in Figure~\ref{fig:1} (see \Cref{section:2} for details).

\begin{figure}[H]
\centering
\[\begin{tikzcd}[column sep=9em, row sep=4em]
\boxed{
\begin{matrix}
Q \text{ quiver} \\
\dd \in \NN^I \text{dimension vector}
\end{matrix}
}
&
\boxed{
\begin{matrix}
G_{\dd} \action \Rep(Q,\dd) \\
\text{group representation}
\end{matrix}
}
\\
\boxed{
\begin{matrix}
\overline{Q} \text{ double quiver} \\
\dd \in \NN^I \text{ same dimension vector}
\end{matrix}
}
&
\boxed{
\begin{matrix}
G_{\dd} \action \Rep(\overline{Q},\dd) \overset{\mugras}{\to} \ggot_{\dd}^* \\
\text{Hamiltonian action}
\end{matrix}
}
\arrow["{\dd\text{-dim rep}}", from=1-1, to=1-2]
\arrow["{\begin{array}{c} \text{symplectic} \\ \text{geometry} \\ \mathrm{T}^*\Rep(Q,\dd) \end{array}}", from=1-2, to=2-2]
\arrow["{\dd\text{-dim rep}}", from=2-1, to=2-2]
\end{tikzcd}\]
\caption{}
\label{fig:1}
\end{figure}

Quiver varieties are defined as geometric invariant theory (GIT) quotients, in the sense of Mumford \cite{MumGIT}, of the scheme-theoretic fiber $\mugras^{-1}(\zeta)$ of the moment map under the action of the base-change group $G_{\dd}$, where $\zeta \in (\ggot_{\dd}^*)^{G_{\dd}}$ and $\chi \colon G_{\dd} \to \CC^*$ is a character:
\begin{align}\label{defn:quiv varieties}
    \Mgotgras_{\zeta,\chi,\dd}(Q) = \mugras^{-1}(\zeta)\sslash_{\chi} G_{\dd}.
\end{align}

In the setting of orthogonal quiver representations, Figure~\ref{fig:1} is replaced by Figure~\ref{fig:2}. 
We obtain the associated quiver varieties $\Mgot_{\zeta,\chi,\dd}(Q)$, where $\mu$ is the moment map for orthogonal quivers introduced in \Cref{moment of symmetric quivers}, distinguished from the boldface notation $\mugras$, $\zeta \in (\ggot_{\dd}^{b})^{\ast}$ fixed by $G^b_{\dd}$, and $\chi\colon G^b_{\dd} \to \CC^*$ is a character (see \Cref{section:2} for the notation concerning symmetric quivers).

\begin{figure}[H]
\centering
\[
\begin{tikzcd}[column sep=9em, row sep=4em]
\boxed{
\begin{matrix}
Q \text{ orthogonal quiver} \\
\dd \in \NN^I \text{ symmetric } \\
\text{dimension vector}
\end{matrix}
}
&
\boxed{
\begin{matrix}
G_{\dd}^b \action \Srep(Q,\dd) \\
\text{group representation}
\end{matrix}
}
\\
\boxed{
\begin{matrix}
\overline{Q} \text{ orthogonal quiver} \\
\dd \in \NN^I \text{ same symmetric } \\
\text{dimension vector}
\end{matrix}
}
&
\boxed{
\begin{matrix}
G_{\dd}^b \action \Srep(\overline{Q},\dd) \overset{\mu}{\to} (\ggot^{b}_{\dd})^* \\
\text{Hamiltonian action}
\end{matrix}
}
\arrow["{\dd\text{-dim rep}}", from=1-1, to=1-2]
\arrow["{\begin{array}{c} \text{symplectic} \\ \text{geometry} \\ \mathrm{T}^*\Srep(Q,\dd) \end{array}}", from=1-2, to=2-2]
\arrow["{\dd\text{-dim rep}}", from=2-1, to=2-2]
\end{tikzcd}
\]
\caption{}
\label{fig:2}
\end{figure}

In \cite{Kro89}, Kronheimer used Hitchin's hyperkähler quotient construction, together with McKay's 
correspondence relating the irreducible representations of finite subgroups
$\Gamma \subset \mathrm{SU}(2)$ to the simply laced Dynkin diagrams $A_n$, $D_n$, and $E_{6},E_{7},E_{8}$, to construct families of hyperkähler $4$-manifolds $X_\zeta$, depending on a parameter $\zeta$.

The author showed that, when $\zeta=0$, one recovers the surface with Kleinian singularity $\CC^2/\Gamma$ associated with the corresponding Dynkin diagram, while, for generic $\zeta$, the varieties $X_\zeta$ are smooth hyperkähler resolutions of $\CC^2/\Gamma$ \cite[Section~3, Corollary~3.2]{Kro89}.
This is the construction that Nakajima later reformulated in terms of quiver varieties associated with the affine Dynkin diagrams $\widetilde{A}_n$, $\widetilde{D}_n$, and $\widetilde{E}_{6},\widetilde{E}_{7},\widetilde{E}_{8}$. 
In this framework, the spaces $X_\zeta$ are realized first as hyperkähler quotients and then as GIT quotients $\Mgotgras_{\zeta_{\mathbb{I}},\zeta_{\mathbb{R}},\delta}(Q)$, where $\delta$ denotes the minimal imaginary root, and $\zeta_{\mathbb{R}}$ and $\zeta_{\mathbb{I}}$ are the real and imaginary parts of $\zeta$.

\subsubsection*{Main results}
Our results in \Cref{Thm} concern the study of Kleinian surface singularities arising in families of orthogonal quiver varieties associated with affine Dynkin diagrams.

The definition of such quivers requires the existence of a contravariant involution.
Many quivers do not admit a contravariant involution; in particular, the identity involution is not contravariant.
Nevertheless, for every classical quiver, one can construct an orthogonal quiver whose symmetric representations are naturally identified with the representations of the classical quiver.
In this sense, the representation theory of orthogonal quivers extends the classical theory of quiver representations.
Not every affine Dynkin diagram admits an orientation compatible with a contravariant involution, that is, an orthogonal quiver structure.
Such affine Dynkin diagrams define six infinite families of quivers with a contravariant involution, listed below up to orientation:

\begin{figure}[H]
    \centering

\caption{}
\label{fig:3}
\end{figure}

For each family, the labels $v$ and $a$ indicate respectively the presence of a fixed vertex or a fixed arrow, used to distinguish the different families, while $c$ denotes the central symmetry.
Taking into account the possibility of choosing a sign for the fixed arrow (see \Cref{defn:ortho quivers}), we obtain ten infinite families of group actions defined by orthogonal quivers whose underlying diagrams are affine Dynkin diagrams.
Different group actions arise from different choices of signs on the fixed arrows of the quivers. The different group actions realized by these quivers are described in \Cref{symmetries Atilde,symmetries Dtilde}.

\begin{rmq}
    There are also orthogonal quiver structures based on affine Dynkin diagrams that are not included in these ten infinite families, namely, those associated with the Jordan quiver $\mathcal{J}$, or, with a slight abuse of notation, with $\widetilde{A}_0$.
    The two remaining orthogonal structures give rise to two distinct families of group actions.
    $$\begin{tikzpicture}[x=0.7pt,y=0.7pt,yscale=-1,xscale=1]
    
    \draw  [color={rgb, 255:red, 0; green, 0; blue, 0 }  ,draw opacity=1 ][fill={rgb, 255:red, 0; green, 0; blue, 0 }  ,fill opacity=1 ][line width=0.75]  (183.9,162.81) .. controls (183.9,161.15) and (185.25,159.81) .. (186.9,159.81) .. controls (188.56,159.81) and (189.9,161.15) .. (189.9,162.81) .. controls (189.9,164.47) and (188.56,165.81) .. (186.9,165.81) .. controls (185.25,165.81) and (183.9,164.47) .. (183.9,162.81) -- cycle ;
    \draw [line width=0.75]    (180.71,157.46) .. controls (161.29,136.6) and (175.86,120.6) .. (188.14,121.17) .. controls (200.06,121.73) and (212.78,135.18) .. (194.5,156.59) ;
    \draw [shift={(192.71,158.6)}, rotate = 305.97] [fill={rgb, 255:red, 0; green, 0; blue, 0 }  ][line width=0.08]  [draw opacity=0] (10.72,-5.15) -- (0,0) -- (10.72,5.15) -- (7.12,0) -- cycle    ;
    
    \draw (76.43,126.57) node [anchor=north west][inner sep=0.75pt]    {$\mathcal{J} \ =\ \widetilde{A}_{0} =$};
    \draw (207,132.4) node [anchor=north west][inner sep=0.75pt]    {$a$};
    \draw (332,121.4) node [anchor=north west][inner sep=0.75pt]    {$\begin{vmatrix}
    s( a) \ =\ +1& \mid & s( a) =-1\\
    \OO_n(\CC) \curvearrowright\ \Sym( \CC^n)& \mid & \OO_n(\CC) \curvearrowright\ \Lambda^2( \CC^n)
    \end{vmatrix}$};
    \end{tikzpicture}$$
    We do not treat this case in our article, since it recovers known results for symmetric pairs of $(\OO_n,\GL_n)$ in lower dimensions.
\end{rmq}

A main contribution of this article is a method for explicitly computing orthogonal quiver varieties and their singularities.
The method is illustrated in the following table, which describes the surfaces obtained as quiver varieties $\Mgot_{0,0,\delta}(Q)$ and $\Mgot_{0,0,2\delta}(Q)$, or, more simply, $\Mgot_{\delta}$ and $\Mgot_{2\delta}$, associated with orthogonal quivers of affine Dynkin type.

\newpage
\begin{thm}\label{Thm}\emph{[Proof in \Cref{proof of Thm}]}
For nine infinite families of orthogonal quivers among the ten families presented in \Cref{fig:3}, we obtain the following table of surface singularities arising from the quiver varieties $\Mgot_{\dd}$.
The sign column corresponds to the choice of sign for the fixed arrows.
\begin{table}[H]
\centering
    \renewcommand{\arraystretch}{1.2}
    \begin{tabular}{|c|c|c||c|c|}
    \hline
    \multicolumn{5}{|c|}{Quiver varieties of orthogonal quivers whose support is an affine Dynkin diagram} \\
    \hline
    Orthogonal quivers & sign of fixed arrow &  dimension $\dd$ & $\Mgot_{\dd}$& Kleinian singularity\\
    \hline
$(\widetilde{A}_{2n-2} ,v\text{-}a)$ &$+1$& $\delta$ & $\VV(xy-z^{2n-1})$& $A_{2n-2}$ \\
    $n \geq 2$&$-1$& $\delta$ & $\{\pt\}$& \\
    &$-1$& $2\delta$ & $\VV(xy-z^{4n-2})$& $A_{4n-3}$\\
    \hline
      & $(+1,+1)$ & $\delta$ & $\VV(xy-z^{2n})$ & $A_{2n-1}$\\
$(\widetilde{A}_{2n-1} ,a\text{-}a)$ & $(+1,-1)$ & $\delta$ & $\{\pt\}$ &\\
    & $(+1,-1)$ & $2\delta$ & $\VV(xy-z^{4n})$ & $A_{4n-1}$\\
    $n \geq 1$& $(-1,-1)$ & $\delta$ & $\{\pt\}$ & \\
    & $(-1,-1)$ & $2\delta$ & $\VV(xy-z^{2n})$ & $A_{2n-1}$\\
    \hline
$(\widetilde{A}_{2n-1} ,v\text{-}v), n \geq 1$ && $\delta$ & $\VV(xy - z^{2n})$ & $A_{2n-1}$ \\
    \hline
    $(\widetilde{A}_{2n-1} ,c), n \geq 1$ & & $\delta$& $\{\pt\}$& \\
    & & $2\delta$ & $\VV(x^{n+1} - y^2x - z^2)$ & $D_{n+2}$ \\
    \hline 
    $(\widetilde{D}_{2n-2} ,v), n \geq 3$& & $\delta$ & $\VV(z^2 - x^{n-1}y + xy^2)$ & $D_{n+1}$ \\
    \hline
    $(\widetilde{D}_{2n-1} ,a)$& $+1$ & $\delta$ & $\VV(z^2-x^{n-1}y + xy^2)$ & $D_{n+1}$ \\
    $n \geq 3$& $-1$ & $\delta$ & $\{\pt\}$ &  \\
    \hline 
    \end{tabular}
\end{table}
\end{thm}

\begin{rmq}
    We expect that the family $(\widetilde{D}_{2n-1},a)$ with sign $-1$ in dimension $2\delta$ also yields a quiver variety corresponding to a Kleinian singularity of type $D$.
\end{rmq}

The proof uses a graphical calculus based on five lemmas to obtain normal forms of cycles in the preprojective algebra $\Pi(Q)$ (see \Cref{graphical calculus}) up to a suitable equivalence relation.
Once these normal forms are established, the remaining relations between the corresponding invariant functions determine, in each case, the equation defining the associated Kleinian singularity.

\subsubsection*{Contents}
In \Cref{section:2}, we introduce the theory of orthogonal quivers, their symmetric representations, and the associated group representations.
We then review the symplectic geometry of quivers and orthogonal quivers in \Cref{carquois symplectique}.
In \Cref{section:3}, we develop the methods used for computations in preprojective algebras and apply them to prove \Cref{Thm}.
A list of notation is provided at the end of the article.

\subsubsection*{Perspectives}
The results presented here suggest several directions for further investigation. In particular, it would be interesting to consider other families of ortho-symplectic quivers.

The results of \Cref{section:2} also extend to orthosymplectic quivers in the sense of \Cref{rmq: orthosymplectic quivers}.
The graphical methods developed in this work provide an effective approach to determining the $ADE$ type of singular surfaces arising in quiver varieties $\Mgot_{\dd}$ associated with orthosymplectic quivers.
However, their complexity grows rapidly with the dimension, number of generators and relations involved. 
It would therefore be interesting to develop more conceptual methods for studying the singularities and their invariants in the other case.

\subsubsection*{Acknowledgements}
The author acknowledges the support of the CDP C2EMPI, together with the French State under the France-2030 programme, the University of Lille, the Initiative of Excellence of the University of Lille, the European Metropolis of Lille for their funding and support of the R-CDP-24-004-C2EMPI project.
I would like to thank my two advisors, R.~Terpereau and O.~Serman, for their guidance and patience.

\newpage
\section{Orthogonal quivers and symmetric representations}\label{section:2}

A \emph{quiver} consists of two maps between finite sets $h,t \colon Q \to I$, called the head and the tail, respectively; the elements of $Q$ are the arrows and the elements of $I$ are the vertices of the quiver.

\subsection{Symmetric representations}

\begin{defn}\label{defn:ortho quivers}
    An \emph{orthogonal quiver} is a quiver $Q$ endowed with two additional pieces of data.
    \begin{itemize}
        \item a contravariant involution  $\tau \colon Q \sqcup I \to Q \sqcup I$ which sends arrows to arrows and vertices to vertices:
        the contravariance condition is expressed by the relations $$\forall a \in Q, \quad  \tau t a = h\tau a \text{ and } \tau h a = t\tau a.$$
        \item a sign function $s \colon Q \to \{\pm 1\}$, invariant under $\tau$.
    \end{itemize}
\end{defn}

\begin{rmq}\label{defn:opp/double quivers}
    When passing from an orthogonal quiver $Q$ to its opposite quiver $Q^{\opp}$ or its double $\overline{Q}$, the involutions and signs on arrows are naturally extended such that $s(a^{*}) = s(a)$ and $\tau(a^{*}) = \tau(a)^{*}$, where $a^{*} \in Q^{\opp}$ is the arrow opposite to $a \in Q$.
\end{rmq}

Throughout this section, we fix an orthogonal quiver.
A representation of a quiver, denoted $(V,v)$, consists of linear maps $v_a:V_{ta} \to V_{ha}$, for all $a \in Q$, on a collection of vector spaces $V = \bigoplus_{i \in I} V_i$.
We call $\dd = (\dim V_i)_{i \in I} \in \NN^I$ the dimension vector of $(V,v)$.
The family of maps $(v_a)_{a \in Q}$ can be viewed as a point in the vector space $$\Rep(Q,V)  = \bigoplus_{a \in Q} \Hom(V_{ta},V_{ha}).$$

For symmetric representations of orthogonal quivers, the space $\bigoplus_{i \in I} V_i$ is endowed with a nondegenerate symmetric bilinear form. 

\begin{defn}\label{defn:srep}
Symmetric representations are representations of quivers satisfying a self-duality condition.
    \begin{itemize}
        \item A \emph{signed space} $(V,b)$ is an $I$-graded vector space $V = \bigoplus_{i \in I} V_i$ equipped with a symmetric nondegenerate bilinear form $b \colon V \times V \to \CC$ such that $b(V_i,V_j)=0$ unless $j=\tau i$, and the induced pairing $b:V_i \times V_{\tau i}\to \CC$ is perfect. This bilinear form defines isomorphisms denoted by $J_i \colon V_i \overset{\sim}{\longrightarrow} V_{\tau i}^*$.
        \item A \emph{symmetric representation}, denoted by $(V,v,b)$, is a signed space $(V,b)$ endowed with a representation $(V,v)$ of $Q$ such that for every $a \in Q$ one has:
        \begin{align*}
            b(v_ax,y) & = s(a)b(x,v_{\tau a}y) \qquad \forall x,y \in V, a \in Q.
        \end{align*}
        \item For a signed space $(V,b)$, there is a linear involution $\repinvolution$ on $\Rep(Q,V)$ given by the formula:
        \begin{align*}
            \repinvolution(v)_a = s(a)\, J^{-1}_{h a} v_{\tau a}^* J_{t a}.
        \end{align*}
        Its fixed points are the symmetric representations.
        
        \item We denote $\Srep(Q,V) = \Rep(Q,V)^{\repinvolution}$ as the fixed subspace of symmetric representations of $Q$ with signed space $(V,b)$.
    \end{itemize}
\end{defn}

\begin{rmq}\label{decomp srep arep}
    The antisymmetric representations of $Q$ (i.e. the fixed points of $\Rep(Q,V)^{-\repinvolution}$) are the symmetric representations of the orthogonal quiver $(Q,\tau,-s)$, and we have the direct sum:
    \begin{align*}
        \Rep(Q, V) = \Srep(Q, V) \oplus \Rep(Q,V)^{-\repinvolution}.
    \end{align*}
\end{rmq}

A dimension vector $\dd \in \NN^I$ is called \emph{symmetric} if $\dd_{\tau i}=\dd_i$ for all $i\in I$.

\begin{defn}\label{standard model}
Let $\dd \in \NN^I$ be a symmetric dimension vector.
The \emph{standard signed space} $(\CC^{\dd},b)$ is defined by setting $(\CC^{\dd})*i=\CC^{\dd_i}$ for all $i\in I$, with the following bilinear forms:
\begin{itemize}
    \item for each fixed vertex $i\in I^\tau$, we equip $\CC^{\dd_i}$ with the symmetric bilinear form $b(x,y)=\sum_k x_k y_k$;
    \item for each pair of vertices $(i,\tau i)$ with $i\notin I^\tau$, we equip $\CC^{\dd_i}\times\CC^{\dd*{\tau i}}$ with the perfect pairing $b(x,y)=\sum_k x_k y_k$.
\end{itemize}
\end{defn}

Up to a choice of bases for the spaces $V_i$, every signed space $(V,b)$ is isomorphic to this standard model.
Hence, a signed space is uniquely determined, up to isomorphism, by its dimension vector $\dd\in\NN^I$.
For the standard signed space $\CC^{\dd}$, we write $\Rep(Q,\dd)$ and $\Srep(Q,\dd)$ instead of $\Rep(Q,\CC^{\dd})$ and $\Srep(Q,\CC^{\dd})$.

For a signed space $(V,b)$, the condition $\repinvolution(v)=v$ determines the space of symmetric representations $\Srep(Q,V)$. We have:
$$\Srep(Q,V) \simeq \bigoplus_{\underset{a \ne \tau a}{\{a,\tau a\} \subset Q}} \Hom(V_{ta}, V_{ha}) 
\oplus \bigoplus_{\underset{s(a) = +1}{a \in Q^\tau}} \Sym(V_{ta}^*) 
\oplus \bigoplus_{\underset{s(a) = -1}{a \in Q^\tau}} {\Lambda}^2(V_{ta}^*).$$

A morphism $\varphi\colon(V_1,v_1,b_1) \to (V_2,v_2,b_2)$ between symmetric representations is a morphism $\varphi\colon(V_1,v_1) \to (V_2,v_2)$ for which the pull-back of the signed form coincides $\varphi_*b_2 = b_1$; any such morphism must be injective.
One can define symmetric subrepresentations, direct sums, and tensor products of symmetric representations.
A version of the Krull--Schmidt decomposition theorem holds for symmetric representations; for further properties, we refer to \cite{DW02}.

\subsection{Group representations}

In this section, we introduce the group $G^b_V$ and its action on the space of symmetric representations $\Srep(Q,V)$.
The elements of this group describe isomorphisms between symmetric quiver representations, and its orbits correspond to isomorphism classes of representations.
Recall that the base-change group of a quiver $G_V = \prod_{i \in I} \GL(V_i)$ acts on $\Rep(Q,V)$ by: 
$$\forall a \in Q,v \in \Rep(Q,V) \text{ and }g \in G_V, \quad (g \cdot v)_a = g_{ha} v_a g_{ta}^{-1}.$$

\begin{defn}\label{defn:group}
    \begin{itemize}
        \item\label{defn:group:base change group} The \emph{base-change group} of $(V,b)$ is the closed subgroup $G_V^b \subset G_V$ consisting of the transformations that preserve the signed form $b$.
        \item It can also be characterized as the fixed-point subgroup of the group involution $\varphi^b$:
            \begin{align*}
                \varphi^b\colon G_V & \longrightarrow G_V \\
                (g_i)_{i \in I} & \longmapsto (J_i^{-1} g_{\tau i}^{*-1} J_i)_{i \in I}. \notag
            \end{align*}
        \item The Lie algebra $\ggot_V=\bigoplus_{i\in I}\End_{\CC}(V_i)$ of $G_V$ admits an involution $\lieinvolution$, obtained by differentiating $\varphi^b$, given by
        $\lieinvolution(\theta)_i=-J_i^{-1}\theta_{\tau i}^*J_i \in \End_{\CC}(V_i)$.
        Its fixed-point subalgebra $\ggot^b_V$ is the Lie algebra of $G_V^b$.
        \item The eigenspaces of $\lieinvolution$ decompose the Lie algebra as
        $$\ggot_V = \ggot_V^b \oplus \ggot_V^{-\lieinvolution}.$$
    \end{itemize}
\end{defn}

The group $G^b_V$ is isomorphic to a product of classical reductive linear and orthogonal groups:
\begin{align*}
    G_V^b & = \{g \in G_V \mid g_*b = b\} = \{(g_i)_{i \in I} \in G_V\mid \forall i \in I, \quad g_{\tau i}^*J_ig_i = J_i\} \\
    & \simeq \prod_{i \in I^\tau} \OO(V_i) \times \prod_{\underset{i \ne \tau i}{\{i,\tau i\} \subset I}} \GL(V_i).
\end{align*}

\begin{lem}\label{lem:sym action}
   The representation $G^b_V \curvearrowright \Rep(Q,V)$ splits as the direct sum
$$\Rep(Q,V) = \Srep(Q,V) \oplus \Rep(Q,V)^{-\repinvolution}.$$
Moreover:
    $$\forall g \in G_V, \quad \varphi^b(g) = \repinvolution \circ g \circ \repinvolution \in G_V.$$ 
\end{lem}

\begin{proof}
    Let $v \in \Rep(Q,V)$, $g \in G_V$, and let $\repinvolution$ be the involution on $\Rep(Q,V)$ induced by the signed space (see \Cref{defn:srep}).
    First, we have:
    \begin{align*}
        (g\cdot\repinvolution)(v)_a &= s(a)g_{ha}J_{ha}^{-1}v_{\tau a}^*J_{ta}g_{ta}^{-1},
    \end{align*}
    then,
    \begin{align*}
        \repinvolution(g\cdot\repinvolution(v))_a & = s(a)s(\tau a)J_{ha}^{-1}\big( g_{h\tau a} J_{h\tau a}^{-1}v_{\tau\tau a}^* J_{t\tau a}g_{t\tau a}^{-1} \big)^*J_{ta} \\
        & = J_{ha}^{-1}g_{t\tau a}^{-1*} J_{t\tau a}^*v_{a} J_{h\tau a}^{-1*} g_{h\tau a}^*J_{ta} \\
        & = J_{ha}^{-1}g_{\tau ha}^{-1*} J_{\tau ha}^*v_{a} J_{ta}^{-1} g_{\tau ta}^*J_{ta} \\
        & = J_{ha}^{-1}g_{\tau ha}^{-1*} J_{ha}v_{a} J_{ta}^{-1} g_{\tau ta}^*J_{ta} \\
        & = \varphi^b(g)_{ha}v_{a} \varphi^b(g)_{ta}^{-1}\\
        & = (\varphi^b(g) \cdot v)_{a}.
    \end{align*}
    It remains to show that $\Srep(Q,V)$ is a $G^b_V$-stable subspace of $\Rep(Q,V)$.
    Let $g \in G^b_V$ and $v \in \Srep(Q,V)$. By definition, $\varphi^b(g)=g$ and $\repinvolution(v)=v$, so we conclude that:
    \begin{align*}
        \repinvolution(g\cdot v) & = \repinvolution \circ g \circ \repinvolution \circ \repinvolution (v) = \varphi^b(g) \cdot \repinvolution(v) = g \cdot v
        \text{ and } g \cdot v \in \Srep(Q,V). \qedhere
    \end{align*}
\end{proof}

\begin{ex}
We may consider the following example of an orthogonal quiver with $8$ vertices and $8$ arrows.
The vertices $1$ and $3$, as well as the arrows $c$ and $e$, are fixed by the involution $\tau$. Every other vertex and arrow is paired with a distinct one under $\tau$.
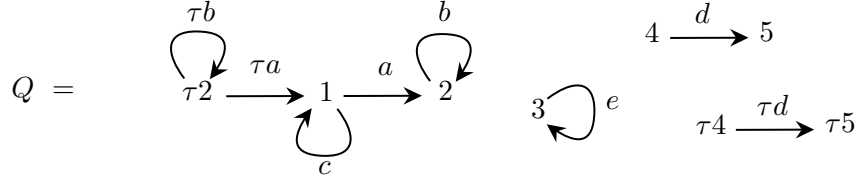
$$
\begin{tikzpicture}[x=0.65pt,y=0.65pt,yscale=-1,xscale=1]

\draw [line width=0.75]    (201.82,149.76) -- (244.15,149.63) ;
\draw [shift={(247.15,149.62)}, rotate = 179.81] [fill={rgb, 255:red, 0; green, 0; blue, 0 }  ][line width=0.08]  [draw opacity=0] (10.72,-5.15) -- (0,0) -- (10.72,5.15) -- (7.12,0) -- cycle    ;
\draw [line width=0.75]    (177.53,139.47) .. controls (167.13,125.87) and (165.13,111.87) .. (185.13,112.27) .. controls (204.23,112.65) and (204.72,122.16) .. (193.93,137.64) ;
\draw [shift={(192.33,139.87)}, rotate = 306.43] [fill={rgb, 255:red, 0; green, 0; blue, 0 }  ][line width=0.08]  [draw opacity=0] (10.72,-5.15) -- (0,0) -- (10.72,5.15) -- (7.12,0) -- cycle    ;
\draw [line width=0.75]    (387.72,151.18) .. controls (401.32,140.78) and (415.32,138.78) .. (414.92,158.78) .. controls (414.53,177.88) and (405.03,178.37) .. (389.55,167.58) ;
\draw [shift={(387.32,165.98)}, rotate = 36.43] [fill={rgb, 255:red, 0; green, 0; blue, 0 }  ][line width=0.08]  [draw opacity=0] (10.72,-5.15) -- (0,0) -- (10.72,5.15) -- (7.12,0) -- cycle    ;
\draw [line width=0.75]    (319.83,140.97) .. controls (309.43,127.37) and (307.43,113.37) .. (327.43,113.77) .. controls (346.53,114.15) and (347.02,123.66) .. (336.23,139.14) ;
\draw [shift={(334.63,141.37)}, rotate = 306.43] [fill={rgb, 255:red, 0; green, 0; blue, 0 }  ][line width=0.08]  [draw opacity=0] (10.72,-5.15) -- (0,0) -- (10.72,5.15) -- (7.12,0) -- cycle    ;
\draw [line width=0.75]    (266.11,157.06) .. controls (276.51,170.66) and (278.51,184.66) .. (258.51,184.26) .. controls (239.41,183.88) and (238.92,174.38) .. (249.71,158.89) ;
\draw [shift={(251.31,156.66)}, rotate = 126.43] [fill={rgb, 255:red, 0; green, 0; blue, 0 }  ][line width=0.08]  [draw opacity=0] (10.72,-5.15) -- (0,0) -- (10.72,5.15) -- (7.12,0) -- cycle    ;
\draw [line width=0.75]    (269.76,149.62) -- (312.09,149.48) ;
\draw [shift={(315.09,149.47)}, rotate = 179.81] [fill={rgb, 255:red, 0; green, 0; blue, 0 }  ][line width=0.08]  [draw opacity=0] (10.72,-5.15) -- (0,0) -- (10.72,5.15) -- (7.12,0) -- cycle    ;
\draw [line width=0.75]    (458.8,116.02) -- (501.13,115.88) ;
\draw [shift={(504.13,115.87)}, rotate = 179.81] [fill={rgb, 255:red, 0; green, 0; blue, 0 }  ][line width=0.08]  [draw opacity=0] (10.72,-5.15) -- (0,0) -- (10.72,5.15) -- (7.12,0) -- cycle    ;
\draw [line width=0.75]    (496.8,169.22) -- (539.13,169.08) ;
\draw [shift={(542.13,169.07)}, rotate = 179.81] [fill={rgb, 255:red, 0; green, 0; blue, 0 }  ][line width=0.08]  [draw opacity=0] (10.72,-5.15) -- (0,0) -- (10.72,5.15) -- (7.12,0) -- cycle    ;

\draw (174.89,139.78) node [anchor=north west][inner sep=0.75pt]  [font=\normalsize]  {$\tau 2$};
\draw (254.09,140.18) node [anchor=north west][inner sep=0.75pt]  [font=\normalsize]  {$1$};
\draw (323.29,139.78) node [anchor=north west][inner sep=0.75pt]  [font=\normalsize]  {$2$};
\draw (376.89,150.18) node [anchor=north west][inner sep=0.75pt]  [font=\normalsize]  {$3$};
\draw (442.59,105.98) node [anchor=north west][inner sep=0.75pt]  [font=\normalsize]  {$4$};
\draw (508.99,105.18) node [anchor=north west][inner sep=0.75pt]  [font=\normalsize]  {$5$};
\draw (472.09,159.48) node [anchor=north west][inner sep=0.75pt]  [font=\normalsize]  {$\tau 4$};
\draw (546.99,158.18) node [anchor=north west][inner sep=0.75pt]  [font=\normalsize]  {$\tau 5$};
\draw (177.59,92.98) node [anchor=north west][inner sep=0.75pt]  [font=\normalsize]  {$\tau b$};
\draw (323.09,92.98) node [anchor=north west][inner sep=0.75pt]  [font=\normalsize]  {$b$};
\draw (288.09,128.48) node [anchor=north west][inner sep=0.75pt]  [font=\normalsize]  {$a$};
\draw (214.09,127.48) node [anchor=north west][inner sep=0.75pt]  [font=\normalsize]  {$\tau a$};
\draw (253.59,185.98) node [anchor=north west][inner sep=0.75pt]  [font=\normalsize]  {$c$};
\draw (420.09,148.48) node [anchor=north west][inner sep=0.75pt]  [font=\normalsize]  {$e$};
\draw (471.59,93.98) node [anchor=north west][inner sep=0.75pt]  [font=\normalsize]  {$d$};
\draw (507.59,147.48) node [anchor=north west][inner sep=0.75pt]  [font=\normalsize]  {$\tau d$};
\draw (76,137.4) node [anchor=north west][inner sep=0.75pt]    {$Q\ =\ $};
\end{tikzpicture}
$$
Signs on unfixed arrows do not affect the family of group representations described by the orthogonal quivers.
Unless otherwise specified, all unfixed arrows are assumed to have sign $+1$.
Otherwise, we choose signs for the arrows $c$ and $e$ to specify the representation space.
If we choose $s(c)=+1$ and $s(e)=-1$, then this orthogonal quiver describes the representation of
$$G^b_V= \OO(V_1)\times \GL(V_2)\times \OO(V_3)\times \GL(V_4)\times \GL(V_5)$$
on the space
$$\Srep(Q,V)= \Hom(V_1,V_2)\oplus \End_{\CC}(V_2)\oplus \Sym(V_1^*)\oplus \Hom(V_4,V_5)\oplus {\Lambda}^{2}(V_3^*),$$
for prescribed vector spaces $V_i$, $1\leq i\leq 5$.
\end{ex}

\subsection{Symplectic geometry for orthogonal quivers}\label{carquois symplectique}
For the remainder of this section, we fix an orthogonal quiver $Q$ and a signed space $(V,b)$.
We recall some classical facts about moment maps and symplectic geometry for quivers. For the classical theory of Hamiltonian reduction for quivers, we refer to \cite[Chapter~9]{QrQv2016}.
The main goal of this section is to prove that restricting the quiver moment map $\mugras$ to the symmetric representations space yields a moment map $\mu$ for the orthogonal quiver theory (see \Cref{moment of symmetric quivers}).

For every arrow $a \in Q$, there is the following perfect pairing: 
\begin{align*}
    \Hom(V_{ta}, V_{ha}) \times & \Hom(V_{ha}, V_{ta})\to \CC, \quad
    (v_a, v_{a^*})\mapsto \tr(v_a v_{a^*}).
\end{align*} 
This is used to identify $\Hom(V_{ta}, V_{ha})^*$ with $\Hom(V_{ha}, V_{ta})$. 
After doing this for all arrows, we obtain a trivialization of the cotangent bundle of $\Rep(Q,V)$, identified with a quiver representation space:
\begin{align*}
    \mathrm{T}^*\Rep(Q,V) = \Rep(Q, V) \times \Rep(Q, V)^* \simeq  \Rep(\overline{Q}, V)
\end{align*}
where $\overline{Q}$ denotes the doubled quiver, obtained by adjoining, for each arrow $a\colon i\to j$ of $Q$, an opposite arrow $a^*\colon j\to i$. Equivalently, the head and tail maps of $\overline{Q}$ are given by
$$(h,h^{\opp}),(t,t^{\opp})\colon Q\sqcup Q^{\opp}\longrightarrow I.$$
The Lie algebra $\ggot_V = \bigoplus_{i \in I} \End_{\CC}(V_i)$ is identified with its dual via a similar pairing:
\begin{align}\label{pairing kappa}
    \kappa\colon \ggot_V \times\ggot_V \to \CC, \quad (\theta,\theta') \mapsto \sum_{i \in I} \tr(\theta_i\theta'_i).
\end{align}
The Liouville symplectic form $\omega$ on $\mathrm{T}^*\Rep(Q,V)$ corresponds to the form:
\begin{align*}
    \forall x, y \in \Rep(\overline{Q},V), \quad \omega(x, y) = \tr\big(\sum_{a \in Q} x_a y_{a^*} - y_a x_{a^*}\big).
\end{align*}
The \emph{infinitesimal action} is given by the vector field $\theta_X \in \Gamma(X,TX)$ associated with $\theta \in \mathfrak{g}_V$. At a point $v \in X = \Rep(\overline{Q},V)$, this yields the tangent vector
$$\theta_X(v) = \frac{d}{dt}\Big|_{t=0} (\exp(t\theta)\cdot v) = \sum_{a \in \overline{Q}}[\theta,v]_a, \ \text{ where } \forall a \in \overline{Q}, \ [\theta, v]_a = \theta_{ha} v_a - v_a \theta_{ta} \in \Hom(V_{ta}, V_{ha}).$$
The vector $[\theta,v] \in T_v\Rep(\overline{Q},V)$ vanishes on the representation $v \in \Rep(\overline{Q},V)$ whenever $\theta \in \End_Q(v)$, i.e.\ when $\theta$ is an endomorphism of the quiver representation $(V,v)$.

\begin{defn}\label{defn:moment map quivers}
    The \emph{quiver moment map} $\mugras\colon \Rep(\overline{Q}, V) \to \ggot_V$ is given by
    \begin{align*}
        \forall i \in I, \forall v \in \Rep(\overline{Q},V), \quad \mugras(v)_i = \sum_{a \in Q,\, ha = i} v_a v_{a^*} - \sum_{a \in Q,\, ta = i} v_{a^*} v_a \in \End_{\CC}(V_i).
    \end{align*}
\end{defn}

\begin{lem}\label{lem: pairing kappa}
    The induced involution $\repinvolution\colon \Rep(\overline{Q}, V) \to \Rep(\overline{Q}, V)$ (see \Cref{defn:srep}) is a symplectomorphism, and the pairing $\kappa:\ggot_V \times \ggot_V \to \CC$ is invariant under the involution $\lieinvolution$.
    Moreover, $\Srep(\overline{Q}, V)$ and $\Rep(\overline{Q}, V)^{-\repinvolution}$ are symplectic vector subspaces of $(\Rep(\overline{Q}, V), \omega)$, and $\kappa$ restricts to a perfect pairing $\kappa^b\colon \ggot^b_V \times \ggot^b_V \to \CC$.
\end{lem}

\begin{proof}
    Since $\repinvolution\colon \Rep(\overline{Q}, V) \to \Rep(\overline{Q}, V)$ is linear, it is a symplectomorphism if 
    $$\forall x, y \in \Rep(\overline{Q}, V), \quad \omega(\repinvolution(x), \repinvolution(y)) = \omega(x, y).$$
    A direct calculation gives
    \begin{align*}
        \omega(\repinvolution(x),\repinvolution(y)) & = \tr\Big(\sum_{a \in Q} \repinvolution(x)_a \repinvolution(y)_{a^*} - \repinvolution(y)_a \repinvolution(x)_{a^*}\Big) \\
        & = \tr\Big(\sum_{a \in Q} J_{ha}^{-1}x_{\tau a}^* J_{ta} J_{ta}^{-1}y_{\tau a^*}^*J_{ha} -J_{ta}^{-1}y_{\tau a}^* J_{ha} J_{ha}^{-1}x_{\tau a^*}^*J_{ta}\Big)\\
        & = \tr\Big(\sum_{a \in Q} x_{\tau a}^* y_{\tau a^*}^*-y_{\tau a}^* x_{\tau a^*}^*\Big) \\
        & =  - \tr\Big(\sum_{a \in Q} x_{\tau a^*} y_{\tau a}-y_{\tau a^*} x_{\tau a}\Big) \\
        & =  - \tr\Big(\sum_{c = \tau a \in Q} x_{c^*} y_{c}-y_{c^*} x_{c}\Big)\\
        & =  \tr\Big(\sum_{c = \tau a \in Q}  x_{c}y_{c^*} - y_{c}x_{c^*}\Big) = \omega(x,y).
    \end{align*}
    If $x \in \Srep(\overline{Q}, V)$ and $y \in \Rep(\overline{Q}, V)^{-\repinvolution}$, we have $\omega(x, y) = 0$, then the restriction $\omega_{|\Srep(\overline{Q}, V)}$ is non-degenerate, and therefore symplectic.
    For all $\theta,\theta' \in \ggot_V$:
    \begin{align*}
        \sum_{i \in I}\tr(\lieinvolution(\theta)_i\lieinvolution(\theta')_i) = \sum_{i \in I} \tr(J_i^{-1}(\theta'_{\tau i}\theta_{\tau i})^*J_i) = \sum_{i \in I} \tr((\theta'_{\tau i}\theta_{\tau i})^*) = \sum_{i \in I} \tr(\theta'_{\tau i}\theta_{\tau i}) = \sum_{i \in I} \tr(\theta_{i}\theta_{i}').
    \end{align*}
    Thus, if we take $\theta \in \ggot_V^b$ and $\theta' \in (\ggot_V)^{-\lieinvolution}$, we obtain $\tr(\theta\theta') = 0,$ and therefore the restriction to $\ggot_V^b$ is a perfect pairing.
\end{proof}

\begin{prop}\label{mu et crochet commute à t et tgot}
    The moment map and the bracket of the infinitesimal action intertwine the two involutions $\repinvolution$ and $\lieinvolution$ (see \Cref{defn:group,defn:srep}):
    $$\forall v \in \Rep(\overline{Q},V), \theta \in \ggot_V, \quad \mugras(\repinvolution(v)) = \lieinvolution(\mugras(v)) \text{ and }\ \repinvolution([\theta, v]) = [\lieinvolution(\theta), \repinvolution(v)].$$
\end{prop}

\begin{proof}
    Let $v \in \Rep(\overline{Q},V)$ and $\theta \in \ggot_V$ then:
    \begin{align*}
        \mugras(\repinvolution(v))_i & = \sum_{\stackrel{a \in Q}{ha  = i}} \repinvolution(v)_a\repinvolution(v)_{a^*} - \sum_{\stackrel{a \in Q}{ta = i}}\repinvolution(v)_{a^*}\repinvolution(v)_a \\
        & = \sum_{\stackrel{a \in Q}{ha  = i}} (J_{ha}^{-1}v_{\tau a}^*J_{ta})(J_{ ta}^{-1}v_{\tau a^*}^*J_{ha}) - \sum_{\stackrel{a \in Q}{ta = i}}(J_{ ta}^{-1}v_{\tau a^*}^*J_{ha})(J_{ ha}^{-1}v_{\tau a}^*J_{ta}) \\
        & = \sum_{\stackrel{a \in Q}{ha  = i}} J_{ha}^{-1}v_{\tau a}^*J_{ta}J_{ ta}^{-1}v_{\tau a^*}^*J_{ha} - \sum_{\stackrel{a \in Q}{ta = i}}J_{ta}^{-1}v_{\tau a^*}^*J_{ha}J_{ha}^{-1}v_{\tau a}^*J_{ta} \\
        & = - J_i^{-1}(\sum_{\stackrel{a \in Q}{ha  = \tau i}} v_av_{a^*} - \sum_{\stackrel{a \in Q}{ta = \tau i}} v_{a^*}v_a)^*J_i = \lieinvolution \mugras(v)_i, \\
        [\lieinvolution(\theta),\repinvolution(v)]_a & = \lieinvolution(\theta)_{ha}\repinvolution(v)_a - \repinvolution(v)_a\lieinvolution(\theta)_{ta} \\
        & = - s(a)(J_{ha}^{-1}\theta_{\tau ha}^*J_{ha})(J_{\tau ha}^{\ast -1}v_{\tau a}^*J_{ta}) -  (J_{\tau ha}^{\ast -1}v_{\tau a}^*J_{ta})(-J_{ta}^{-1}\theta_{\tau ta}^*J_{ta}) \\
        & = - s(a)(J_{\tau ha}^{-1*}\theta_{\tau ha}^*J_{ha}J_{ ha}^{ -1}v_{\tau a}^*J_{ta} -  J_{\tau ha}^{\ast -1}v_{\tau a}^*J_{ta}J_{ta}^{-1}\theta_{\tau ta}^*J_{ta}) \\
        & = s(a)J_{\tau ha}^{-1*}(\theta_{\tau ha}^*v_{\tau a}^* -  v_{\tau a}^*\theta_{\tau ta}^*)J_{ta} = \repinvolution([\theta,v])_a. \qedhere
    \end{align*}
\end{proof}

\begin{rmq}\label{rmq sign rules}
    The moment map and the bracket are bilinear:
    $$\mugras\colon \Rep(Q, V) \times \Rep(Q^{\opp}, V) \to \ggot_V \quad \text{and} \quad [\cdot, \cdot]\colon \ggot_V \times \Rep(\overline{Q}, V) \to \Rep(\overline{Q}, V).$$
    They respect the sign rules for the involutions $\repinvolution$ and $\lieinvolution$ (from \Cref{defn:srep,defn:group}); thus $$\mugras(\Srep(\overline{Q},V)) \subset \ggot_V^b.$$
\end{rmq}

\begin{cor}\label{moment of symmetric quivers}
    The restriction of the moment map $\mugras$ (from \Cref{defn:moment map quivers}) defines a moment map for orthogonal quivers:
    $$\mu\colon \Srep(\overline{Q},V) \to \Rep(\overline{Q},V) \overset{\mugras}{\to}\underset{\simeq \ggot_V^*}{\ggot_V} \to \underset{\simeq(\ggot^b_V)^*}{\ggot^b_V}.$$
\end{cor}

\begin{proof}
    In classical symplectic geometry, since $G^b_V \subset G_V$ is a closed subgroup, the moment map is obtained by considering the composition
    $$\Srep(\overline{Q},V) \to \Rep(\overline{Q},V) \to \ggot_V^* \to (\ggot_V^b)^*,$$
    where the last arrow $\ggot_V^* \to (\ggot_V^b)^*$ is induced by the natural closed immersion $\ggot_V^b \inj \ggot_V$.
    
    For quivers, there is a chosen identification $\ggot_V \overset{\kappa}{\simeq} \ggot_V^*$. For orthogonal quivers, before identifying with $\kappa$ and $\kappa^b$, the last arrow in
    $$\Srep(\overline{Q},V) \to \Rep(\overline{Q},V) \to \ggot_V^* \to (\ggot_V^b)^*$$
    is also given by the composition
    $$\ggot_V^* \overset{\kappa}{\simeq} \ggot_V = \ggot_V^b \oplus (\ggot_V)^{-\lieinvolution} \to \ggot_V^b \overset{\kappa^b}{\simeq} (\ggot_V^b)^*,$$
    since \Cref{lem: pairing kappa} shows that the following diagram commutes
    \begin{align*}
        \begin{matrix}
            \ggot_V = & \ggot_V^b \oplus\ggot_V^{-\lieinvolution}&  \surj& \ggot_V^b \\
            \downarrow \kappa& & & \downarrow \kappa^b \\
            (\ggot_V)^* = & (\ggot_V^b)^* \oplus(\ggot_V^{-\lieinvolution})^*&  \surj& (\ggot_V^b)^*. \\
        \end{matrix}
    \end{align*}
    \Cref{mu et crochet commute à t et tgot} ensures that the moment map $\mugras\colon \Rep(\overline{Q},V) \to \ggot_V$ sends $\Srep(\overline{Q},V)$ into $\ggot_V^b$. 
    Hence, the restricted map $\mu\colon \Srep(\overline{Q},V) \to \ggot_V^b$ is well defined. Since $\mugras$ is $G_V$-equivariant, it is in particular $G_V^b$-equivariant.
    Moreover, for every $\theta \in \ggot_V^b$, the tangent vector $[\theta,v]$ belongs to $T_v\Srep(\overline{Q},V)\subset T_v\Rep(\overline{Q},V)$. 
    The defining differential property of the moment map,
    $$\forall v \in \Rep(\overline{Q},V),\ \theta \in \ggot_V, \quad d_v\langle \mugras(v), \theta \rangle = \iota_{\theta_{\Rep}}\omega,$$
    remains valid after restricting to the closed invariant subspace $\Srep(\overline{Q},V)$ and to $\ggot_V^b$, provided that these restrictions agree with the identification induced by $\kappa$, which follows from \Cref{mu et crochet commute à t et tgot},
    \begin{align*}
        \forall v \in \Srep(\overline{Q},V),\theta \in \ggot_V^b, \quad d_v\langle \mu(v), \theta \rangle = \iota_{\theta_{\Srep}}\omega. \quad  \qedhere
    \end{align*}
\end{proof}

\begin{rmq}\label{rmq: orthosymplectic quivers}
    The different statements presented in \Cref{section:2} remain valid for ortho-symplectic quivers, in the following sense:
    \begin{itemize}
        \item An \emph{ortho-symplectic quiver} is a quiver equipped with a contravariant involution $(Q,\tau)$ and a sign function
        $$s\colon Q\sqcup I\longrightarrow\{\pm1\},$$
        such that $s_{\mid I}$ is $\tau$-invariant and, for every arrow $a\colon i\to j$, $s(a)s(\tau a)=s(ha)s(ta).$
        \item A \emph{signed space} $(V,b)$ associated with an ortho-symplectic quiver is a graded vector space $V = \bigoplus_{i \in I} V_i$ equipped with a nondegenerate bilinear form $b$ such that:
        \begin{itemize}
            \item for every fixed vertex $i\in I^\tau$, the restriction $b_i$ is symmetric if $s(i)=+1$ and skew-symmetric if $s(i)=-1$;
            \item for every pair of non-fixed vertices $\{i,\tau i\}$, $b$ which is symmetric on $V_i\oplus V_{\tau i}$ if $s(i)=+1$, skew-symmetric if $s(i)=-1$, and induces a perfect pairing between $V_i$ and $V_{\tau i}$.
        \end{itemize}
    \end{itemize}
    Orthogonal quivers are recovered as the special case where $s_{\mid I}=+1$.
\end{rmq}

\begin{rmq}
    Let $\zeta \in (\ggot^b_V)^{G^b_V}$. We can consider the two respective quiver varieties $$\Mgot_{\zeta,0,\dd} =\mu^{-1}(\zeta)\sslash G^b_{\dd}\text{ and } \Mgotgras_{\zeta,0,\dd} = \mugras^{-1}(\zeta)\sslash G_{\dd}.$$
    Then the involution $\repinvolution$ (see \Cref{defn:srep}) descends to an involution $[\repinvolution]$ on $\Mgotgras_{\zeta,0,\dd}$. 
    As shown in \cite[Proposition 9.2.1]{li2018}, there is a natural closed embedding
    $$\Mgot_{\zeta,0,\dd} \hookrightarrow (\Mgotgras_{\zeta,0,\dd})^{[\repinvolution]}.$$
    According to \cite{Nak25classical}, this embedding is stated to be an isomorphism.
\end{rmq}

\newpage
\section{Orthogonal quivers of affine Dynkin type}\label{section:3}

\subsection{Path algebra and preprojective algebra}

In this section, we fix an orthogonal quiver $Q$ and a signed space $(V,b)$ of dimension $\dd$ (see \Cref{defn:srep}).
We begin by recalling the definitions of the path algebra and the preprojective algebra associated with a quiver. 
For these notions, we refer to \cite{QrQv2016} and \cite{Ringel91}.
We also introduce several lemmas forming what we call a \emph{graphical calculus}, in order to classify cycles up to an equivalence relation (see \Cref{defn:invariant tr relation}).

\begin{defn}\label{defn:path algebra}
    The \emph{path algebra} $\CC Q$ is the associative $\CC$-algebra defined as the free $\CC$-vector space generated by the finite paths contained in $Q$, including the trivial paths $e_i$ of length zero at each vertex $i \in I$.
    The product of two paths $p$ and $q$ is defined as the concatenation $pq$ if $tp = hq$, and as zero otherwise. Hence, the elements $e_i$ are idempotents in $\CC Q$.
    The algebra is graded by path length; for a path $p$, its length is denoted $\ell(p)$.
    By definition, $I$ gives a basis $(e_i)_{i \in I}$ for $\CC Q_0$, while $Q$ gives a basis for $\CC Q_1$.
\end{defn}

\begin{thm}\label{bijection module}{\emph{\cite[Thm.~1.7]{QrQv2016}}}
    There exists a natural equivalence of categories
    $$\Rep(Q) \leftrightarrow \CC Q\text{-mod},$$
    between the category of representations of the quiver $Q$ and the category of left $\CC Q$-modules.
\end{thm}

\begin{cor}\label{bijection module 2}
    We have natural equivalences between the following full subcategories:
    \begin{align*}
        \Rep(Q,V) &\leftrightarrow \big\{ \varphi \in \Mor_{\CC\text{-}\mathrm{alg}}(\CC Q,\End_{\CC}(V)) \mid \forall i\in I, \varphi(e_i) = \mathrm{Id}_{V_i} \big\}, \\
        \Srep(Q,V) &\leftrightarrow \big\{ \varphi \in\Mor_{\CC\text{-}\mathrm{alg}}(\CC Q,\End_{\CC}(V))^{\repinvolution} \mid \forall i\in I, \varphi(e_i) = \mathrm{Id}_{V_i} \big\}.
    \end{align*}
\end{cor}

\begin{notation}\label{path morphism}
    For any path $p = a_1\cdots a_\ell$, we denote its representation under $v \in \Srep(Q,V)$ by
    $$v(p)=v(a_1)\cdots v(a_\ell).$$
    Throughout the sequel, we use the fact that every path $p$ defines an affine morphism $\Srep(Q,V)\to \Hom(V_{tp},V_{hp}) \subset \End_{\CC}(V).$
\end{notation}

The notion of a preprojective algebra $\Pi(Q)$ was first introduced in the work of Gelfand and Ponomarev \cite{GP79}.
The modern presentation as the quotient of the path algebra of the doubled quiver by the moment map relations emerged later in the work of Dlab and Ringel \cite{Dlab80}.

\begin{defn}\label{defn:preprojectiv algebra}{\cite[Introduction]{Ringel91}}
    The \emph{preprojective algebra} is the quotient noncommutative algebra $\Pi(Q)=\CC\overline{Q}/(\mu_Q),$ where:
    \begin{itemize}\label{defn:formal moment}
        \item $\mu_i = \sum_{\stackrel{a \in Q}{ha = i}} aa^* - \sum_{\stackrel{a \in Q}{ta = i}}a^*a$,
        \item $\mu_Q = \sum_{i \in I} \mu_i$ is the formal moment, and $(\mu_Q)$ is the two-sided ideal generated by $\mu_Q$.
    \end{itemize}
\end{defn}

\begin{lem}\label{bijection alg prepro}
    Let $\mugras\colon \Rep(\overline{Q},V) \to \ggot_V$ and $\mu\colon \Srep(\overline{Q},V) \to \ggot^b_V$ be the two moment maps from \Cref{defn:moment map quivers} and \Cref{moment of symmetric quivers}. 
    We have the following natural equivalences of full subcategories:
    \begin{align*}
        \mugras^{-1}(0) & \leftrightarrow \big\{ \varphi \in\Mor_{\CC\text{-}\mathrm{alg}}(\Pi(Q), \End_{\CC}(V))\mid \forall i\in I, \varphi(e_i) = \mathrm{Id}_{V_i} \big\}, \\
        \mu^{-1}(0) & \leftrightarrow \big\{ \varphi \in \Mor_{\CC\text{-}\mathrm{alg}}(\Pi(Q), \End_{\CC}(V))^{\repinvolution}\mid \forall i\in I, \varphi(e_i) = \mathrm{Id}_{V_i} \big\}.
    \end{align*}
\end{lem}

\begin{proof}
    Noting that $\ggot_V \subset \End_{\CC}(V)$, for every $v \in \Rep(\overline{Q},V)$ there is an identification between $v(\mu_Q)=v\!\left(\sum_{i \in I}\mu_i\right)$ (in the sense of \Cref{bijection module}) and the value of the moment map $\mugras(v)\in\ggot_V$.

    The equivalences are the restrictions of \Cref{bijection module 2} to the subcategories $\mugras^{-1}(0)$ and $\mu^{-1}(0)$, together with the universal property of quotient rings applied to the two-sided ideal generated by $(\mu_Q)$ and to the ring morphisms $v \in \Mor_{\CC\text{-}\mathrm{alg}}(\CC\overline{Q},\End_{\CC}(V))$.
\end{proof}

\begin{rmq}
    For a fixed orthogonal quiver, \Cref{bijection alg prepro} allows us to lift equalities between cycles in $\Pi(Q)$, without fixing a dimension vector, to trace-equivalence relations as defined in \Cref{defn:invariant tr relation} and explained in \Cref{rmq: prepro trace}.
\end{rmq}

Cycles are paths $a_1 \cdots a_{\ell}$ such that $t(a_{\ell})=h(a_1)$.
Their linear span forms a subalgebra of the path algebras $\CC Q,\CC \overline{Q}$ and $\Pi(Q)$, since the quotient is taken by the ideal $\mu_Q$ generated by linear combinations of cycles.
The representations of a cycle takes values in $\ggot_V = \bigoplus_{i \in I} \End_{\CC}(V_i)$.
For the study of quiver varieties, we will use invariants defined by these cycles.

\begin{defn}\label{defn:invariant tr relation}
    Let $C$ be a cycle of $\overline{Q}$. We denote by
    $\tr(C), \tr(\Lambda^kC), \det(C)$ the regular functions:
    \begin{align*}
        \tr(C)\colon\ & Z\to \CC, \quad 
        v \mapsto \tr(v(C)), \\
        \tr(\Lambda^kC)\colon\ & Z \to \CC, \quad 
        v \mapsto \tr(\Lambda^kv(C)), \\
        \det(C)\colon\ & Z\to \CC, \quad 
        v \mapsto \det(v(C)).
    \end{align*}
    For various closed subsets $Z \subset \Rep(\overline{Q},V)$, such as $\mu^{-1}(0)$ and $\Srep(\overline{Q},V)$,
    these functions define invariants under the actions of the groups $G_V$ and $G_V^b$, whenever $Z$ is stable under the corresponding group action.
    Moreover, $\tr(\Lambda^k C)$ is homogeneous of degree $k\cdot\ell(C)$.
    When $Z=\mu^{-1}(0)$, we write $C\sim_{\mathrm{tr}}C'$ and say that $C$ and $C'$ are trace-equivalent whenever $\tr(C)=\tr(C')$ in $\CC[\mu^{-1}(0)]$.
\end{defn}

To prove \Cref{Thm}, we need two results from invariant theory, namely \Cref{thm:invariant procesi and zubkov}, describing the invariant rings of quiver representations and symmetric quiver representations. 
We then restrict these invariants to the closed subset $\mu^{-1}(0)\subset \Srep(\overline{Q},V).$

\begin{thm}\label{thm:invariant procesi and zubkov}
{\emph{\cite[Thm.~1]{LbProcesi90}; \cite[Thm.~1--2]{Zub03}}}\\
    The two invariant rings $\CC[\Rep(Q,V)]^{G_V}$ and $\CC[\Srep(Q,V)]^{G^b_V}$ are generated by traces of cycles contained in $Q$:
    $$\CC[\Rep(Q,V)]^{G_V} = \CC\big[\tr(C):\Rep(Q,V)\to \CC \mid \forall C \text{ cycle in } Q\big],$$
    $$\CC[\Srep(Q,V)]^{G_V^b} = \CC\big[\tr(C):\Srep(Q,V) \to \CC \mid  \forall C \text{ cycle in } Q\big],$$
    and the morphism $\CC[\Rep(Q,V)]^{G_V} \to \CC[\Srep(Q,V)]^{G_V^b}$ induced by the restriction of regular functions is surjective.
\end{thm}

\subsection{Graphical calculus}\label{graphical calculus}
The \emph{graphical calculus} is used to establish relations between invariants associated with cycles contained in $\overline{Q}$, whose basic rules are given by the following remark and five lemmas.

For simplicity throughout this section, we fix an orthogonal quiver $Q$ and a symmetric dimension vector $\dd$, and we assume that $(V,b)$ is the standard signed space $(\CC^{\dd},b)$ defined in \Cref{standard model}.
Recall that $\mu \colon \Srep(\overline{Q},\dd)\longrightarrow \ggot_{\dd}^b$ is the moment map introduced in \Cref{moment of symmetric quivers}.

\begin{rmq}\label{Hamilton-Cayley scheme morphism}
    Let $C$ be a cycle in $Q$ based at a vertex $i \in I$. Consider the affine scheme morphism
    \begin{align*}
        \chi_C(C) & =\sum_{k=0}^{\dd_i} (-1)^k \tr(\Lambda^k C)\, C^{\dd_i -k} \colon \Rep(Q,V) \to \ggot_V \\
        & v \mapsto \sum_{k=0}^{\dd_i} (-1)^k \tr(\Lambda^k v(C))\, v(C)^{\dd_i-k} .
    \end{align*}
    Then the Cayley--Hamilton relation gives that this morphism $\chi_C(C)$ is equal to the constant morphism $0$.
\end{rmq}

The trace-equivalence relation $\sim_{\tr}$ (from \Cref{defn:invariant tr relation}) can be translated into graphical transformations that deform cycles drawn in $\overline{Q}$.
First, note that when $(h,t)\colon\overline{Q} \to I \times I$ is injective, i.e.\ there are no distinct arrows $a \neq a'$ such that $ha=ha'$ and $ta=ta'$, any path drawn defines at most one path in the path algebra $\CC\overline{Q}$.
With this convention, we represent a path in $\CC\overline{Q}$ or $\Pi(Q)$ by its diagram, keeping in mind that each diagram corresponds to a formal path and that the composition of paths is read from right to left.
Any computation involving such diagrams can therefore be translated into a formal composition of arrows.

\begin{lem}\label{lem:pqqp}
    For two paths $p$ and $q$, if $C=pq$ and $C'=qp$ are cycles, then they define the same invariant functions:
$$\text{for all }k\geq 0, \qquad \tr(\Lambda^k C)=\tr(\Lambda^k C')\in\CC[\Rep(Q,\dd)].$$
    In particular, $C\sim_{\tr}C'$.
\end{lem}
\begin{center}
\begin{tikzpicture}[x=0.5pt,y=0.5pt,yscale=-1,xscale=1]
\draw  [color={rgb, 255:red, 0; green, 0; blue, 0 }  ,draw opacity=1 ][fill={rgb, 255:red, 0; green, 0; blue, 0 }  ,fill opacity=1 ][line width=0.75]  (314,138.33) .. controls (314,136.68) and (315.34,135.33) .. (317,135.33) .. controls (318.66,135.33) and (320,136.68) .. (320,138.33) .. controls (320,139.99) and (318.66,141.33) .. (317,141.33) .. controls (315.34,141.33) and (314,139.99) .. (314,138.33) -- cycle ;
\draw  [color={rgb, 255:red, 0; green, 0; blue, 0 }  ,draw opacity=1 ][fill={rgb, 255:red, 0; green, 0; blue, 0 }  ,fill opacity=1 ][line width=0.75]  (464,138.33) .. controls (464,136.68) and (465.34,135.33) .. (467,135.33) .. controls (468.66,135.33) and (470,136.68) .. (470,138.33) .. controls (470,139.99) and (468.66,141.33) .. (467,141.33) .. controls (465.34,141.33) and (464,139.99) .. (464,138.33) -- cycle ;
\draw  [color={rgb, 255:red, 0; green, 0; blue, 0 }  ,draw opacity=1 ][fill={rgb, 255:red, 0; green, 0; blue, 0 }  ,fill opacity=1 ][line width=0.75]  (349.5,96.33) .. controls (349.5,94.68) and (350.84,93.33) .. (352.5,93.33) .. controls (354.16,93.33) and (355.5,94.68) .. (355.5,96.33) .. controls (355.5,97.99) and (354.16,99.33) .. (352.5,99.33) .. controls (350.84,99.33) and (349.5,97.99) .. (349.5,96.33) -- cycle ;
\draw  [color={rgb, 255:red, 0; green, 0; blue, 0 }  ,draw opacity=1 ][fill={rgb, 255:red, 0; green, 0; blue, 0 }  ,fill opacity=1 ][line width=0.75]  (428.5,96.33) .. controls (428.5,94.68) and (429.84,93.33) .. (431.5,93.33) .. controls (433.16,93.33) and (434.5,94.68) .. (434.5,96.33) .. controls (434.5,97.99) and (433.16,99.33) .. (431.5,99.33) .. controls (429.84,99.33) and (428.5,97.99) .. (428.5,96.33) -- cycle ;
\draw  [color={rgb, 255:red, 0; green, 0; blue, 0 }  ,draw opacity=1 ][fill={rgb, 255:red, 0; green, 0; blue, 0 }  ,fill opacity=1 ][line width=0.75]  (349.5,181.33) .. controls (349.5,179.68) and (350.84,178.33) .. (352.5,178.33) .. controls (354.16,178.33) and (355.5,179.68) .. (355.5,181.33) .. controls (355.5,182.99) and (354.16,184.33) .. (352.5,184.33) .. controls (350.84,184.33) and (349.5,182.99) .. (349.5,181.33) -- cycle ;
\draw  [color={rgb, 255:red, 0; green, 0; blue, 0 }  ,draw opacity=1 ][fill={rgb, 255:red, 0; green, 0; blue, 0 }  ,fill opacity=1 ][line width=0.75]  (428.5,181.33) .. controls (428.5,179.68) and (429.84,178.33) .. (431.5,178.33) .. controls (433.16,178.33) and (434.5,179.68) .. (434.5,181.33) .. controls (434.5,182.99) and (433.16,184.33) .. (431.5,184.33) .. controls (429.84,184.33) and (428.5,182.99) .. (428.5,181.33) -- cycle ;
\draw [line width=0.75]    (424.33,180.77) .. controls (406,181.1) and (372.33,178.77) .. (352.67,173.77) .. controls (333,168.77) and (322.67,155.43) .. (323.33,138.1) .. controls (324,120.77) and (328.67,107.43) .. (352,102.77) .. controls (375.33,98.1) and (406.33,95.43) .. (429,102.1) .. controls (451.67,108.77) and (460.67,119.1) .. (460.33,138.43) .. controls (460.02,156.7) and (451.96,166.04) .. (439.26,176.02) ;
\draw [shift={(437,177.77)}, rotate = 321.86] [fill={rgb, 255:red, 0; green, 0; blue, 0 }  ][line width=0.08]  [draw opacity=0] (10.72,-5.15) -- (0,0) -- (10.72,5.15) -- (7.12,0) -- cycle    ;
\draw  [color={rgb, 255:red, 0; green, 0; blue, 0 }  ,draw opacity=1 ][fill={rgb, 255:red, 0; green, 0; blue, 0 }  ,fill opacity=1 ][line width=0.75]  (535,138.67) .. controls (535,137.01) and (536.34,135.67) .. (538,135.67) .. controls (539.66,135.67) and (541,137.01) .. (541,138.67) .. controls (541,140.32) and (539.66,141.67) .. (538,141.67) .. controls (536.34,141.67) and (535,140.32) .. (535,138.67) -- cycle ;
\draw  [color={rgb, 255:red, 0; green, 0; blue, 0 }  ,draw opacity=1 ][fill={rgb, 255:red, 0; green, 0; blue, 0 }  ,fill opacity=1 ][line width=0.75]  (685,138.67) .. controls (685,137.01) and (686.34,135.67) .. (688,135.67) .. controls (689.66,135.67) and (691,137.01) .. (691,138.67) .. controls (691,140.32) and (689.66,141.67) .. (688,141.67) .. controls (686.34,141.67) and (685,140.32) .. (685,138.67) -- cycle ;
\draw  [color={rgb, 255:red, 0; green, 0; blue, 0 }  ,draw opacity=1 ][fill={rgb, 255:red, 0; green, 0; blue, 0 }  ,fill opacity=1 ][line width=0.75]  (570.5,96.67) .. controls (570.5,95.01) and (571.84,93.67) .. (573.5,93.67) .. controls (575.16,93.67) and (576.5,95.01) .. (576.5,96.67) .. controls (576.5,98.32) and (575.16,99.67) .. (573.5,99.67) .. controls (571.84,99.67) and (570.5,98.32) .. (570.5,96.67) -- cycle ;
\draw  [color={rgb, 255:red, 0; green, 0; blue, 0 }  ,draw opacity=1 ][fill={rgb, 255:red, 0; green, 0; blue, 0 }  ,fill opacity=1 ][line width=0.75]  (649.5,96.67) .. controls (649.5,95.01) and (650.84,93.67) .. (652.5,93.67) .. controls (654.16,93.67) and (655.5,95.01) .. (655.5,96.67) .. controls (655.5,98.32) and (654.16,99.67) .. (652.5,99.67) .. controls (650.84,99.67) and (649.5,98.32) .. (649.5,96.67) -- cycle ;
\draw  [color={rgb, 255:red, 0; green, 0; blue, 0 }  ,draw opacity=1 ][fill={rgb, 255:red, 0; green, 0; blue, 0 }  ,fill opacity=1 ][line width=0.75]  (570.5,181.67) .. controls (570.5,180.01) and (571.84,178.67) .. (573.5,178.67) .. controls (575.16,178.67) and (576.5,180.01) .. (576.5,181.67) .. controls (576.5,183.32) and (575.16,184.67) .. (573.5,184.67) .. controls (571.84,184.67) and (570.5,183.32) .. (570.5,181.67) -- cycle ;
\draw  [color={rgb, 255:red, 0; green, 0; blue, 0 }  ,draw opacity=1 ][fill={rgb, 255:red, 0; green, 0; blue, 0 }  ,fill opacity=1 ][line width=0.75]  (649.5,181.67) .. controls (649.5,180.01) and (650.84,178.67) .. (652.5,178.67) .. controls (654.16,178.67) and (655.5,180.01) .. (655.5,181.67) .. controls (655.5,183.32) and (654.16,184.67) .. (652.5,184.67) .. controls (650.84,184.67) and (649.5,183.32) .. (649.5,181.67) -- cycle ;
\draw [line width=0.75]    (540.67,132.1) .. controls (546.67,111.77) and (556.67,106.43) .. (575,102.1) .. controls (593.33,97.77) and (623.67,95.43) .. (650,102.1) .. controls (676.33,108.77) and (681,120.43) .. (681.67,138.1) .. controls (682.33,155.77) and (674.33,172.43) .. (649,175.43) .. controls (623.67,178.43) and (598.67,176.43) .. (576.33,174.77) .. controls (555.34,173.2) and (550.25,161.62) .. (544.46,148.05) ;
\draw [shift={(543.33,145.43)}, rotate = 66.87] [fill={rgb, 255:red, 0; green, 0; blue, 0 }  ][line width=0.08]  [draw opacity=0] (10.72,-5.15) -- (0,0) -- (10.72,5.15) -- (7.12,0) -- cycle    ;
\draw  [color={rgb, 255:red, 0; green, 0; blue, 0 }  ,draw opacity=1 ][fill={rgb, 255:red, 0; green, 0; blue, 0 }  ,fill opacity=1 ] (63,140) .. controls (63,138.34) and (64.34,137) .. (66,137) .. controls (67.66,137) and (69,138.34) .. (69,140) .. controls (69,141.66) and (67.66,143) .. (66,143) .. controls (64.34,143) and (63,141.66) .. (63,140) -- cycle ;
\draw  [color={rgb, 255:red, 0; green, 0; blue, 0 }  ,draw opacity=1 ][fill={rgb, 255:red, 0; green, 0; blue, 0 }  ,fill opacity=1 ] (213,140) .. controls (213,138.34) and (214.34,137) .. (216,137) .. controls (217.66,137) and (219,138.34) .. (219,140) .. controls (219,141.66) and (217.66,143) .. (216,143) .. controls (214.34,143) and (213,141.66) .. (213,140) -- cycle ;
\draw  [color={rgb, 255:red, 0; green, 0; blue, 0 }  ,draw opacity=1 ][fill={rgb, 255:red, 0; green, 0; blue, 0 }  ,fill opacity=1 ] (98.5,98) .. controls (98.5,96.34) and (99.84,95) .. (101.5,95) .. controls (103.16,95) and (104.5,96.34) .. (104.5,98) .. controls (104.5,99.66) and (103.16,101) .. (101.5,101) .. controls (99.84,101) and (98.5,99.66) .. (98.5,98) -- cycle ;
\draw  [color={rgb, 255:red, 0; green, 0; blue, 0 }  ,draw opacity=1 ][fill={rgb, 255:red, 0; green, 0; blue, 0 }  ,fill opacity=1 ] (177.5,98) .. controls (177.5,96.34) and (178.84,95) .. (180.5,95) .. controls (182.16,95) and (183.5,96.34) .. (183.5,98) .. controls (183.5,99.66) and (182.16,101) .. (180.5,101) .. controls (178.84,101) and (177.5,99.66) .. (177.5,98) -- cycle ;
\draw  [color={rgb, 255:red, 0; green, 0; blue, 0 }  ,draw opacity=1 ][fill={rgb, 255:red, 0; green, 0; blue, 0 }  ,fill opacity=1 ] (98.5,183) .. controls (98.5,181.34) and (99.84,180) .. (101.5,180) .. controls (103.16,180) and (104.5,181.34) .. (104.5,183) .. controls (104.5,184.66) and (103.16,186) .. (101.5,186) .. controls (99.84,186) and (98.5,184.66) .. (98.5,183) -- cycle ;
\draw  [color={rgb, 255:red, 0; green, 0; blue, 0 }  ,draw opacity=1 ][fill={rgb, 255:red, 0; green, 0; blue, 0 }  ,fill opacity=1 ] (177.5,183) .. controls (177.5,181.34) and (178.84,180) .. (180.5,180) .. controls (182.16,180) and (183.5,181.34) .. (183.5,183) .. controls (183.5,184.66) and (182.16,186) .. (180.5,186) .. controls (178.84,186) and (177.5,184.66) .. (177.5,183) -- cycle ;
\draw [line width=1.5]    (67.03,133.38) .. controls (71.62,118.53) and (73.3,113.84) .. (91.41,103.04) ;
\draw [shift={(94.74,101.09)}, rotate = 149.01] [fill={rgb, 255:red, 0; green, 0; blue, 0 }  ][line width=0.08]  [draw opacity=0] (13.4,-6.43) -- (0,0) -- (13.4,6.44) -- (8.9,0) -- cycle    ;
\draw [line width=1.5]    (108.74,95.09) .. controls (127.96,91.79) and (143.11,91.34) .. (170.25,94.6) ;
\draw [shift={(174.17,95.09)}, rotate = 186.82] [fill={rgb, 255:red, 0; green, 0; blue, 0 }  ][line width=0.08]  [draw opacity=0] (13.4,-6.43) -- (0,0) -- (13.4,6.44) -- (8.9,0) -- cycle    ;
\draw [line width=1.5]    (173.88,184.63) .. controls (155.72,189.41) and (137.88,190.23) .. (112.15,185.35) ;
\draw [shift={(108.45,184.63)}, rotate = 10.54] [fill={rgb, 255:red, 0; green, 0; blue, 0 }  ][line width=0.08]  [draw opacity=0] (13.4,-6.43) -- (0,0) -- (13.4,6.44) -- (8.9,0) -- cycle    ;
\draw [line width=1.5]    (214.74,147.38) .. controls (210.15,162.23) and (208.46,166.92) .. (190.36,177.71) ;
\draw [shift={(187.03,179.66)}, rotate = 329.01] [fill={rgb, 255:red, 0; green, 0; blue, 0 }  ][line width=0.08]  [draw opacity=0] (13.4,-6.43) -- (0,0) -- (13.4,6.44) -- (8.9,0) -- cycle    ;
\draw [line width=1.5]    (187.96,101.07) .. controls (201.79,108.17) and (206.12,110.64) .. (213.6,130.35) ;
\draw [shift={(214.95,133.97)}, rotate = 249.01] [fill={rgb, 255:red, 0; green, 0; blue, 0 }  ][line width=0.08]  [draw opacity=0] (13.4,-6.43) -- (0,0) -- (13.4,6.44) -- (8.9,0) -- cycle    ;
\draw [line width=1.5]    (94.37,179.68) .. controls (80.55,172.58) and (76.22,170.11) .. (68.73,150.4) ;
\draw [shift={(67.39,146.78)}, rotate = 69.01] [fill={rgb, 255:red, 0; green, 0; blue, 0 }  ][line width=0.08]  [draw opacity=0] (13.4,-6.43) -- (0,0) -- (13.4,6.44) -- (8.9,0) -- cycle    ;

\draw (480.6,131) node [anchor=north west][inner sep=0.75pt]    {$\sim _{\tr}{}$};
\draw (18,128.4) node [anchor=north west][inner sep=0.75pt]    {$Q=$};
\draw (-103,104.67) node [anchor=north west][inner sep=0.75pt]   [align=left] {\begin{minipage}[lt]{61.68pt}\setlength\topsep{0pt}
For instance 
\begin{center}
if we set
\end{center}

\end{minipage}};
\draw (252,129) node [anchor=north west][inner sep=0.75pt]   [align=left] {then };
\end{tikzpicture}
\end{center}

\begin{proof}
    Let $v \in \Rep(Q,\dd),k\geq 0$ and $C = pq$ a cycle: 
    \begin{align*}
        \tr(\Lambda^kv(C)) = \tr(\Lambda^kv(p)\Lambda^kv(q)) = \tr(\Lambda^kv(q)\Lambda^kv(p)) = \tr(\Lambda^kv(C')).\quad \qedhere
    \end{align*}
\end{proof}

\begin{lem}\label{lem:tracetau}
    Let $ C = a_1 \cdots a_{\ell}$ be a cycle in $\overline{Q}$.
    Then $\tau(C) \sim_{\tr} \pm C$.
    More precisely, $\tau(C) \sim_{\tr} \left(\prod_{k=1}^\ell s(a_k)\right) C$.
\end{lem}
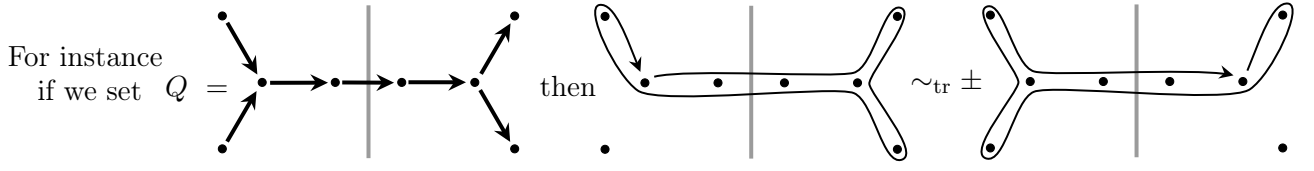
    \begin{center}
\begin{tikzpicture}[x=0.5pt,y=0.5pt,yscale=-1,xscale=1]
\draw [color={rgb, 255:red, 155; green, 155; blue, 155 }  ,draw opacity=1 ][line width=1.5]    (409,85.33) -- (409,200.33) ;
\draw  [color={rgb, 255:red, 0; green, 0; blue, 0 }  ,draw opacity=1 ][fill={rgb, 255:red, 0; green, 0; blue, 0 }  ,fill opacity=1 ] (326,143.33) .. controls (326,141.68) and (327.34,140.33) .. (329,140.33) .. controls (330.66,140.33) and (332,141.68) .. (332,143.33) .. controls (332,144.99) and (330.66,146.33) .. (329,146.33) .. controls (327.34,146.33) and (326,144.99) .. (326,143.33) -- cycle ;
\draw  [color={rgb, 255:red, 0; green, 0; blue, 0 }  ,draw opacity=1 ][fill={rgb, 255:red, 0; green, 0; blue, 0 }  ,fill opacity=1 ] (381,143.33) .. controls (381,141.68) and (382.34,140.33) .. (384,140.33) .. controls (385.66,140.33) and (387,141.68) .. (387,143.33) .. controls (387,144.99) and (385.66,146.33) .. (384,146.33) .. controls (382.34,146.33) and (381,144.99) .. (381,143.33) -- cycle ;
\draw  [color={rgb, 255:red, 0; green, 0; blue, 0 }  ,draw opacity=1 ][fill={rgb, 255:red, 0; green, 0; blue, 0 }  ,fill opacity=1 ] (486,143.33) .. controls (486,141.68) and (487.34,140.33) .. (489,140.33) .. controls (490.66,140.33) and (492,141.68) .. (492,143.33) .. controls (492,144.99) and (490.66,146.33) .. (489,146.33) .. controls (487.34,146.33) and (486,144.99) .. (486,143.33) -- cycle ;
\draw  [color={rgb, 255:red, 0; green, 0; blue, 0 }  ,draw opacity=1 ][fill={rgb, 255:red, 0; green, 0; blue, 0 }  ,fill opacity=1 ] (296,193.33) .. controls (296,191.68) and (297.34,190.33) .. (299,190.33) .. controls (300.66,190.33) and (302,191.68) .. (302,193.33) .. controls (302,194.99) and (300.66,196.33) .. (299,196.33) .. controls (297.34,196.33) and (296,194.99) .. (296,193.33) -- cycle ;
\draw  [color={rgb, 255:red, 0; green, 0; blue, 0 }  ,draw opacity=1 ][fill={rgb, 255:red, 0; green, 0; blue, 0 }  ,fill opacity=1 ] (296,93.33) .. controls (296,91.68) and (297.34,90.33) .. (299,90.33) .. controls (300.66,90.33) and (302,91.68) .. (302,93.33) .. controls (302,94.99) and (300.66,96.33) .. (299,96.33) .. controls (297.34,96.33) and (296,94.99) .. (296,93.33) -- cycle ;
\draw  [color={rgb, 255:red, 0; green, 0; blue, 0 }  ,draw opacity=1 ][fill={rgb, 255:red, 0; green, 0; blue, 0 }  ,fill opacity=1 ] (516,193.33) .. controls (516,191.68) and (517.34,190.33) .. (519,190.33) .. controls (520.66,190.33) and (522,191.68) .. (522,193.33) .. controls (522,194.99) and (520.66,196.33) .. (519,196.33) .. controls (517.34,196.33) and (516,194.99) .. (516,193.33) -- cycle ;
\draw  [color={rgb, 255:red, 0; green, 0; blue, 0 }  ,draw opacity=1 ][fill={rgb, 255:red, 0; green, 0; blue, 0 }  ,fill opacity=1 ] (516,93.33) .. controls (516,91.68) and (517.34,90.33) .. (519,90.33) .. controls (520.66,90.33) and (522,91.68) .. (522,93.33) .. controls (522,94.99) and (520.66,96.33) .. (519,96.33) .. controls (517.34,96.33) and (516,94.99) .. (516,93.33) -- cycle ;
\draw  [color={rgb, 255:red, 0; green, 0; blue, 0 }  ,draw opacity=1 ][fill={rgb, 255:red, 0; green, 0; blue, 0 }  ,fill opacity=1 ] (431,143.33) .. controls (431,141.68) and (432.34,140.33) .. (434,140.33) .. controls (435.66,140.33) and (437,141.68) .. (437,143.33) .. controls (437,144.99) and (435.66,146.33) .. (434,146.33) .. controls (432.34,146.33) and (431,144.99) .. (431,143.33) -- cycle ;
\draw [line width=0.75]    (336,138.73) .. controls (361.6,132.33) and (470.17,139.82) .. (485.89,135.82) .. controls (501.6,131.82) and (508.55,82.37) .. (521.6,86.96) .. controls (534.65,91.55) and (497.31,136.39) .. (497.89,144.1) .. controls (498.46,151.82) and (536.91,196.16) .. (522.4,201.13) .. controls (507.89,206.1) and (501.31,157.53) .. (486.4,151.13) .. controls (471.49,144.73) and (340.4,160.73) .. (323.6,149.13) .. controls (306.8,137.53) and (282.8,86.73) .. (294.8,84.73) .. controls (306.32,82.81) and (315.63,110.38) .. (324.49,132.41) ;
\draw [shift={(325.6,135.13)}, rotate = 247.67] [fill={rgb, 255:red, 0; green, 0; blue, 0 }  ][line width=0.08]  [draw opacity=0] (10.72,-5.15) -- (0,0) -- (10.72,5.15) -- (7.12,0) -- cycle    ;
\draw [color={rgb, 255:red, 155; green, 155; blue, 155 }  ,draw opacity=1 ][line width=1.5]    (698.97,84.33) -- (698.97,199.33) ;
\draw  [color={rgb, 255:red, 0; green, 0; blue, 0 }  ,draw opacity=1 ][fill={rgb, 255:red, 0; green, 0; blue, 0 }  ,fill opacity=1 ] (781.97,142.33) .. controls (781.97,140.68) and (780.62,139.33) .. (778.97,139.33) .. controls (777.31,139.33) and (775.97,140.68) .. (775.97,142.33) .. controls (775.97,143.99) and (777.31,145.33) .. (778.97,145.33) .. controls (780.62,145.33) and (781.97,143.99) .. (781.97,142.33) -- cycle ;
\draw  [color={rgb, 255:red, 0; green, 0; blue, 0 }  ,draw opacity=1 ][fill={rgb, 255:red, 0; green, 0; blue, 0 }  ,fill opacity=1 ] (726.97,142.33) .. controls (726.97,140.68) and (725.62,139.33) .. (723.97,139.33) .. controls (722.31,139.33) and (720.97,140.68) .. (720.97,142.33) .. controls (720.97,143.99) and (722.31,145.33) .. (723.97,145.33) .. controls (725.62,145.33) and (726.97,143.99) .. (726.97,142.33) -- cycle ;
\draw  [color={rgb, 255:red, 0; green, 0; blue, 0 }  ,draw opacity=1 ][fill={rgb, 255:red, 0; green, 0; blue, 0 }  ,fill opacity=1 ] (621.97,142.33) .. controls (621.97,140.68) and (620.62,139.33) .. (618.97,139.33) .. controls (617.31,139.33) and (615.97,140.68) .. (615.97,142.33) .. controls (615.97,143.99) and (617.31,145.33) .. (618.97,145.33) .. controls (620.62,145.33) and (621.97,143.99) .. (621.97,142.33) -- cycle ;
\draw  [color={rgb, 255:red, 0; green, 0; blue, 0 }  ,draw opacity=1 ][fill={rgb, 255:red, 0; green, 0; blue, 0 }  ,fill opacity=1 ] (811.97,192.33) .. controls (811.97,190.68) and (810.62,189.33) .. (808.97,189.33) .. controls (807.31,189.33) and (805.97,190.68) .. (805.97,192.33) .. controls (805.97,193.99) and (807.31,195.33) .. (808.97,195.33) .. controls (810.62,195.33) and (811.97,193.99) .. (811.97,192.33) -- cycle ;
\draw  [color={rgb, 255:red, 0; green, 0; blue, 0 }  ,draw opacity=1 ][fill={rgb, 255:red, 0; green, 0; blue, 0 }  ,fill opacity=1 ] (811.97,92.33) .. controls (811.97,90.68) and (810.62,89.33) .. (808.97,89.33) .. controls (807.31,89.33) and (805.97,90.68) .. (805.97,92.33) .. controls (805.97,93.99) and (807.31,95.33) .. (808.97,95.33) .. controls (810.62,95.33) and (811.97,93.99) .. (811.97,92.33) -- cycle ;
\draw  [color={rgb, 255:red, 0; green, 0; blue, 0 }  ,draw opacity=1 ][fill={rgb, 255:red, 0; green, 0; blue, 0 }  ,fill opacity=1 ] (591.97,192.33) .. controls (591.97,190.68) and (590.62,189.33) .. (588.97,189.33) .. controls (587.31,189.33) and (585.97,190.68) .. (585.97,192.33) .. controls (585.97,193.99) and (587.31,195.33) .. (588.97,195.33) .. controls (590.62,195.33) and (591.97,193.99) .. (591.97,192.33) -- cycle ;
\draw  [color={rgb, 255:red, 0; green, 0; blue, 0 }  ,draw opacity=1 ][fill={rgb, 255:red, 0; green, 0; blue, 0 }  ,fill opacity=1 ] (591.97,92.33) .. controls (591.97,90.68) and (590.62,89.33) .. (588.97,89.33) .. controls (587.31,89.33) and (585.97,90.68) .. (585.97,92.33) .. controls (585.97,93.99) and (587.31,95.33) .. (588.97,95.33) .. controls (590.62,95.33) and (591.97,93.99) .. (591.97,92.33) -- cycle ;
\draw  [color={rgb, 255:red, 0; green, 0; blue, 0 }  ,draw opacity=1 ][fill={rgb, 255:red, 0; green, 0; blue, 0 }  ,fill opacity=1 ] (676.97,142.33) .. controls (676.97,140.68) and (675.62,139.33) .. (673.97,139.33) .. controls (672.31,139.33) and (670.97,140.68) .. (670.97,142.33) .. controls (670.97,143.99) and (672.31,145.33) .. (673.97,145.33) .. controls (675.62,145.33) and (676.97,143.99) .. (676.97,142.33) -- cycle ;
\draw [line width=0.75]    (768.98,137.11) .. controls (738.88,131.84) and (637.25,138.68) .. (622.08,134.82) .. controls (606.37,130.82) and (599.41,81.37) .. (586.37,85.96) .. controls (573.32,90.55) and (610.65,135.39) .. (610.08,143.1) .. controls (609.51,150.82) and (571.05,195.16) .. (585.57,200.13) .. controls (600.08,205.1) and (606.65,156.53) .. (621.57,150.13) .. controls (636.48,143.73) and (767.57,159.73) .. (784.37,148.13) .. controls (801.17,136.53) and (825.17,85.73) .. (813.17,83.73) .. controls (801.17,81.73) and (791.57,111.73) .. (782.37,134.13) ;
\draw [shift={(771.97,137.73)}, rotate = 194.04] [fill={rgb, 255:red, 0; green, 0; blue, 0 }  ][line width=0.08]  [draw opacity=0] (10.72,-5.15) -- (0,0) -- (10.72,5.15) -- (7.12,0) -- cycle    ;
\draw [color={rgb, 255:red, 155; green, 155; blue, 155 }  ,draw opacity=1 ][line width=1.5]    (121,85) -- (121,200) ;
\draw  [color={rgb, 255:red, 0; green, 0; blue, 0 }  ,draw opacity=1 ][fill={rgb, 255:red, 0; green, 0; blue, 0 }  ,fill opacity=1 ] (38,143) .. controls (38,141.34) and (39.34,140) .. (41,140) .. controls (42.66,140) and (44,141.34) .. (44,143) .. controls (44,144.66) and (42.66,146) .. (41,146) .. controls (39.34,146) and (38,144.66) .. (38,143) -- cycle ;
\draw  [color={rgb, 255:red, 0; green, 0; blue, 0 }  ,draw opacity=1 ][fill={rgb, 255:red, 0; green, 0; blue, 0 }  ,fill opacity=1 ] (93,143) .. controls (93,141.34) and (94.34,140) .. (96,140) .. controls (97.66,140) and (99,141.34) .. (99,143) .. controls (99,144.66) and (97.66,146) .. (96,146) .. controls (94.34,146) and (93,144.66) .. (93,143) -- cycle ;
\draw  [color={rgb, 255:red, 0; green, 0; blue, 0 }  ,draw opacity=1 ][fill={rgb, 255:red, 0; green, 0; blue, 0 }  ,fill opacity=1 ] (198,143) .. controls (198,141.34) and (199.34,140) .. (201,140) .. controls (202.66,140) and (204,141.34) .. (204,143) .. controls (204,144.66) and (202.66,146) .. (201,146) .. controls (199.34,146) and (198,144.66) .. (198,143) -- cycle ;
\draw  [color={rgb, 255:red, 0; green, 0; blue, 0 }  ,draw opacity=1 ][fill={rgb, 255:red, 0; green, 0; blue, 0 }  ,fill opacity=1 ] (8,193) .. controls (8,191.34) and (9.34,190) .. (11,190) .. controls (12.66,190) and (14,191.34) .. (14,193) .. controls (14,194.66) and (12.66,196) .. (11,196) .. controls (9.34,196) and (8,194.66) .. (8,193) -- cycle ;
\draw  [color={rgb, 255:red, 0; green, 0; blue, 0 }  ,draw opacity=1 ][fill={rgb, 255:red, 0; green, 0; blue, 0 }  ,fill opacity=1 ] (8,93) .. controls (8,91.34) and (9.34,90) .. (11,90) .. controls (12.66,90) and (14,91.34) .. (14,93) .. controls (14,94.66) and (12.66,96) .. (11,96) .. controls (9.34,96) and (8,94.66) .. (8,93) -- cycle ;
\draw  [color={rgb, 255:red, 0; green, 0; blue, 0 }  ,draw opacity=1 ][fill={rgb, 255:red, 0; green, 0; blue, 0 }  ,fill opacity=1 ] (228,193) .. controls (228,191.34) and (229.34,190) .. (231,190) .. controls (232.66,190) and (234,191.34) .. (234,193) .. controls (234,194.66) and (232.66,196) .. (231,196) .. controls (229.34,196) and (228,194.66) .. (228,193) -- cycle ;
\draw  [color={rgb, 255:red, 0; green, 0; blue, 0 }  ,draw opacity=1 ][fill={rgb, 255:red, 0; green, 0; blue, 0 }  ,fill opacity=1 ] (228,93) .. controls (228,91.34) and (229.34,90) .. (231,90) .. controls (232.66,90) and (234,91.34) .. (234,93) .. controls (234,94.66) and (232.66,96) .. (231,96) .. controls (229.34,96) and (228,94.66) .. (228,93) -- cycle ;
\draw [line width=1.5]    (14.76,98.1) -- (35.35,133.64) ;
\draw [shift={(37.36,137.1)}, rotate = 239.91] [fill={rgb, 255:red, 0; green, 0; blue, 0 }  ][line width=0.08]  [draw opacity=0] (13.4,-6.43) -- (0,0) -- (13.4,6.44) -- (8.9,0) -- cycle    ;
\draw [line width=1.5]    (14.56,187.5) -- (34.74,152.95) ;
\draw [shift={(36.76,149.5)}, rotate = 120.29] [fill={rgb, 255:red, 0; green, 0; blue, 0 }  ][line width=0.08]  [draw opacity=0] (13.4,-6.43) -- (0,0) -- (13.4,6.44) -- (8.9,0) -- cycle    ;
\draw [line width=1.5]    (46.96,143.3) -- (86.53,143.22) ;
\draw [shift={(90.53,143.22)}, rotate = 179.89] [fill={rgb, 255:red, 0; green, 0; blue, 0 }  ][line width=0.08]  [draw opacity=0] (13.4,-6.43) -- (0,0) -- (13.4,6.44) -- (8.9,0) -- cycle    ;
\draw [line width=1.5]    (151.07,142.88) -- (190.36,142.9) ;
\draw [shift={(194.36,142.91)}, rotate = 180.03] [fill={rgb, 255:red, 0; green, 0; blue, 0 }  ][line width=0.08]  [draw opacity=0] (13.4,-6.43) -- (0,0) -- (13.4,6.44) -- (8.9,0) -- cycle    ;
\draw [line width=1.5]    (205.02,148.77) -- (225.62,184.31) ;
\draw [shift={(227.62,187.77)}, rotate = 239.91] [fill={rgb, 255:red, 0; green, 0; blue, 0 }  ][line width=0.08]  [draw opacity=0] (13.4,-6.43) -- (0,0) -- (13.4,6.44) -- (8.9,0) -- cycle    ;
\draw [line width=1.5]    (205.69,136.56) -- (225.87,102.01) ;
\draw [shift={(227.89,98.56)}, rotate = 120.29] [fill={rgb, 255:red, 0; green, 0; blue, 0 }  ][line width=0.08]  [draw opacity=0] (13.4,-6.43) -- (0,0) -- (13.4,6.44) -- (8.9,0) -- cycle    ;
\draw  [color={rgb, 255:red, 0; green, 0; blue, 0 }  ,draw opacity=1 ][fill={rgb, 255:red, 0; green, 0; blue, 0 }  ,fill opacity=1 ] (143,143) .. controls (143,141.34) and (144.34,140) .. (146,140) .. controls (147.66,140) and (149,141.34) .. (149,143) .. controls (149,144.66) and (147.66,146) .. (146,146) .. controls (144.34,146) and (143,144.66) .. (143,143) -- cycle ;
\draw [line width=1.5]    (101.23,143.05) -- (136.23,143.2) ;
\draw [shift={(140.23,143.22)}, rotate = 180.24] [fill={rgb, 255:red, 0; green, 0; blue, 0 }  ][line width=0.08]  [draw opacity=0] (13.4,-6.43) -- (0,0) -- (13.4,6.44) -- (8.9,0) -- cycle    ;

\draw (527,133.07) node [anchor=north west][inner sep=0.75pt]    {$\sim _{\tr} \pm $};
\draw (-34,136.4) node [anchor=north west][inner sep=0.75pt]    {$Q\ =\ $};
\draw (248.13,136) node [anchor=north west][inner sep=0.75pt]   [align=left] {then};
\draw (-152.33,113.33) node [anchor=north west][inner sep=0.75pt]   [align=left] {\begin{minipage}[lt]{61.68pt}\setlength\topsep{0pt}
For instance 
\begin{center}
if we set
\end{center}
\end{minipage}};
\end{tikzpicture}
\end{center}
    
\begin{proof}
    Let $C = a_1 \cdots a_{\ell}$ be a cycle and let $v \in \Srep(Q,\dd)$ be a symmetric representation. Then:
    \begin{align*}
        v(C) & = \repinvolution(v)(C) = \big(\prod_{k=1}^ls(a_k)\big)(J_{\tau tC}^{-1}v_{\tau a_1}^{*} \cdots v_{\tau a_{\ell}}^*J_{\tau hC}) \\
        & = \big(\prod_{k=1}^ls(a_k)\big)(J_{\tau tC}^{-1}(v_{\tau a_{\ell}} \cdots v_{\tau a_1})^*J_{\tau tC}) \\
        & = \big(\prod_{k=1}^ls(a_k)\big)(J_{\tau tC}^{-1}(v(\tau a_{\ell} \cdots \tau a_1))^*J_{\tau tC})\\
        & = \big(\prod_{k=1}^ls(a_k)\big)(J_{\tau tC}^{-1}v(\tau (C))^*J_{\tau tC}), \\
        & \tr(v(C)) = \big(\prod_{k=1}^ls(a_k)\big)\tr(v(\tau (C))), \text{ and thus } C \sim_{\tr} \pm \tau (C). \qedhere
    \end{align*}
\end{proof}

We say that a cycle is a \emph{short loop} if it is of the form $aa^*$ for some $a \in \overline{Q}$.

\begin{rmq}\label{rmq: relation moment}
    The relations in $\Pi(Q)$ given by the formal moment map (from \Cref{defn:formal moment}) can be interpreted as stating that, for every vertex $i \in I$, the sum of short loops ending at $i$ equals the sum of short loops starting at $i$.
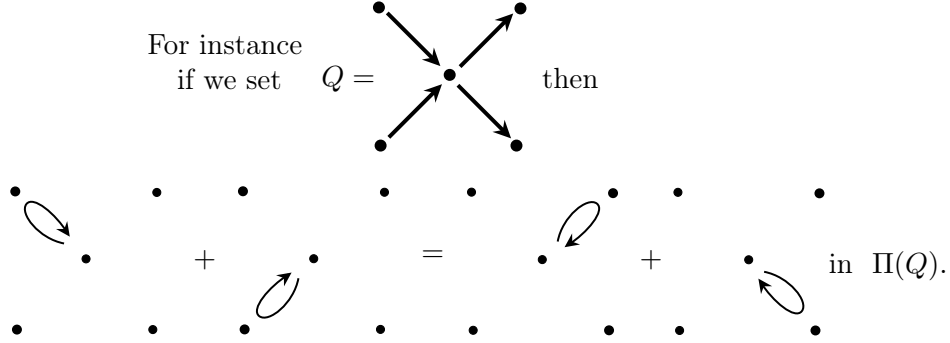
\begin{center}
\begin{tikzpicture}[x=0.5pt,y=0.5pt,yscale=-1,xscale=1]
\draw  [color={rgb, 255:red, 0; green, 0; blue, 0 }  ,draw opacity=1 ][fill={rgb, 255:red, 0; green, 0; blue, 0 }  ,fill opacity=1 ][line width=1.5]  (341.33,80.67) .. controls (341.33,79.01) and (342.68,77.67) .. (344.33,77.67) .. controls (345.99,77.67) and (347.33,79.01) .. (347.33,80.67) .. controls (347.33,82.32) and (345.99,83.67) .. (344.33,83.67) .. controls (342.68,83.67) and (341.33,82.32) .. (341.33,80.67) -- cycle ;
\draw  [color={rgb, 255:red, 0; green, 0; blue, 0 }  ,draw opacity=1 ][fill={rgb, 255:red, 0; green, 0; blue, 0 }  ,fill opacity=1 ][line width=1.5]  (289.33,133.87) .. controls (289.33,132.21) and (290.68,130.87) .. (292.33,130.87) .. controls (293.99,130.87) and (295.33,132.21) .. (295.33,133.87) .. controls (295.33,135.52) and (293.99,136.87) .. (292.33,136.87) .. controls (290.68,136.87) and (289.33,135.52) .. (289.33,133.87) -- cycle ;
\draw [line width=1.5]    (298.53,127.67) -- (335.7,90.5) ;
\draw [shift={(338.53,87.67)}, rotate = 135] [fill={rgb, 255:red, 0; green, 0; blue, 0 }  ][line width=0.08]  [draw opacity=0] (13.4,-6.43) -- (0,0) -- (13.4,6.44) -- (8.9,0) -- cycle    ;
\draw [line width=1.5]    (352.13,74.47) -- (389.3,37.3) ;
\draw [shift={(392.13,34.47)}, rotate = 135] [fill={rgb, 255:red, 0; green, 0; blue, 0 }  ][line width=0.08]  [draw opacity=0] (13.4,-6.43) -- (0,0) -- (13.4,6.44) -- (8.9,0) -- cycle    ;
\draw [line width=1.5]    (350.53,88.47) -- (387.7,125.64) ;
\draw [shift={(390.53,128.47)}, rotate = 225] [fill={rgb, 255:red, 0; green, 0; blue, 0 }  ][line width=0.08]  [draw opacity=0] (13.4,-6.43) -- (0,0) -- (13.4,6.44) -- (8.9,0) -- cycle    ;
\draw [line width=1.5]    (297.33,34.07) -- (334.5,71.24) ;
\draw [shift={(337.33,74.07)}, rotate = 225] [fill={rgb, 255:red, 0; green, 0; blue, 0 }  ][line width=0.08]  [draw opacity=0] (13.4,-6.43) -- (0,0) -- (13.4,6.44) -- (8.9,0) -- cycle    ;
\draw  [color={rgb, 255:red, 0; green, 0; blue, 0 }  ,draw opacity=1 ][fill={rgb, 255:red, 0; green, 0; blue, 0 }  ,fill opacity=1 ][line width=1.5]  (391.58,133.87) .. controls (391.58,132.21) and (392.93,130.87) .. (394.58,130.87) .. controls (396.24,130.87) and (397.58,132.21) .. (397.58,133.87) .. controls (397.58,135.52) and (396.24,136.87) .. (394.58,136.87) .. controls (392.93,136.87) and (391.58,135.52) .. (391.58,133.87) -- cycle ;
\draw  [color={rgb, 255:red, 0; green, 0; blue, 0 }  ,draw opacity=1 ][fill={rgb, 255:red, 0; green, 0; blue, 0 }  ,fill opacity=1 ][line width=1.5]  (288.08,29.92) .. controls (288.08,28.26) and (289.43,26.92) .. (291.08,26.92) .. controls (292.74,26.92) and (294.08,28.26) .. (294.08,29.92) .. controls (294.08,31.57) and (292.74,32.92) .. (291.08,32.92) .. controls (289.43,32.92) and (288.08,31.57) .. (288.08,29.92) -- cycle ;
\draw  [color={rgb, 255:red, 0; green, 0; blue, 0 }  ,draw opacity=1 ][fill={rgb, 255:red, 0; green, 0; blue, 0 }  ,fill opacity=1 ][line width=1.5]  (394.58,30.67) .. controls (394.58,29.01) and (395.93,27.67) .. (397.58,27.67) .. controls (399.24,27.67) and (400.58,29.01) .. (400.58,30.67) .. controls (400.58,32.32) and (399.24,33.67) .. (397.58,33.67) .. controls (395.93,33.67) and (394.58,32.32) .. (394.58,30.67) -- cycle ;
\draw  [color={rgb, 255:red, 0; green, 0; blue, 0 }  ,draw opacity=1 ][fill={rgb, 255:red, 0; green, 0; blue, 0 }  ,fill opacity=1 ] (239,219) .. controls (239,217.34) and (240.34,216) .. (242,216) .. controls (243.66,216) and (245,217.34) .. (245,219) .. controls (245,220.66) and (243.66,222) .. (242,222) .. controls (240.34,222) and (239,220.66) .. (239,219) -- cycle ;
\draw  [color={rgb, 255:red, 0; green, 0; blue, 0 }  ,draw opacity=1 ][fill={rgb, 255:red, 0; green, 0; blue, 0 }  ,fill opacity=1 ][line width=0.75]  (187,272.2) .. controls (187,270.54) and (188.34,269.2) .. (190,269.2) .. controls (191.66,269.2) and (193,270.54) .. (193,272.2) .. controls (193,273.86) and (191.66,275.2) .. (190,275.2) .. controls (188.34,275.2) and (187,273.86) .. (187,272.2) -- cycle ;
\draw  [color={rgb, 255:red, 0; green, 0; blue, 0 }  ,draw opacity=1 ][fill={rgb, 255:red, 0; green, 0; blue, 0 }  ,fill opacity=1 ] (289.25,272.2) .. controls (289.25,270.54) and (290.59,269.2) .. (292.25,269.2) .. controls (293.91,269.2) and (295.25,270.54) .. (295.25,272.2) .. controls (295.25,273.86) and (293.91,275.2) .. (292.25,275.2) .. controls (290.59,275.2) and (289.25,273.86) .. (289.25,272.2) -- cycle ;
\draw  [color={rgb, 255:red, 0; green, 0; blue, 0 }  ,draw opacity=1 ][fill={rgb, 255:red, 0; green, 0; blue, 0 }  ,fill opacity=1 ][line width=0.75]  (185.75,168.25) .. controls (185.75,166.59) and (187.09,165.25) .. (188.75,165.25) .. controls (190.41,165.25) and (191.75,166.59) .. (191.75,168.25) .. controls (191.75,169.91) and (190.41,171.25) .. (188.75,171.25) .. controls (187.09,171.25) and (185.75,169.91) .. (185.75,168.25) -- cycle ;
\draw  [color={rgb, 255:red, 0; green, 0; blue, 0 }  ,draw opacity=1 ][fill={rgb, 255:red, 0; green, 0; blue, 0 }  ,fill opacity=1 ] (292.25,169) .. controls (292.25,167.34) and (293.59,166) .. (295.25,166) .. controls (296.91,166) and (298.25,167.34) .. (298.25,169) .. controls (298.25,170.66) and (296.91,172) .. (295.25,172) .. controls (293.59,172) and (292.25,170.66) .. (292.25,169) -- cycle ;
\draw  [color={rgb, 255:red, 0; green, 0; blue, 0 }  ,draw opacity=1 ][fill={rgb, 255:red, 0; green, 0; blue, 0 }  ,fill opacity=1 ] (411,219.67) .. controls (411,218.01) and (412.34,216.67) .. (414,216.67) .. controls (415.66,216.67) and (417,218.01) .. (417,219.67) .. controls (417,221.32) and (415.66,222.67) .. (414,222.67) .. controls (412.34,222.67) and (411,221.32) .. (411,219.67) -- cycle ;
\draw  [color={rgb, 255:red, 0; green, 0; blue, 0 }  ,draw opacity=1 ][fill={rgb, 255:red, 0; green, 0; blue, 0 }  ,fill opacity=1 ] (359,272.87) .. controls (359,271.21) and (360.34,269.87) .. (362,269.87) .. controls (363.66,269.87) and (365,271.21) .. (365,272.87) .. controls (365,274.52) and (363.66,275.87) .. (362,275.87) .. controls (360.34,275.87) and (359,274.52) .. (359,272.87) -- cycle ;
\draw  [color={rgb, 255:red, 0; green, 0; blue, 0 }  ,draw opacity=1 ][fill={rgb, 255:red, 0; green, 0; blue, 0 }  ,fill opacity=1 ][line width=0.75]  (461.25,272.87) .. controls (461.25,271.21) and (462.59,269.87) .. (464.25,269.87) .. controls (465.91,269.87) and (467.25,271.21) .. (467.25,272.87) .. controls (467.25,274.52) and (465.91,275.87) .. (464.25,275.87) .. controls (462.59,275.87) and (461.25,274.52) .. (461.25,272.87) -- cycle ;
\draw  [color={rgb, 255:red, 0; green, 0; blue, 0 }  ,draw opacity=1 ][fill={rgb, 255:red, 0; green, 0; blue, 0 }  ,fill opacity=1 ] (357.75,168.92) .. controls (357.75,167.26) and (359.09,165.92) .. (360.75,165.92) .. controls (362.41,165.92) and (363.75,167.26) .. (363.75,168.92) .. controls (363.75,170.57) and (362.41,171.92) .. (360.75,171.92) .. controls (359.09,171.92) and (357.75,170.57) .. (357.75,168.92) -- cycle ;
\draw  [color={rgb, 255:red, 0; green, 0; blue, 0 }  ,draw opacity=1 ][fill={rgb, 255:red, 0; green, 0; blue, 0 }  ,fill opacity=1 ][line width=0.75]  (464.25,169.67) .. controls (464.25,168.01) and (465.59,166.67) .. (467.25,166.67) .. controls (468.91,166.67) and (470.25,168.01) .. (470.25,169.67) .. controls (470.25,171.32) and (468.91,172.67) .. (467.25,172.67) .. controls (465.59,172.67) and (464.25,171.32) .. (464.25,169.67) -- cycle ;
\draw [line width=0.75]    (230.5,233.7) .. controls (227.04,252.55) and (208.15,267.66) .. (201.02,262.22) .. controls (194.24,257.05) and (208.12,240.31) .. (223.01,230.84) ;
\draw [shift={(225.37,229.41)}, rotate = 144.5] [fill={rgb, 255:red, 0; green, 0; blue, 0 }  ][line width=0.08]  [draw opacity=0] (10.72,-5.15) -- (0,0) -- (10.72,5.15) -- (7.12,0) -- cycle    ;
\draw [line width=0.75]    (425.61,206.08) .. controls (429.04,187.22) and (447.9,172.08) .. (455.04,177.5) .. controls (461.82,182.66) and (447.98,199.42) .. (433.11,208.92) ;
\draw [shift={(430.75,210.36)}, rotate = 324.39] [fill={rgb, 255:red, 0; green, 0; blue, 0 }  ][line width=0.08]  [draw opacity=0] (10.72,-5.15) -- (0,0) -- (10.72,5.15) -- (7.12,0) -- cycle    ;
\draw  [color={rgb, 255:red, 0; green, 0; blue, 0 }  ,draw opacity=1 ][fill={rgb, 255:red, 0; green, 0; blue, 0 }  ,fill opacity=1 ] (67.67,219.13) .. controls (67.67,217.48) and (69.01,216.13) .. (70.67,216.13) .. controls (72.32,216.13) and (73.67,217.48) .. (73.67,219.13) .. controls (73.67,220.79) and (72.32,222.13) .. (70.67,222.13) .. controls (69.01,222.13) and (67.67,220.79) .. (67.67,219.13) -- cycle ;
\draw  [color={rgb, 255:red, 0; green, 0; blue, 0 }  ,draw opacity=1 ][fill={rgb, 255:red, 0; green, 0; blue, 0 }  ,fill opacity=1 ][line width=0.75]  (15.67,272.33) .. controls (15.67,270.68) and (17.01,269.33) .. (18.67,269.33) .. controls (20.32,269.33) and (21.67,270.68) .. (21.67,272.33) .. controls (21.67,273.99) and (20.32,275.33) .. (18.67,275.33) .. controls (17.01,275.33) and (15.67,273.99) .. (15.67,272.33) -- cycle ;
\draw  [color={rgb, 255:red, 0; green, 0; blue, 0 }  ,draw opacity=1 ][fill={rgb, 255:red, 0; green, 0; blue, 0 }  ,fill opacity=1 ] (117.92,272.33) .. controls (117.92,270.68) and (119.26,269.33) .. (120.92,269.33) .. controls (122.57,269.33) and (123.92,270.68) .. (123.92,272.33) .. controls (123.92,273.99) and (122.57,275.33) .. (120.92,275.33) .. controls (119.26,275.33) and (117.92,273.99) .. (117.92,272.33) -- cycle ;
\draw  [color={rgb, 255:red, 0; green, 0; blue, 0 }  ,draw opacity=1 ][fill={rgb, 255:red, 0; green, 0; blue, 0 }  ,fill opacity=1 ][line width=0.75]  (14.42,168.38) .. controls (14.42,166.73) and (15.76,165.38) .. (17.42,165.38) .. controls (19.07,165.38) and (20.42,166.73) .. (20.42,168.38) .. controls (20.42,170.04) and (19.07,171.38) .. (17.42,171.38) .. controls (15.76,171.38) and (14.42,170.04) .. (14.42,168.38) -- cycle ;
\draw  [color={rgb, 255:red, 0; green, 0; blue, 0 }  ,draw opacity=1 ][fill={rgb, 255:red, 0; green, 0; blue, 0 }  ,fill opacity=1 ] (120.92,169.13) .. controls (120.92,167.48) and (122.26,166.13) .. (123.92,166.13) .. controls (125.57,166.13) and (126.92,167.48) .. (126.92,169.13) .. controls (126.92,170.79) and (125.57,172.13) .. (123.92,172.13) .. controls (122.26,172.13) and (120.92,170.79) .. (120.92,169.13) -- cycle ;
\draw [line width=0.75]    (54.13,207.54) .. controls (35.3,203.97) and (20.29,185) .. (25.78,177.9) .. controls (30.98,171.15) and (47.64,185.12) .. (57.03,200.06) ;
\draw [shift={(58.45,202.42)}, rotate = 234.81] [fill={rgb, 255:red, 0; green, 0; blue, 0 }  ][line width=0.08]  [draw opacity=0] (10.72,-5.15) -- (0,0) -- (10.72,5.15) -- (7.12,0) -- cycle    ;
\draw  [color={rgb, 255:red, 0; green, 0; blue, 0 }  ,draw opacity=1 ][fill={rgb, 255:red, 0; green, 0; blue, 0 }  ,fill opacity=1 ] (566.33,219.8) .. controls (566.33,218.14) and (567.68,216.8) .. (569.33,216.8) .. controls (570.99,216.8) and (572.33,218.14) .. (572.33,219.8) .. controls (572.33,221.46) and (570.99,222.8) .. (569.33,222.8) .. controls (567.68,222.8) and (566.33,221.46) .. (566.33,219.8) -- cycle ;
\draw  [color={rgb, 255:red, 0; green, 0; blue, 0 }  ,draw opacity=1 ][fill={rgb, 255:red, 0; green, 0; blue, 0 }  ,fill opacity=1 ] (514.33,273) .. controls (514.33,271.34) and (515.68,270) .. (517.33,270) .. controls (518.99,270) and (520.33,271.34) .. (520.33,273) .. controls (520.33,274.66) and (518.99,276) .. (517.33,276) .. controls (515.68,276) and (514.33,274.66) .. (514.33,273) -- cycle ;
\draw  [color={rgb, 255:red, 0; green, 0; blue, 0 }  ,draw opacity=1 ][fill={rgb, 255:red, 0; green, 0; blue, 0 }  ,fill opacity=1 ][line width=0.75]  (616.58,273) .. controls (616.58,271.34) and (617.93,270) .. (619.58,270) .. controls (621.24,270) and (622.58,271.34) .. (622.58,273) .. controls (622.58,274.66) and (621.24,276) .. (619.58,276) .. controls (617.93,276) and (616.58,274.66) .. (616.58,273) -- cycle ;
\draw  [color={rgb, 255:red, 0; green, 0; blue, 0 }  ,draw opacity=1 ][fill={rgb, 255:red, 0; green, 0; blue, 0 }  ,fill opacity=1 ] (513.08,169.05) .. controls (513.08,167.39) and (514.43,166.05) .. (516.08,166.05) .. controls (517.74,166.05) and (519.08,167.39) .. (519.08,169.05) .. controls (519.08,170.71) and (517.74,172.05) .. (516.08,172.05) .. controls (514.43,172.05) and (513.08,170.71) .. (513.08,169.05) -- cycle ;
\draw  [color={rgb, 255:red, 0; green, 0; blue, 0 }  ,draw opacity=1 ][fill={rgb, 255:red, 0; green, 0; blue, 0 }  ,fill opacity=1 ][line width=0.75]  (619.58,169.8) .. controls (619.58,168.14) and (620.93,166.8) .. (622.58,166.8) .. controls (624.24,166.8) and (625.58,168.14) .. (625.58,169.8) .. controls (625.58,171.46) and (624.24,172.8) .. (622.58,172.8) .. controls (620.93,172.8) and (619.58,171.46) .. (619.58,169.8) -- cycle ;
\draw [line width=0.75]    (580.93,228.84) .. controls (599.77,232.37) and (614.81,251.31) .. (609.34,258.42) .. controls (604.15,265.18) and (587.46,251.25) .. (578.04,236.32) ;
\draw [shift={(576.61,233.96)}, rotate = 54.7] [fill={rgb, 255:red, 0; green, 0; blue, 0 }  ][line width=0.08]  [draw opacity=0] (10.72,-5.15) -- (0,0) -- (10.72,5.15) -- (7.12,0) -- cycle    ;

\draw (246,72.4) node [anchor=north west][inner sep=0.75pt]    {$Q=$};
\draw (321.33,211.07) node [anchor=north west][inner sep=0.75pt]    {$=$};
\draw (627.67,211.07) node [anchor=north west][inner sep=0.75pt]    {$\text{in } \ \Pi ( Q) .$};
\draw (149.6,210.33) node [anchor=north west][inner sep=0.75pt]    {$+$};
\draw (485.6,210.27) node [anchor=north west][inner sep=0.75pt]    {$+$};
\draw (412.12,73.93) node [anchor=north west][inner sep=0.75pt]   [align=left] {then};
\draw (115,48) node [anchor=north west][inner sep=0.75pt]   [align=left] {\begin{minipage}[lt]{61.68pt}\setlength\topsep{0pt}
For instance 
\begin{center}
if we set
\end{center}
\end{minipage}};
\end{tikzpicture}
\end{center}
\end{rmq}

\begin{lem}\label{lem: sliding short loops}
    Let $p = a_1 \cdots a_{\ell}$ be a path in $\overline{Q}$ such that, for each $a_k$ with $1\leq k < l$, there are exactly two distinct arrows in $Q$ incident to the vertex $ta_k=ha_{k+1} \in I$. Then:
    $$a_1a_1^*p = \pm pa_{\ell}^*a_{\ell} \in \Pi(Q).$$
\end{lem}
\begin{center}
\begin{tikzpicture}[x=0.5pt,y=0.5pt,yscale=-1,xscale=1]

\draw [line width=0.75]    (725.1,138.16) .. controls (713.1,132.45) and (674.21,125.89) .. (674.21,141.22) .. controls (674.21,156.55) and (723.95,154.07) .. (723.95,146.73) .. controls (723.95,139.4) and (690.02,141.36) .. (674.21,141.22) .. controls (658.63,141.08) and (549.28,140.92) .. (494.21,140.93) ;
\draw [shift={(491.73,140.93)}, rotate = 359.98] [fill={rgb, 255:red, 0; green, 0; blue, 0 }  ][line width=0.08]  [draw opacity=0] (10.72,-5.15) -- (0,0) -- (10.72,5.15) -- (7.12,0) -- cycle    ;
\draw  [color={rgb, 255:red, 0; green, 0; blue, 0 }  ,draw opacity=1 ][fill={rgb, 255:red, 0; green, 0; blue, 0 }  ,fill opacity=1 ][line width=0.75]  (671.21,141.22) .. controls (671.21,139.56) and (672.55,138.22) .. (674.21,138.22) .. controls (675.87,138.22) and (677.21,139.56) .. (677.21,141.22) .. controls (677.21,142.88) and (675.87,144.22) .. (674.21,144.22) .. controls (672.55,144.22) and (671.21,142.88) .. (671.21,141.22) -- cycle ;
\draw  [color={rgb, 255:red, 0; green, 0; blue, 0 }  ,draw opacity=1 ][fill={rgb, 255:red, 0; green, 0; blue, 0 }  ,fill opacity=1 ][line width=0.75]  (730.56,141.06) .. controls (730.56,139.4) and (731.91,138.06) .. (733.56,138.06) .. controls (735.22,138.06) and (736.56,139.4) .. (736.56,141.06) .. controls (736.56,142.71) and (735.22,144.06) .. (733.56,144.06) .. controls (731.91,144.06) and (730.56,142.71) .. (730.56,141.06) -- cycle ;
\draw  [color={rgb, 255:red, 0; green, 0; blue, 0 }  ,draw opacity=1 ][fill={rgb, 255:red, 0; green, 0; blue, 0 }  ,fill opacity=1 ][line width=0.75]  (482.11,141.63) .. controls (482.11,139.97) and (483.45,138.63) .. (485.11,138.63) .. controls (486.76,138.63) and (488.11,139.97) .. (488.11,141.63) .. controls (488.11,143.28) and (486.76,144.63) .. (485.11,144.63) .. controls (483.45,144.63) and (482.11,143.28) .. (482.11,141.63) -- cycle ;
\draw  [color={rgb, 255:red, 0; green, 0; blue, 0 }  ,draw opacity=1 ][fill={rgb, 255:red, 0; green, 0; blue, 0 }  ,fill opacity=1 ][line width=0.75]  (530.11,141.23) .. controls (530.11,139.57) and (531.45,138.23) .. (533.11,138.23) .. controls (534.76,138.23) and (536.11,139.57) .. (536.11,141.23) .. controls (536.11,142.88) and (534.76,144.23) .. (533.11,144.23) .. controls (531.45,144.23) and (530.11,142.88) .. (530.11,141.23) -- cycle ;
\draw  [color={rgb, 255:red, 0; green, 0; blue, 0 }  ,draw opacity=1 ][fill={rgb, 255:red, 0; green, 0; blue, 0 }  ,fill opacity=1 ][line width=0.75]  (579.71,141.23) .. controls (579.71,139.57) and (581.05,138.23) .. (582.71,138.23) .. controls (584.36,138.23) and (585.71,139.57) .. (585.71,141.23) .. controls (585.71,142.88) and (584.36,144.23) .. (582.71,144.23) .. controls (581.05,144.23) and (579.71,142.88) .. (579.71,141.23) -- cycle ;
\draw [line width=0.75]    (436.12,141.4) .. controls (419.72,141.4) and (320.92,141.09) .. (305.64,140.96) .. controls (290.36,140.83) and (255.68,123.2) .. (256.04,140.96) .. controls (256.4,158.72) and (306.48,161.6) .. (306.48,154) .. controls (306.48,146.4) and (274.08,140.8) .. (256.04,140.96) .. controls (238.63,141.11) and (241.53,140.69) .. (217.4,140.67) ;
\draw [shift={(214.67,140.67)}, rotate = 359.98] [fill={rgb, 255:red, 0; green, 0; blue, 0 }  ][line width=0.08]  [draw opacity=0] (10.72,-5.15) -- (0,0) -- (10.72,5.15) -- (7.12,0) -- cycle    ;
\draw  [color={rgb, 255:red, 0; green, 0; blue, 0 }  ,draw opacity=1 ][fill={rgb, 255:red, 0; green, 0; blue, 0 }  ,fill opacity=1 ][line width=0.75]  (205.04,141.36) .. controls (205.04,139.7) and (206.38,138.36) .. (208.04,138.36) .. controls (209.7,138.36) and (211.04,139.7) .. (211.04,141.36) .. controls (211.04,143.02) and (209.7,144.36) .. (208.04,144.36) .. controls (206.38,144.36) and (205.04,143.02) .. (205.04,141.36) -- cycle ;
\draw  [color={rgb, 255:red, 0; green, 0; blue, 0 }  ,draw opacity=1 ][fill={rgb, 255:red, 0; green, 0; blue, 0 }  ,fill opacity=1 ][line width=0.75]  (253.04,140.96) .. controls (253.04,139.3) and (254.38,137.96) .. (256.04,137.96) .. controls (257.7,137.96) and (259.04,139.3) .. (259.04,140.96) .. controls (259.04,142.62) and (257.7,143.96) .. (256.04,143.96) .. controls (254.38,143.96) and (253.04,142.62) .. (253.04,140.96) -- cycle ;
\draw  [color={rgb, 255:red, 0; green, 0; blue, 0 }  ,draw opacity=1 ][fill={rgb, 255:red, 0; green, 0; blue, 0 }  ,fill opacity=1 ][line width=0.75]  (302.64,140.96) .. controls (302.64,139.3) and (303.98,137.96) .. (305.64,137.96) .. controls (307.3,137.96) and (308.64,139.3) .. (308.64,140.96) .. controls (308.64,142.62) and (307.3,143.96) .. (305.64,143.96) .. controls (303.98,143.96) and (302.64,142.62) .. (302.64,140.96) -- cycle ;
\draw  [color={rgb, 255:red, 0; green, 0; blue, 0 }  ,draw opacity=1 ][fill={rgb, 255:red, 0; green, 0; blue, 0 }  ,fill opacity=1 ][line width=0.75]  (393,141.07) .. controls (393,139.41) and (394.34,138.07) .. (396,138.07) .. controls (397.66,138.07) and (399,139.41) .. (399,141.07) .. controls (399,142.72) and (397.66,144.07) .. (396,144.07) .. controls (394.34,144.07) and (393,142.72) .. (393,141.07) -- cycle ;
\draw  [color={rgb, 255:red, 0; green, 0; blue, 0 }  ,draw opacity=1 ][fill={rgb, 255:red, 0; green, 0; blue, 0 }  ,fill opacity=1 ][line width=0.75]  (440.64,141.36) .. controls (440.64,139.7) and (441.98,138.36) .. (443.64,138.36) .. controls (445.3,138.36) and (446.64,139.7) .. (446.64,141.36) .. controls (446.64,143.02) and (445.3,144.36) .. (443.64,144.36) .. controls (441.98,144.36) and (440.64,143.02) .. (440.64,141.36) -- cycle ;
\draw [line width=0.75]    (168.95,141) .. controls (152.55,141) and (73.52,140.83) .. (58.24,140.69) .. controls (42.96,140.56) and (-41.32,140.53) .. (-40.12,150.53) .. controls (-38.92,160.53) and (10.6,152.45) .. (8.64,140.69) .. controls (6.78,129.52) and (-19.99,132.79) .. (-38.13,133.28) ;
\draw [shift={(-40.92,133.33)}, rotate = 359.49] [fill={rgb, 255:red, 0; green, 0; blue, 0 }  ][line width=0.08]  [draw opacity=0] (10.72,-5.15) -- (0,0) -- (10.72,5.15) -- (7.12,0) -- cycle    ;
\draw  [color={rgb, 255:red, 0; green, 0; blue, 0 }  ,draw opacity=1 ][fill={rgb, 255:red, 0; green, 0; blue, 0 }  ,fill opacity=1 ][line width=0.75]  (-42.36,141.09) .. controls (-42.36,139.44) and (-41.02,138.09) .. (-39.36,138.09) .. controls (-37.7,138.09) and (-36.36,139.44) .. (-36.36,141.09) .. controls (-36.36,142.75) and (-37.7,144.09) .. (-39.36,144.09) .. controls (-41.02,144.09) and (-42.36,142.75) .. (-42.36,141.09) -- cycle ;
\draw  [color={rgb, 255:red, 0; green, 0; blue, 0 }  ,draw opacity=1 ][fill={rgb, 255:red, 0; green, 0; blue, 0 }  ,fill opacity=1 ][line width=0.75]  (5.64,140.69) .. controls (5.64,139.04) and (6.98,137.69) .. (8.64,137.69) .. controls (10.3,137.69) and (11.64,139.04) .. (11.64,140.69) .. controls (11.64,142.35) and (10.3,143.69) .. (8.64,143.69) .. controls (6.98,143.69) and (5.64,142.35) .. (5.64,140.69) -- cycle ;
\draw  [color={rgb, 255:red, 0; green, 0; blue, 0 }  ,draw opacity=1 ][fill={rgb, 255:red, 0; green, 0; blue, 0 }  ,fill opacity=1 ][line width=0.75]  (55.24,140.69) .. controls (55.24,139.04) and (56.58,137.69) .. (58.24,137.69) .. controls (59.9,137.69) and (61.24,139.04) .. (61.24,140.69) .. controls (61.24,142.35) and (59.9,143.69) .. (58.24,143.69) .. controls (56.58,143.69) and (55.24,142.35) .. (55.24,140.69) -- cycle ;
\draw  [color={rgb, 255:red, 0; green, 0; blue, 0 }  ,draw opacity=1 ][fill={rgb, 255:red, 0; green, 0; blue, 0 }  ,fill opacity=1 ][line width=0.75]  (133.6,140.8) .. controls (133.6,139.14) and (134.94,137.8) .. (136.6,137.8) .. controls (138.26,137.8) and (139.6,139.14) .. (139.6,140.8) .. controls (139.6,142.46) and (138.26,143.8) .. (136.6,143.8) .. controls (134.94,143.8) and (133.6,142.46) .. (133.6,140.8) -- cycle ;
\draw  [color={rgb, 255:red, 0; green, 0; blue, 0 }  ,draw opacity=1 ][fill={rgb, 255:red, 0; green, 0; blue, 0 }  ,fill opacity=1 ][line width=0.75]  (172.44,141.09) .. controls (172.44,139.44) and (173.78,138.09) .. (175.44,138.09) .. controls (177.1,138.09) and (178.44,139.44) .. (178.44,141.09) .. controls (178.44,142.75) and (177.1,144.09) .. (175.44,144.09) .. controls (173.78,144.09) and (172.44,142.75) .. (172.44,141.09) -- cycle ;
\draw  [color={rgb, 255:red, 0; green, 0; blue, 0 }  ,draw opacity=1 ][fill={rgb, 255:red, 0; green, 0; blue, 0 }  ,fill opacity=1 ][line width=0.75]  (276.93,50.43) .. controls (276.93,48.77) and (278.27,47.43) .. (279.93,47.43) .. controls (281.58,47.43) and (282.93,48.77) .. (282.93,50.43) .. controls (282.93,52.08) and (281.58,53.43) .. (279.93,53.43) .. controls (278.27,53.43) and (276.93,52.08) .. (276.93,50.43) -- cycle ;
\draw  [color={rgb, 255:red, 0; green, 0; blue, 0 }  ,draw opacity=1 ][fill={rgb, 255:red, 0; green, 0; blue, 0 }  ,fill opacity=1 ][line width=0.75]  (324.93,50.03) .. controls (324.93,48.37) and (326.27,47.03) .. (327.93,47.03) .. controls (329.58,47.03) and (330.93,48.37) .. (330.93,50.03) .. controls (330.93,51.68) and (329.58,53.03) .. (327.93,53.03) .. controls (326.27,53.03) and (324.93,51.68) .. (324.93,50.03) -- cycle ;
\draw  [color={rgb, 255:red, 0; green, 0; blue, 0 }  ,draw opacity=1 ][fill={rgb, 255:red, 0; green, 0; blue, 0 }  ,fill opacity=1 ][line width=0.75]  (374.53,50.03) .. controls (374.53,48.37) and (375.87,47.03) .. (377.53,47.03) .. controls (379.18,47.03) and (380.53,48.37) .. (380.53,50.03) .. controls (380.53,51.68) and (379.18,53.03) .. (377.53,53.03) .. controls (375.87,53.03) and (374.53,51.68) .. (374.53,50.03) -- cycle ;
\draw  [color={rgb, 255:red, 0; green, 0; blue, 0 }  ,draw opacity=1 ][fill={rgb, 255:red, 0; green, 0; blue, 0 }  ,fill opacity=1 ][line width=0.75]  (424.49,49.73) .. controls (424.49,48.08) and (425.83,46.73) .. (427.49,46.73) .. controls (429.14,46.73) and (430.49,48.08) .. (430.49,49.73) .. controls (430.49,51.39) and (429.14,52.73) .. (427.49,52.73) .. controls (425.83,52.73) and (424.49,51.39) .. (424.49,49.73) -- cycle ;
\draw [line width=1.5]    (286.49,50.23) -- (315.82,50.23) ;
\draw [shift={(319.82,50.23)}, rotate = 180] [fill={rgb, 255:red, 0; green, 0; blue, 0 }  ][line width=0.08]  [draw opacity=0] (13.4,-6.43) -- (0,0) -- (13.4,6.44) -- (8.9,0) -- cycle    ;
\draw [line width=1.5]    (335.49,50.23) -- (364.82,50.23) ;
\draw [shift={(368.82,50.23)}, rotate = 180] [fill={rgb, 255:red, 0; green, 0; blue, 0 }  ][line width=0.08]  [draw opacity=0] (13.4,-6.43) -- (0,0) -- (13.4,6.44) -- (8.9,0) -- cycle    ;
\draw [line width=1.5]    (385.49,50.23) -- (414.82,50.23) ;
\draw [shift={(418.82,50.23)}, rotate = 180] [fill={rgb, 255:red, 0; green, 0; blue, 0 }  ][line width=0.08]  [draw opacity=0] (13.4,-6.43) -- (0,0) -- (13.4,6.44) -- (8.9,0) -- cycle    ;
\draw  [color={rgb, 255:red, 0; green, 0; blue, 0 }  ,draw opacity=1 ][fill={rgb, 255:red, 0; green, 0; blue, 0 }  ,fill opacity=1 ][line width=0.75]  (477.06,49.45) .. controls (477.06,47.63) and (478.4,46.16) .. (480.06,46.16) .. controls (481.71,46.16) and (483.06,47.63) .. (483.06,49.45) .. controls (483.06,51.26) and (481.71,52.73) .. (480.06,52.73) .. controls (478.4,52.73) and (477.06,51.26) .. (477.06,49.45) -- cycle ;
\draw [line width=1.5]    (438.06,49.99) -- (467.39,49.99) ;
\draw [shift={(471.39,49.99)}, rotate = 180] [fill={rgb, 255:red, 0; green, 0; blue, 0 }  ][line width=0.08]  [draw opacity=0] (13.4,-6.43) -- (0,0) -- (13.4,6.44) -- (8.9,0) -- cycle    ;
\draw  [color={rgb, 255:red, 0; green, 0; blue, 0 }  ,draw opacity=1 ][fill={rgb, 255:red, 0; green, 0; blue, 0 }  ,fill opacity=1 ][line width=0.75]  (529.06,49.45) .. controls (529.06,47.63) and (530.4,46.16) .. (532.06,46.16) .. controls (533.71,46.16) and (535.06,47.63) .. (535.06,49.45) .. controls (535.06,51.26) and (533.71,52.73) .. (532.06,52.73) .. controls (530.4,52.73) and (529.06,51.26) .. (529.06,49.45) -- cycle ;
\draw [line width=1.5]    (490.06,49.99) -- (519.39,49.99) ;
\draw [shift={(523.39,49.99)}, rotate = 180] [fill={rgb, 255:red, 0; green, 0; blue, 0 }  ][line width=0.08]  [draw opacity=0] (13.4,-6.43) -- (0,0) -- (13.4,6.44) -- (8.9,0) -- cycle    ;

\draw (510,118.07) node [anchor=north west][inner sep=0.75pt]    {$a_{1}$};
\draw (554.11,119.51) node [anchor=north west][inner sep=0.75pt]    {$a_{2}$};
\draw (232.93,117.8) node [anchor=north west][inner sep=0.75pt]    {$a_{1}$};
\draw (419.47,118.4) node [anchor=north west][inner sep=0.75pt]    {$a_{l}$};
\draw (275.87,160) node [anchor=north west][inner sep=0.75pt]    {$a_{2}^{*}$};
\draw (277.04,119.24) node [anchor=north west][inner sep=0.75pt]    {$a_{2}$};
\draw (-22.33,107.27) node [anchor=north west][inner sep=0.75pt]    {$a_{1}$};
\draw (151.27,118.93) node [anchor=north west][inner sep=0.75pt]    {$a_{l}$};
\draw (-22.33,159.33) node [anchor=north west][inner sep=0.75pt]    {$a_{1}^{*}$};
\draw (29.64,118.97) node [anchor=north west][inner sep=0.75pt]    {$a_{2}$};
\draw (181.24,137.11) node [anchor=north west][inner sep=0.75pt]    {$=$};
\draw (453.67,139.07) node [anchor=north west][inner sep=0.75pt]    {$=$};
\draw (693.33,110.78) node [anchor=north west][inner sep=0.75pt]    {$a_{l}$};
\draw (693.9,152.5) node [anchor=north west][inner sep=0.75pt]    {$a_{l}^{*}$};
\draw (218.75,39.47) node [anchor=north west][inner sep=0.75pt]    {$Q\ =$};
\draw (77,15) node [anchor=north west][inner sep=0.75pt]   [align=left] {\begin{minipage}[lt]{61.68pt}\setlength\topsep{0pt}
For instance 
\begin{center}
if we set
\end{center}

\end{minipage}};
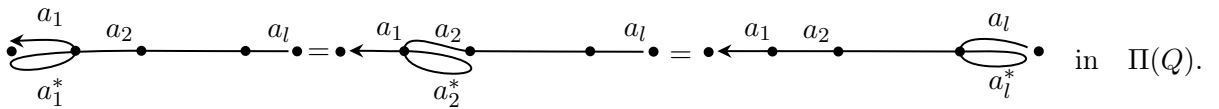
\draw (758.67,136.07) node [anchor=north west][inner sep=0.75pt]    {$\text{in } \ \ \Pi ( Q) .$};
\draw (584,39) node [anchor=north west][inner sep=0.75pt]   [align=left] {then };
\end{tikzpicture}
\end{center}

\begin{proof}
    We proceed by induction on $\ell$, if $\ell=1$, the statement is trivial.
    Now assume that for $p' = a_1 \cdots a_{\ell}$ we have that $p'a_{\ell}^*a_{\ell} = \pm a_1a_1^*p'$, and set $p = a_1 \cdots a_la_{\ell+1}$.
    If $a_{\ell} =a_{\ell+1}^*$ ( i.e.\ $a_{\ell}^* = a_{\ell+1}$) then:
    \begin{align*}
        pa_{\ell+1}^*a_{\ell+1} & = a_1 \cdots a_{\ell} a_{\ell+1}a_{\ell+1}^*a_{\ell+1} = a_1 \cdots a_{\ell} a_{\ell}^*a_{\ell}a_{\ell+1} = p'a_{\ell}^*a_{\ell}a_{\ell+1}  = \pm a_1a_1^*p'a_{\ell+1}  = \pm a_1a_1^*p.
    \end{align*}
    In the remaining case, i.e. when $a_{\ell} \neq a_{\ell+1}^*$, the intersection
    $$Q \cap \{a_{\ell},a_{\ell}^*,a_{\ell+1},a_{\ell+1}^*\}$$
    contains exactly two elements by assumption. 
    Moreover, the cases $\{a_{\ell},a_{\ell}^*\}$ and $\{a_{\ell+1},a_{\ell+1}^*\}$ cannot occur.
    In the four remaining cases, the relation $\mu_{ta_\ell}$ implies that $a_{\ell}^*a_{\ell} = \pm a_{\ell+1}a_{\ell+1}^* \in \Pi(Q).$
    We apply the previous calculation, up to a sign.
    Note that the sign changes in the induction step only when the two arrows in $Q$ have opposite orientations.
\end{proof}

\begin{rmq}\label{rmq: prepro trace}
    If $C,C'$ are two cycles in $\overline{Q}$ whose classes in $\Pi(Q)$ are equal, then the corresponding morphisms of affine schemes
    $$
    \begin{aligned}
        C \colon \Rep(\overline{Q},V)&\longrightarrow\ggot_V,
        &v&\mapsto v(C),\\
        C' \colon \Rep(\overline{Q},V)&\longrightarrow\ggot_V,
        &v&\mapsto v(C'),
    \end{aligned}
    $$
    have the same restriction to $\mu^{-1}(0)$. Composing with the trace morphism $\ggot_V\overset{\tr}{\longrightarrow}\CC$, we obtain $\tr(C)=\tr(C')\in\CC[\mu^{-1}(0)]$, and hence $C\sim_{\tr}C'$.
    Therefore,
    $$C=C'\text{ in }\Pi(Q) \quad\Longrightarrow\quad C\sim_{\tr}C'.$$
    
    Trace-equivalence relations, as well as the relations obtained from the \Cref{Hamilton-Cayley scheme morphism}, depend on the choice of a dimension vector and do not, in general, define equalities in $\Pi(Q)$.
    Therefore, in what follows, we first establish equalities in $\Pi(Q)$,
    such as the one in \Cref{lem: sliding short loops}, from which we can deduce trace-equivalence relations using \Cref{bijection alg prepro}.
    After fixing a dimension vector, we then consider the additional trace-equivalence relations obtained from the \Cref{Hamilton-Cayley scheme morphism}.
\end{rmq}

\begin{lem}\label{lem: tlp short loop}
    Let $p = a_1 \cdots a_{\ell}$ be a path in $\overline{Q}$ such that, for each $a_k$ with $1\leq k < l$, there are exactly two distinct arrows in $Q$ incident to the vertex $ta_k=ha_{k+1}$. Then $(a_1a_1^*)^n \sim_{\tr} \pm(a_{\ell}^*a_{\ell})^n$, for all $n \in \NN$.
\end{lem}
\begin{center}
\begin{tikzpicture}[x=0.5pt,y=0.5pt,yscale=-1,xscale=1]

\draw [line width=0.75]    (135.5,95.1) .. controls (154.5,100.6) and (174.7,104.6) .. (176.2,97.6) .. controls (177.7,90.6) and (131.5,82.1) .. (141.5,90.6) .. controls (151.5,99.1) and (176.7,86.6) .. (176.2,80.1) .. controls (175.73,73.96) and (153.82,76.75) .. (137.32,79.6) ;
\draw [shift={(134.5,80.1)}, rotate = 349.82] [fill={rgb, 255:red, 0; green, 0; blue, 0 }  ][line width=0.08]  [draw opacity=0] (10.72,-5.15) -- (0,0) -- (10.72,5.15) -- (7.12,0) -- cycle    ;
\draw [line width=0.75]    (620.41,82.35) .. controls (601.41,76.85) and (581.21,72.85) .. (579.71,79.85) .. controls (578.21,86.85) and (624.41,95.35) .. (614.41,86.85) .. controls (604.41,78.35) and (579.21,90.85) .. (579.71,97.35) .. controls (580.18,103.5) and (602.08,100.71) .. (618.58,97.85) ;
\draw [shift={(621.41,97.35)}, rotate = 169.82] [fill={rgb, 255:red, 0; green, 0; blue, 0 }  ][line width=0.08]  [draw opacity=0] (10.72,-5.15) -- (0,0) -- (10.72,5.15) -- (7.12,0) -- cycle    ;
\draw  [color={rgb, 255:red, 0; green, 0; blue, 0 }  ,draw opacity=1 ][fill={rgb, 255:red, 0; green, 0; blue, 0 }  ,fill opacity=1 ][line width=0.75]  (230.93,33.43) .. controls (230.93,31.77) and (232.27,30.43) .. (233.93,30.43) .. controls (235.58,30.43) and (236.93,31.77) .. (236.93,33.43) .. controls (236.93,35.08) and (235.58,36.43) .. (233.93,36.43) .. controls (232.27,36.43) and (230.93,35.08) .. (230.93,33.43) -- cycle ;
\draw  [color={rgb, 255:red, 0; green, 0; blue, 0 }  ,draw opacity=1 ][fill={rgb, 255:red, 0; green, 0; blue, 0 }  ,fill opacity=1 ][line width=0.75]  (278.93,33.03) .. controls (278.93,31.37) and (280.27,30.03) .. (281.93,30.03) .. controls (283.58,30.03) and (284.93,31.37) .. (284.93,33.03) .. controls (284.93,34.68) and (283.58,36.03) .. (281.93,36.03) .. controls (280.27,36.03) and (278.93,34.68) .. (278.93,33.03) -- cycle ;
\draw  [color={rgb, 255:red, 0; green, 0; blue, 0 }  ,draw opacity=1 ][fill={rgb, 255:red, 0; green, 0; blue, 0 }  ,fill opacity=1 ][line width=0.75]  (328.53,33.03) .. controls (328.53,31.37) and (329.87,30.03) .. (331.53,30.03) .. controls (333.18,30.03) and (334.53,31.37) .. (334.53,33.03) .. controls (334.53,34.68) and (333.18,36.03) .. (331.53,36.03) .. controls (329.87,36.03) and (328.53,34.68) .. (328.53,33.03) -- cycle ;
\draw  [color={rgb, 255:red, 0; green, 0; blue, 0 }  ,draw opacity=1 ][fill={rgb, 255:red, 0; green, 0; blue, 0 }  ,fill opacity=1 ][line width=0.75]  (378.49,32.73) .. controls (378.49,31.08) and (379.83,29.73) .. (381.49,29.73) .. controls (383.14,29.73) and (384.49,31.08) .. (384.49,32.73) .. controls (384.49,34.39) and (383.14,35.73) .. (381.49,35.73) .. controls (379.83,35.73) and (378.49,34.39) .. (378.49,32.73) -- cycle ;
\draw [line width=1.5]    (240.49,33.23) -- (269.82,33.23) ;
\draw [shift={(273.82,33.23)}, rotate = 180] [fill={rgb, 255:red, 0; green, 0; blue, 0 }  ][line width=0.08]  [draw opacity=0] (13.4,-6.43) -- (0,0) -- (13.4,6.44) -- (8.9,0) -- cycle    ;
\draw [line width=1.5]    (289.49,33.23) -- (318.82,33.23) ;
\draw [shift={(322.82,33.23)}, rotate = 180] [fill={rgb, 255:red, 0; green, 0; blue, 0 }  ][line width=0.08]  [draw opacity=0] (13.4,-6.43) -- (0,0) -- (13.4,6.44) -- (8.9,0) -- cycle    ;
\draw [line width=1.5]    (339.49,33.23) -- (368.82,33.23) ;
\draw [shift={(372.82,33.23)}, rotate = 180] [fill={rgb, 255:red, 0; green, 0; blue, 0 }  ][line width=0.08]  [draw opacity=0] (13.4,-6.43) -- (0,0) -- (13.4,6.44) -- (8.9,0) -- cycle    ;
\draw  [color={rgb, 255:red, 0; green, 0; blue, 0 }  ,draw opacity=1 ][fill={rgb, 255:red, 0; green, 0; blue, 0 }  ,fill opacity=1 ][line width=0.75]  (431.06,32.45) .. controls (431.06,30.63) and (432.4,29.16) .. (434.06,29.16) .. controls (435.71,29.16) and (437.06,30.63) .. (437.06,32.45) .. controls (437.06,34.26) and (435.71,35.73) .. (434.06,35.73) .. controls (432.4,35.73) and (431.06,34.26) .. (431.06,32.45) -- cycle ;
\draw [line width=1.5]    (392.06,32.99) -- (421.39,32.99) ;
\draw [shift={(425.39,32.99)}, rotate = 180] [fill={rgb, 255:red, 0; green, 0; blue, 0 }  ][line width=0.08]  [draw opacity=0] (13.4,-6.43) -- (0,0) -- (13.4,6.44) -- (8.9,0) -- cycle    ;
\draw  [color={rgb, 255:red, 0; green, 0; blue, 0 }  ,draw opacity=1 ][fill={rgb, 255:red, 0; green, 0; blue, 0 }  ,fill opacity=1 ][line width=0.75]  (483.06,32.45) .. controls (483.06,30.63) and (484.4,29.16) .. (486.06,29.16) .. controls (487.71,29.16) and (489.06,30.63) .. (489.06,32.45) .. controls (489.06,34.26) and (487.71,35.73) .. (486.06,35.73) .. controls (484.4,35.73) and (483.06,34.26) .. (483.06,32.45) -- cycle ;
\draw [line width=1.5]    (444.06,32.99) -- (473.39,32.99) ;
\draw [shift={(477.39,32.99)}, rotate = 180] [fill={rgb, 255:red, 0; green, 0; blue, 0 }  ][line width=0.08]  [draw opacity=0] (13.4,-6.43) -- (0,0) -- (13.4,6.44) -- (8.9,0) -- cycle    ;
\draw  [color={rgb, 255:red, 0; green, 0; blue, 0 }  ,draw opacity=1 ][fill={rgb, 255:red, 0; green, 0; blue, 0 }  ,fill opacity=1 ][line width=0.75]  (80.93,88.43) .. controls (80.93,86.77) and (82.27,85.43) .. (83.93,85.43) .. controls (85.58,85.43) and (86.93,86.77) .. (86.93,88.43) .. controls (86.93,90.08) and (85.58,91.43) .. (83.93,91.43) .. controls (82.27,91.43) and (80.93,90.08) .. (80.93,88.43) -- cycle ;
\draw  [color={rgb, 255:red, 0; green, 0; blue, 0 }  ,draw opacity=1 ][fill={rgb, 255:red, 0; green, 0; blue, 0 }  ,fill opacity=1 ][line width=0.75]  (128.93,88.03) .. controls (128.93,86.37) and (130.27,85.03) .. (131.93,85.03) .. controls (133.58,85.03) and (134.93,86.37) .. (134.93,88.03) .. controls (134.93,89.68) and (133.58,91.03) .. (131.93,91.03) .. controls (130.27,91.03) and (128.93,89.68) .. (128.93,88.03) -- cycle ;
\draw  [color={rgb, 255:red, 0; green, 0; blue, 0 }  ,draw opacity=1 ][fill={rgb, 255:red, 0; green, 0; blue, 0 }  ,fill opacity=1 ][line width=0.75]  (178.53,88.03) .. controls (178.53,86.37) and (179.87,85.03) .. (181.53,85.03) .. controls (183.18,85.03) and (184.53,86.37) .. (184.53,88.03) .. controls (184.53,89.68) and (183.18,91.03) .. (181.53,91.03) .. controls (179.87,91.03) and (178.53,89.68) .. (178.53,88.03) -- cycle ;
\draw  [color={rgb, 255:red, 0; green, 0; blue, 0 }  ,draw opacity=1 ][fill={rgb, 255:red, 0; green, 0; blue, 0 }  ,fill opacity=1 ][line width=0.75]  (228.49,87.73) .. controls (228.49,86.08) and (229.83,84.73) .. (231.49,84.73) .. controls (233.14,84.73) and (234.49,86.08) .. (234.49,87.73) .. controls (234.49,89.39) and (233.14,90.73) .. (231.49,90.73) .. controls (229.83,90.73) and (228.49,89.39) .. (228.49,87.73) -- cycle ;
\draw  [color={rgb, 255:red, 0; green, 0; blue, 0 }  ,draw opacity=1 ][fill={rgb, 255:red, 0; green, 0; blue, 0 }  ,fill opacity=1 ][line width=0.75]  (281.06,87.45) .. controls (281.06,85.63) and (282.4,84.16) .. (284.06,84.16) .. controls (285.71,84.16) and (287.06,85.63) .. (287.06,87.45) .. controls (287.06,89.26) and (285.71,90.73) .. (284.06,90.73) .. controls (282.4,90.73) and (281.06,89.26) .. (281.06,87.45) -- cycle ;
\draw  [color={rgb, 255:red, 0; green, 0; blue, 0 }  ,draw opacity=1 ][fill={rgb, 255:red, 0; green, 0; blue, 0 }  ,fill opacity=1 ][line width=0.75]  (333.06,87.45) .. controls (333.06,85.63) and (334.4,84.16) .. (336.06,84.16) .. controls (337.71,84.16) and (339.06,85.63) .. (339.06,87.45) .. controls (339.06,89.26) and (337.71,90.73) .. (336.06,90.73) .. controls (334.4,90.73) and (333.06,89.26) .. (333.06,87.45) -- cycle ;
\draw  [color={rgb, 255:red, 0; green, 0; blue, 0 }  ,draw opacity=1 ][fill={rgb, 255:red, 0; green, 0; blue, 0 }  ,fill opacity=1 ][line width=0.75]  (422.93,89.09) .. controls (422.93,87.44) and (424.27,86.09) .. (425.93,86.09) .. controls (427.58,86.09) and (428.93,87.44) .. (428.93,89.09) .. controls (428.93,90.75) and (427.58,92.09) .. (425.93,92.09) .. controls (424.27,92.09) and (422.93,90.75) .. (422.93,89.09) -- cycle ;
\draw  [color={rgb, 255:red, 0; green, 0; blue, 0 }  ,draw opacity=1 ][fill={rgb, 255:red, 0; green, 0; blue, 0 }  ,fill opacity=1 ][line width=0.75]  (470.93,88.69) .. controls (470.93,87.04) and (472.27,85.69) .. (473.93,85.69) .. controls (475.58,85.69) and (476.93,87.04) .. (476.93,88.69) .. controls (476.93,90.35) and (475.58,91.69) .. (473.93,91.69) .. controls (472.27,91.69) and (470.93,90.35) .. (470.93,88.69) -- cycle ;
\draw  [color={rgb, 255:red, 0; green, 0; blue, 0 }  ,draw opacity=1 ][fill={rgb, 255:red, 0; green, 0; blue, 0 }  ,fill opacity=1 ][line width=0.75]  (520.53,88.69) .. controls (520.53,87.04) and (521.87,85.69) .. (523.53,85.69) .. controls (525.18,85.69) and (526.53,87.04) .. (526.53,88.69) .. controls (526.53,90.35) and (525.18,91.69) .. (523.53,91.69) .. controls (521.87,91.69) and (520.53,90.35) .. (520.53,88.69) -- cycle ;
\draw  [color={rgb, 255:red, 0; green, 0; blue, 0 }  ,draw opacity=1 ][fill={rgb, 255:red, 0; green, 0; blue, 0 }  ,fill opacity=1 ][line width=0.75]  (570.49,88.4) .. controls (570.49,86.74) and (571.83,85.4) .. (573.49,85.4) .. controls (575.14,85.4) and (576.49,86.74) .. (576.49,88.4) .. controls (576.49,90.06) and (575.14,91.4) .. (573.49,91.4) .. controls (571.83,91.4) and (570.49,90.06) .. (570.49,88.4) -- cycle ;
\draw  [color={rgb, 255:red, 0; green, 0; blue, 0 }  ,draw opacity=1 ][fill={rgb, 255:red, 0; green, 0; blue, 0 }  ,fill opacity=1 ][line width=0.75]  (623.06,88.11) .. controls (623.06,86.3) and (624.4,84.83) .. (626.06,84.83) .. controls (627.71,84.83) and (629.06,86.3) .. (629.06,88.11) .. controls (629.06,89.93) and (627.71,91.4) .. (626.06,91.4) .. controls (624.4,91.4) and (623.06,89.93) .. (623.06,88.11) -- cycle ;
\draw  [color={rgb, 255:red, 0; green, 0; blue, 0 }  ,draw opacity=1 ][fill={rgb, 255:red, 0; green, 0; blue, 0 }  ,fill opacity=1 ][line width=0.75]  (669.06,87.11) .. controls (669.06,85.3) and (670.4,83.83) .. (672.06,83.83) .. controls (673.71,83.83) and (675.06,85.3) .. (675.06,87.11) .. controls (675.06,88.93) and (673.71,90.4) .. (672.06,90.4) .. controls (670.4,90.4) and (669.06,88.93) .. (669.06,87.11) -- cycle ;

\draw (362.59,84.77) node [anchor=north west][inner sep=0.75pt]    {$\sim _{\tr}$};
\draw (19.4,78) node [anchor=north west][inner sep=0.75pt]   [align=left] {then };
\draw (172.75,22.47) node [anchor=north west][inner sep=0.75pt]    {$Q\ =$};
\draw (44,3) node [anchor=north west][inner sep=0.75pt]   [align=left] {\begin{minipage}[lt]{61.68pt}\setlength\topsep{0pt}
For instance 
\begin{center}
if we set
\end{center}

\end{minipage}};
\end{tikzpicture}
\end{center}
\begin{proof}
    We proceed by induction on $\ell$. The case $\ell=1$ is a consequence of \Cref{lem:pqqp}.
    Now suppose the statement holds for $\ell$, and let us prove it for $\ell+1$.
    By the induction hypothesis, we have $(a_1a_1^*)^n \sim_{\tr} \pm (a_{\ell}^*a_{\ell})^n$ for all $n \in \NN$.
    By assumption, the relation $\mu_{ta_\ell}$ implies that $a_{\ell}^*a_{\ell} = \pm a_{\ell+1}a_{\ell+1}^* \in \Pi(Q).$
    By \Cref{lem:pqqp}, we obtain $(a_1a_1^*)^n \sim_{\tr} \pm (a_{\ell+1}^*a_{\ell+1})^n$ for all $n \in \NN$.
\end{proof}

\begin{notation}\label{homotopic class}
    Let $Q$ be a quiver, and denote by $|Q|$ the underlying topological graph obtained by forgetting the orientation of the arrows.
    We also consider the natural map $|\overline{Q}| \to |Q|$ obtained by identifying each pair of arrows $a,a^*$ with a single edge of $|Q|$.
    Any cycle in $\overline{Q}$ is mapped to a cycle in $|Q|$, and we can consider its homotopy class in $\pi_1(|Q|)$.
\end{notation}

\begin{lem}\label{cycle homotopically trivial}
    Let $C = a_1 \cdots a_{\ell}$ be a cycle in $\overline{Q}$ such that, for each $a_k$ with $1\leq k < l$, there are exactly two distinct arrows in $Q$ incident to the vertex $ta_k=ha_{k+1}$, and suppose that $C$ is homotopically trivial in $|Q|$. Then:
    $$C  = \pm(a_{\ell}^*a_{\ell})^{\ell/2} \in \Pi(Q).$$
\end{lem}
\begin{center}
\begin{tikzpicture}[x=0.5pt,y=0.5pt,yscale=-1,xscale=1]

\draw [line width=0.75]    (236.67,89.43) .. controls (205.87,87.43) and (212.21,86.4) .. (192.21,87.6) .. controls (172.21,88.8) and (86.61,81.2) .. (86.21,96.4) .. controls (85.81,111.6) and (176.21,104.4) .. (193.41,104.4) .. controls (210.1,104.4) and (208.2,104.93) .. (235.09,102.04) ;
\draw [shift={(237.67,101.77)}, rotate = 173.83] [fill={rgb, 255:red, 0; green, 0; blue, 0 }  ][line width=0.08]  [draw opacity=0] (10.72,-5.15) -- (0,0) -- (10.72,5.15) -- (7.12,0) -- cycle    ;
\draw  [color={rgb, 255:red, 0; green, 0; blue, 0 }  ,draw opacity=1 ][fill={rgb, 255:red, 0; green, 0; blue, 0 }  ,fill opacity=1 ][line width=0.75]  (92.77,96.96) .. controls (92.77,95.3) and (94.12,93.96) .. (95.77,93.96) .. controls (97.43,93.96) and (98.77,95.3) .. (98.77,96.96) .. controls (98.77,98.62) and (97.43,99.96) .. (95.77,99.96) .. controls (94.12,99.96) and (92.77,98.62) .. (92.77,96.96) -- cycle ;
\draw  [color={rgb, 255:red, 0; green, 0; blue, 0 }  ,draw opacity=1 ][fill={rgb, 255:red, 0; green, 0; blue, 0 }  ,fill opacity=1 ][line width=0.75]  (140.77,96.56) .. controls (140.77,94.9) and (142.12,93.56) .. (143.77,93.56) .. controls (145.43,93.56) and (146.77,94.9) .. (146.77,96.56) .. controls (146.77,98.22) and (145.43,99.56) .. (143.77,99.56) .. controls (142.12,99.56) and (140.77,98.22) .. (140.77,96.56) -- cycle ;
\draw  [color={rgb, 255:red, 0; green, 0; blue, 0 }  ,draw opacity=1 ][fill={rgb, 255:red, 0; green, 0; blue, 0 }  ,fill opacity=1 ][line width=0.75]  (190.37,96.56) .. controls (190.37,94.9) and (191.72,93.56) .. (193.37,93.56) .. controls (195.03,93.56) and (196.37,94.9) .. (196.37,96.56) .. controls (196.37,98.22) and (195.03,99.56) .. (193.37,99.56) .. controls (191.72,99.56) and (190.37,98.22) .. (190.37,96.56) -- cycle ;
\draw  [color={rgb, 255:red, 0; green, 0; blue, 0 }  ,draw opacity=1 ][fill={rgb, 255:red, 0; green, 0; blue, 0 }  ,fill opacity=1 ][line width=0.75]  (240.73,96.27) .. controls (240.73,94.61) and (242.08,93.27) .. (243.73,93.27) .. controls (245.39,93.27) and (246.73,94.61) .. (246.73,96.27) .. controls (246.73,97.92) and (245.39,99.27) .. (243.73,99.27) .. controls (242.08,99.27) and (240.73,97.92) .. (240.73,96.27) -- cycle ;
\draw [line width=0.75]    (435.13,87.43) .. controls (411.13,82.43) and (362.88,78.53) .. (346.08,81.33) .. controls (329.28,84.13) and (384.88,112.13) .. (384.88,94.93) .. controls (384.88,77.73) and (333.28,102.53) .. (344.88,108.53) .. controls (356.02,114.29) and (410.52,109.52) .. (434.91,103.23) ;
\draw [shift={(437.8,102.43)}, rotate = 163.61] [fill={rgb, 255:red, 0; green, 0; blue, 0 }  ][line width=0.08]  [draw opacity=0] (10.72,-5.15) -- (0,0) -- (10.72,5.15) -- (7.12,0) -- cycle    ;
\draw  [color={rgb, 255:red, 0; green, 0; blue, 0 }  ,draw opacity=1 ][fill={rgb, 255:red, 0; green, 0; blue, 0 }  ,fill opacity=1 ][line width=0.75]  (291.84,95.49) .. controls (291.84,93.84) and (293.18,92.49) .. (294.84,92.49) .. controls (296.5,92.49) and (297.84,93.84) .. (297.84,95.49) .. controls (297.84,97.15) and (296.5,98.49) .. (294.84,98.49) .. controls (293.18,98.49) and (291.84,97.15) .. (291.84,95.49) -- cycle ;
\draw  [color={rgb, 255:red, 0; green, 0; blue, 0 }  ,draw opacity=1 ][fill={rgb, 255:red, 0; green, 0; blue, 0 }  ,fill opacity=1 ][line width=0.75]  (339.84,95.09) .. controls (339.84,93.44) and (341.18,92.09) .. (342.84,92.09) .. controls (344.5,92.09) and (345.84,93.44) .. (345.84,95.09) .. controls (345.84,96.75) and (344.5,98.09) .. (342.84,98.09) .. controls (341.18,98.09) and (339.84,96.75) .. (339.84,95.09) -- cycle ;
\draw  [color={rgb, 255:red, 0; green, 0; blue, 0 }  ,draw opacity=1 ][fill={rgb, 255:red, 0; green, 0; blue, 0 }  ,fill opacity=1 ][line width=0.75]  (389.44,95.09) .. controls (389.44,93.44) and (390.78,92.09) .. (392.44,92.09) .. controls (394.1,92.09) and (395.44,93.44) .. (395.44,95.09) .. controls (395.44,96.75) and (394.1,98.09) .. (392.44,98.09) .. controls (390.78,98.09) and (389.44,96.75) .. (389.44,95.09) -- cycle ;
\draw  [color={rgb, 255:red, 0; green, 0; blue, 0 }  ,draw opacity=1 ][fill={rgb, 255:red, 0; green, 0; blue, 0 }  ,fill opacity=1 ][line width=0.75]  (439.8,94.8) .. controls (439.8,93.14) and (441.14,91.8) .. (442.8,91.8) .. controls (444.46,91.8) and (445.8,93.14) .. (445.8,94.8) .. controls (445.8,96.46) and (444.46,97.8) .. (442.8,97.8) .. controls (441.14,97.8) and (439.8,96.46) .. (439.8,94.8) -- cycle ;
\draw [line width=0.75]    (639.47,87.43) .. controls (615.47,82.43) and (602.6,73.63) .. (591.17,78.49) .. controls (579.74,83.34) and (618.31,96.49) .. (625.46,89.91) .. controls (632.6,83.34) and (589.46,85.06) .. (589.46,94.49) .. controls (589.46,103.91) and (620.03,108.2) .. (625.74,97.63) .. controls (631.46,87.06) and (580.6,109.34) .. (594.6,112.49) .. controls (607.83,115.46) and (623.87,109.24) .. (637.89,101.56) ;
\draw [shift={(640.31,100.2)}, rotate = 150.38] [fill={rgb, 255:red, 0; green, 0; blue, 0 }  ][line width=0.08]  [draw opacity=0] (10.72,-5.15) -- (0,0) -- (10.72,5.15) -- (7.12,0) -- cycle    ;
\draw  [color={rgb, 255:red, 0; green, 0; blue, 0 }  ,draw opacity=1 ][fill={rgb, 255:red, 0; green, 0; blue, 0 }  ,fill opacity=1 ][line width=0.75]  (496.17,95.49) .. controls (496.17,93.84) and (497.52,92.49) .. (499.17,92.49) .. controls (500.83,92.49) and (502.17,93.84) .. (502.17,95.49) .. controls (502.17,97.15) and (500.83,98.49) .. (499.17,98.49) .. controls (497.52,98.49) and (496.17,97.15) .. (496.17,95.49) -- cycle ;
\draw  [color={rgb, 255:red, 0; green, 0; blue, 0 }  ,draw opacity=1 ][fill={rgb, 255:red, 0; green, 0; blue, 0 }  ,fill opacity=1 ][line width=0.75]  (544.17,95.09) .. controls (544.17,93.44) and (545.52,92.09) .. (547.17,92.09) .. controls (548.83,92.09) and (550.17,93.44) .. (550.17,95.09) .. controls (550.17,96.75) and (548.83,98.09) .. (547.17,98.09) .. controls (545.52,98.09) and (544.17,96.75) .. (544.17,95.09) -- cycle ;
\draw  [color={rgb, 255:red, 0; green, 0; blue, 0 }  ,draw opacity=1 ][fill={rgb, 255:red, 0; green, 0; blue, 0 }  ,fill opacity=1 ][line width=0.75]  (593.77,95.09) .. controls (593.77,93.44) and (595.12,92.09) .. (596.77,92.09) .. controls (598.43,92.09) and (599.77,93.44) .. (599.77,95.09) .. controls (599.77,96.75) and (598.43,98.09) .. (596.77,98.09) .. controls (595.12,98.09) and (593.77,96.75) .. (593.77,95.09) -- cycle ;
\draw  [color={rgb, 255:red, 0; green, 0; blue, 0 }  ,draw opacity=1 ][fill={rgb, 255:red, 0; green, 0; blue, 0 }  ,fill opacity=1 ][line width=0.75]  (644.13,94.8) .. controls (644.13,93.14) and (645.48,91.8) .. (647.13,91.8) .. controls (648.79,91.8) and (650.13,93.14) .. (650.13,94.8) .. controls (650.13,96.46) and (648.79,97.8) .. (647.13,97.8) .. controls (645.48,97.8) and (644.13,96.46) .. (644.13,94.8) -- cycle ;
\draw  [color={rgb, 255:red, 0; green, 0; blue, 0 }  ,draw opacity=1 ][fill={rgb, 255:red, 0; green, 0; blue, 0 }  ,fill opacity=1 ][line width=0.75]  (275.17,41.83) .. controls (275.17,40.17) and (276.52,38.83) .. (278.17,38.83) .. controls (279.83,38.83) and (281.17,40.17) .. (281.17,41.83) .. controls (281.17,43.48) and (279.83,44.83) .. (278.17,44.83) .. controls (276.52,44.83) and (275.17,43.48) .. (275.17,41.83) -- cycle ;
\draw  [color={rgb, 255:red, 0; green, 0; blue, 0 }  ,draw opacity=1 ][fill={rgb, 255:red, 0; green, 0; blue, 0 }  ,fill opacity=1 ][line width=0.75]  (323.17,41.43) .. controls (323.17,39.77) and (324.52,38.43) .. (326.17,38.43) .. controls (327.83,38.43) and (329.17,39.77) .. (329.17,41.43) .. controls (329.17,43.08) and (327.83,44.43) .. (326.17,44.43) .. controls (324.52,44.43) and (323.17,43.08) .. (323.17,41.43) -- cycle ;
\draw  [color={rgb, 255:red, 0; green, 0; blue, 0 }  ,draw opacity=1 ][fill={rgb, 255:red, 0; green, 0; blue, 0 }  ,fill opacity=1 ][line width=0.75]  (372.77,41.43) .. controls (372.77,39.77) and (374.12,38.43) .. (375.77,38.43) .. controls (377.43,38.43) and (378.77,39.77) .. (378.77,41.43) .. controls (378.77,43.08) and (377.43,44.43) .. (375.77,44.43) .. controls (374.12,44.43) and (372.77,43.08) .. (372.77,41.43) -- cycle ;
\draw  [color={rgb, 255:red, 0; green, 0; blue, 0 }  ,draw opacity=1 ][fill={rgb, 255:red, 0; green, 0; blue, 0 }  ,fill opacity=1 ][line width=0.75]  (422.73,41.13) .. controls (422.73,39.48) and (424.08,38.13) .. (425.73,38.13) .. controls (427.39,38.13) and (428.73,39.48) .. (428.73,41.13) .. controls (428.73,42.79) and (427.39,44.13) .. (425.73,44.13) .. controls (424.08,44.13) and (422.73,42.79) .. (422.73,41.13) -- cycle ;
\draw [line width=1.5]    (284.73,41.63) -- (314.07,41.63) ;
\draw [shift={(318.07,41.63)}, rotate = 180] [fill={rgb, 255:red, 0; green, 0; blue, 0 }  ][line width=0.08]  [draw opacity=0] (13.4,-6.43) -- (0,0) -- (13.4,6.44) -- (8.9,0) -- cycle    ;
\draw [line width=1.5]    (333.73,41.63) -- (363.07,41.63) ;
\draw [shift={(367.07,41.63)}, rotate = 180] [fill={rgb, 255:red, 0; green, 0; blue, 0 }  ][line width=0.08]  [draw opacity=0] (13.4,-6.43) -- (0,0) -- (13.4,6.44) -- (8.9,0) -- cycle    ;
\draw [line width=1.5]    (383.73,41.63) -- (413.07,41.63) ;
\draw [shift={(417.07,41.63)}, rotate = 180] [fill={rgb, 255:red, 0; green, 0; blue, 0 }  ][line width=0.08]  [draw opacity=0] (13.4,-6.43) -- (0,0) -- (13.4,6.44) -- (8.9,0) -- cycle    ;

\draw (258,92.73) node [anchor=north west][inner sep=0.75pt]    {$=$};
\draw (462.33,91.73) node [anchor=north west][inner sep=0.75pt]    {$=$};
\draw (218,31.87) node [anchor=north west][inner sep=0.75pt]    {$Q\ =$};
\draw (478,29) node [anchor=north west][inner sep=0.75pt]   [align=left] {then };
\draw (62,7.67) node [anchor=north west][inner sep=0.75pt]   [align=left] {\begin{minipage}[lt]{61.68pt}\setlength\topsep{0pt}
For instance 
\begin{center}
if we set
\end{center}

\end{minipage}};
\draw (674,88.07) node [anchor=north west][inner sep=0.75pt]    {$\text{in \ } \Pi ( Q) .$};
\end{tikzpicture}
\end{center}

\begin{proof}
    We proceed by induction on $\ell$, and consider the integer
    $$r=\#\{\, 1\leq k < \ell \mid a_k \neq a_{k+1}^* \,\}.$$
    
    We begin with the induction step. Assume that $\ell \geq 2$ and that the statement holds for any cycle of length strictly smaller than $\ell$.
    
    If $r=0$, then by definition $C = (a_{\ell}^*a_{\ell})^{\ell/2}$.
    Assume now that $1\leq r<\ell-1$, so that $C$ contains a short loop. By \Cref{lem: sliding short loops}, we may write
    $$C = \pm C' a_{\ell}^*a_{\ell},$$
    where $C' = a_1\cdots b$ is a shorter cycle with $b\in\{a_1^*,a_{\ell}\}$.
    Since both $C$ and the short loop $a_{\ell}^*a_{\ell}$ are homotopically trivial in $|Q|$, the same holds for $C'$. By the induction hypothesis applied to $C'$, we obtain
    $$
    C'=\pm (b^*b)^{\ell/2-1}.
    $$
    If $b=a_{\ell}$, the conclusion follows immediately. Otherwise, $b=a_1^*$, and applying the relation $\mu_{ha_1}$ gives $b^*b=\pm a_{\ell}^*a_{\ell}$.
    
    Finally, consider the case $r=\ell-1$. Then, for every $1\leq k < \ell$,
    we have $a_{k+1}\neq a_k^*$, and each arrow $a_{k+1}\notin \{a_k,a_k^*,\ldots,a_1,a_1^*\}$ is a new arrow of the quiver.
    The map sending the cycle $C$ to $|Q|$, obtained by contracting each pair $(a_k,a_k^*)$, therefore embeds $C$ as a nontrivial cycle in $|Q|$, contradicting the assumption that $C$ is homotopically trivial.
    
    For the initialization, let $\ell=2$. If $r=0$, then $C$ is a short loop and the statement is immediate. If $r=1$, then $C$ is not homotopically trivial in $\pi_1(|Q|)$.
\end{proof}

\subsection{Symmetries of \texorpdfstring{$\widetilde{A}$}{A~tilde} quivers}\label{symmetries Atilde}

In this section, we first apply the graphical calculus to quivers of affine type $\widetilde{A}$. We then study the various symmetry classes of these quivers and compute embeddings of the corresponding quiver varieties into affine spaces that can be realized in $\CC^3$.
Recall that the fundamental root $\delta$ of a quiver of type $\widetilde{A}$ is the dimension vector with value $1$ at every vertex, namely
$$\delta=(1,\ldots,1).$$

\begin{prop}\label{prop:topo cycle}
Let $n \geq 1$, and let $\widetilde{A}_{n-1}$ be the cyclic quiver with $n$ vertices:
$$
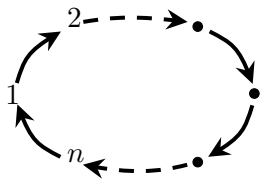
\begin{tikzpicture}[x=0.6pt,y=0.6pt,yscale=-1,xscale=1]

\draw  [color={rgb, 255:red, 0; green, 0; blue, 0 }  ,draw opacity=1 ][fill={rgb, 255:red, 0; green, 0; blue, 0 }  ,fill opacity=1 ] (422,150) .. controls (422,148.34) and (423.34,147) .. (425,147) .. controls (426.66,147) and (428,148.34) .. (428,150) .. controls (428,151.66) and (426.66,153) .. (425,153) .. controls (423.34,153) and (422,151.66) .. (422,150) -- cycle ;
\draw  [color={rgb, 255:red, 0; green, 0; blue, 0 }  ,draw opacity=1 ][fill={rgb, 255:red, 0; green, 0; blue, 0 }  ,fill opacity=1 ] (386.5,108) .. controls (386.5,106.34) and (387.84,105) .. (389.5,105) .. controls (391.16,105) and (392.5,106.34) .. (392.5,108) .. controls (392.5,109.66) and (391.16,111) .. (389.5,111) .. controls (387.84,111) and (386.5,109.66) .. (386.5,108) -- cycle ;
\draw  [color={rgb, 255:red, 0; green, 0; blue, 0 }  ,draw opacity=1 ][fill={rgb, 255:red, 0; green, 0; blue, 0 }  ,fill opacity=1 ] (386.5,193) .. controls (386.5,191.34) and (387.84,190) .. (389.5,190) .. controls (391.16,190) and (392.5,191.34) .. (392.5,193) .. controls (392.5,194.66) and (391.16,196) .. (389.5,196) .. controls (387.84,196) and (386.5,194.66) .. (386.5,193) -- cycle ;
\draw [line width=1.5]    (276.03,143.38) .. controls (280.62,128.53) and (282.3,123.84) .. (300.41,113.04) ;
\draw [shift={(303.74,111.09)}, rotate = 149.01] [fill={rgb, 255:red, 0; green, 0; blue, 0 }  ][line width=0.08]  [draw opacity=0] (13.4,-6.43) -- (0,0) -- (13.4,6.44) -- (8.9,0) -- cycle    ;
\draw [line width=1.5]  [dash pattern={on 5.63pt off 4.5pt}]  (317.74,105.09) .. controls (336.96,101.79) and (352.11,101.34) .. (379.25,104.6) ;
\draw [shift={(383.17,105.09)}, rotate = 186.82] [fill={rgb, 255:red, 0; green, 0; blue, 0 }  ][line width=0.08]  [draw opacity=0] (13.4,-6.43) -- (0,0) -- (13.4,6.44) -- (8.9,0) -- cycle    ;
\draw [line width=1.5]  [dash pattern={on 5.63pt off 4.5pt}]  (382.88,194.63) .. controls (364.72,199.41) and (346.88,200.23) .. (321.15,195.35) ;
\draw [shift={(317.45,194.63)}, rotate = 10.54] [fill={rgb, 255:red, 0; green, 0; blue, 0 }  ][line width=0.08]  [draw opacity=0] (13.4,-6.43) -- (0,0) -- (13.4,6.44) -- (8.9,0) -- cycle    ;
\draw [line width=1.5]    (423.74,157.38) .. controls (419.15,172.23) and (417.46,176.92) .. (399.36,187.71) ;
\draw [shift={(396.03,189.66)}, rotate = 329.01] [fill={rgb, 255:red, 0; green, 0; blue, 0 }  ][line width=0.08]  [draw opacity=0] (13.4,-6.43) -- (0,0) -- (13.4,6.44) -- (8.9,0) -- cycle    ;
\draw [line width=1.5]    (396.96,111.07) .. controls (410.79,118.17) and (415.12,120.64) .. (422.6,140.35) ;
\draw [shift={(423.95,143.97)}, rotate = 249.01] [fill={rgb, 255:red, 0; green, 0; blue, 0 }  ][line width=0.08]  [draw opacity=0] (13.4,-6.43) -- (0,0) -- (13.4,6.44) -- (8.9,0) -- cycle    ;
\draw [line width=1.5]    (303.37,189.68) .. controls (289.55,182.58) and (285.22,180.11) .. (277.73,160.4) ;
\draw [shift={(276.39,156.78)}, rotate = 69.01] [fill={rgb, 255:red, 0; green, 0; blue, 0 }  ][line width=0.08]  [draw opacity=0] (13.4,-6.43) -- (0,0) -- (13.4,6.44) -- (8.9,0) -- cycle    ;

\draw (267.29,142.98) node [anchor=north west][inner sep=0.75pt]  [font=\normalsize]  {$1$};
\draw (306.09,94.98) node [anchor=north west][inner sep=0.75pt]  [font=\normalsize]  {$2$};
\draw (306.09,183.38) node [anchor=north west][inner sep=0.75pt]  [font=\normalsize]  {$n$};
\end{tikzpicture}$$

Denote $C_1, C_2, C_3$ the following three cycles with $\ell(C_1) = \ell(C_2) = n$ and $\ell(C_3) = 2$.
$$
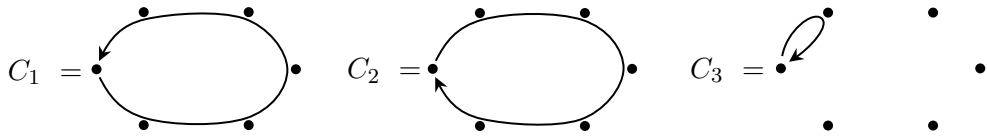
\begin{tikzpicture}[x=0.5pt,y=0.5pt,yscale=-1,xscale=1]

\draw  [color={rgb, 255:red, 0; green, 0; blue, 0 }  ,draw opacity=1 ][fill={rgb, 255:red, 0; green, 0; blue, 0 }  ,fill opacity=1 ][line width=0.75]  (252.33,141) .. controls (252.33,139.34) and (253.68,138) .. (255.33,138) .. controls (256.99,138) and (258.33,139.34) .. (258.33,141) .. controls (258.33,142.66) and (256.99,144) .. (255.33,144) .. controls (253.68,144) and (252.33,142.66) .. (252.33,141) -- cycle ;
\draw  [color={rgb, 255:red, 0; green, 0; blue, 0 }  ,draw opacity=1 ][fill={rgb, 255:red, 0; green, 0; blue, 0 }  ,fill opacity=1 ][line width=0.75]  (402.33,141) .. controls (402.33,139.34) and (403.68,138) .. (405.33,138) .. controls (406.99,138) and (408.33,139.34) .. (408.33,141) .. controls (408.33,142.66) and (406.99,144) .. (405.33,144) .. controls (403.68,144) and (402.33,142.66) .. (402.33,141) -- cycle ;
\draw  [color={rgb, 255:red, 0; green, 0; blue, 0 }  ,draw opacity=1 ][fill={rgb, 255:red, 0; green, 0; blue, 0 }  ,fill opacity=1 ][line width=0.75]  (287.83,99) .. controls (287.83,97.34) and (289.18,96) .. (290.83,96) .. controls (292.49,96) and (293.83,97.34) .. (293.83,99) .. controls (293.83,100.66) and (292.49,102) .. (290.83,102) .. controls (289.18,102) and (287.83,100.66) .. (287.83,99) -- cycle ;
\draw  [color={rgb, 255:red, 0; green, 0; blue, 0 }  ,draw opacity=1 ][fill={rgb, 255:red, 0; green, 0; blue, 0 }  ,fill opacity=1 ][line width=0.75]  (366.83,99) .. controls (366.83,97.34) and (368.18,96) .. (369.83,96) .. controls (371.49,96) and (372.83,97.34) .. (372.83,99) .. controls (372.83,100.66) and (371.49,102) .. (369.83,102) .. controls (368.18,102) and (366.83,100.66) .. (366.83,99) -- cycle ;
\draw  [color={rgb, 255:red, 0; green, 0; blue, 0 }  ,draw opacity=1 ][fill={rgb, 255:red, 0; green, 0; blue, 0 }  ,fill opacity=1 ][line width=0.75]  (287.83,184) .. controls (287.83,182.34) and (289.18,181) .. (290.83,181) .. controls (292.49,181) and (293.83,182.34) .. (293.83,184) .. controls (293.83,185.66) and (292.49,187) .. (290.83,187) .. controls (289.18,187) and (287.83,185.66) .. (287.83,184) -- cycle ;
\draw  [color={rgb, 255:red, 0; green, 0; blue, 0 }  ,draw opacity=1 ][fill={rgb, 255:red, 0; green, 0; blue, 0 }  ,fill opacity=1 ][line width=0.75]  (366.83,184) .. controls (366.83,182.34) and (368.18,181) .. (369.83,181) .. controls (371.49,181) and (372.83,182.34) .. (372.83,184) .. controls (372.83,185.66) and (371.49,187) .. (369.83,187) .. controls (368.18,187) and (366.83,185.66) .. (366.83,184) -- cycle ;
\draw [line width=0.75]    (257.62,134.93) .. controls (264.51,121.6) and (272.47,109.58) .. (292.69,104.02) .. controls (312.91,98.47) and (348.91,98.02) .. (368.02,104.24) .. controls (387.13,110.47) and (399.13,126.69) .. (398.91,140.91) .. controls (398.69,155.13) and (384.2,173.33) .. (365.8,179.36) .. controls (347.4,185.38) and (311.98,183.11) .. (292.24,178.69) .. controls (273.6,174.51) and (265.23,163.76) .. (258.36,149.65) ;
\draw [shift={(257.18,147.16)}, rotate = 64.05] [fill={rgb, 255:red, 0; green, 0; blue, 0 }  ][line width=0.08]  [draw opacity=0] (10.72,-5.15) -- (0,0) -- (10.72,5.15) -- (7.12,0) -- cycle    ;
\draw  [color={rgb, 255:red, 0; green, 0; blue, 0 }  ,draw opacity=1 ][fill={rgb, 255:red, 0; green, 0; blue, 0 }  ,fill opacity=1 ][line width=0.75]  (-0.67,141.33) .. controls (-0.67,142.99) and (0.68,144.33) .. (2.33,144.33) .. controls (3.99,144.33) and (5.33,142.99) .. (5.33,141.33) .. controls (5.33,139.68) and (3.99,138.33) .. (2.33,138.33) .. controls (0.68,138.33) and (-0.67,139.68) .. (-0.67,141.33) -- cycle ;
\draw  [color={rgb, 255:red, 0; green, 0; blue, 0 }  ,draw opacity=1 ][fill={rgb, 255:red, 0; green, 0; blue, 0 }  ,fill opacity=1 ][line width=0.75]  (149.33,141.33) .. controls (149.33,142.99) and (150.68,144.33) .. (152.33,144.33) .. controls (153.99,144.33) and (155.33,142.99) .. (155.33,141.33) .. controls (155.33,139.68) and (153.99,138.33) .. (152.33,138.33) .. controls (150.68,138.33) and (149.33,139.68) .. (149.33,141.33) -- cycle ;
\draw  [color={rgb, 255:red, 0; green, 0; blue, 0 }  ,draw opacity=1 ][fill={rgb, 255:red, 0; green, 0; blue, 0 }  ,fill opacity=1 ][line width=0.75]  (34.83,183.33) .. controls (34.83,184.99) and (36.18,186.33) .. (37.83,186.33) .. controls (39.49,186.33) and (40.83,184.99) .. (40.83,183.33) .. controls (40.83,181.68) and (39.49,180.33) .. (37.83,180.33) .. controls (36.18,180.33) and (34.83,181.68) .. (34.83,183.33) -- cycle ;
\draw  [color={rgb, 255:red, 0; green, 0; blue, 0 }  ,draw opacity=1 ][fill={rgb, 255:red, 0; green, 0; blue, 0 }  ,fill opacity=1 ][line width=0.75]  (113.83,183.33) .. controls (113.83,184.99) and (115.18,186.33) .. (116.83,186.33) .. controls (118.49,186.33) and (119.83,184.99) .. (119.83,183.33) .. controls (119.83,181.68) and (118.49,180.33) .. (116.83,180.33) .. controls (115.18,180.33) and (113.83,181.68) .. (113.83,183.33) -- cycle ;
\draw  [color={rgb, 255:red, 0; green, 0; blue, 0 }  ,draw opacity=1 ][fill={rgb, 255:red, 0; green, 0; blue, 0 }  ,fill opacity=1 ][line width=0.75]  (34.83,98.33) .. controls (34.83,99.99) and (36.18,101.33) .. (37.83,101.33) .. controls (39.49,101.33) and (40.83,99.99) .. (40.83,98.33) .. controls (40.83,96.68) and (39.49,95.33) .. (37.83,95.33) .. controls (36.18,95.33) and (34.83,96.68) .. (34.83,98.33) -- cycle ;
\draw  [color={rgb, 255:red, 0; green, 0; blue, 0 }  ,draw opacity=1 ][fill={rgb, 255:red, 0; green, 0; blue, 0 }  ,fill opacity=1 ][line width=0.75]  (113.83,98.33) .. controls (113.83,99.99) and (115.18,101.33) .. (116.83,101.33) .. controls (118.49,101.33) and (119.83,99.99) .. (119.83,98.33) .. controls (119.83,96.68) and (118.49,95.33) .. (116.83,95.33) .. controls (115.18,95.33) and (113.83,96.68) .. (113.83,98.33) -- cycle ;
\draw [line width=0.75]    (4.62,147.4) .. controls (11.51,160.73) and (19.47,172.76) .. (39.69,178.31) .. controls (59.91,183.87) and (95.91,184.31) .. (115.02,178.09) .. controls (134.13,171.87) and (146.13,155.64) .. (145.91,141.42) .. controls (145.69,127.2) and (131.2,109) .. (112.8,102.98) .. controls (94.4,96.96) and (58.98,99.22) .. (39.24,103.64) .. controls (20.6,107.82) and (12.23,118.57) .. (5.36,132.68) ;
\draw [shift={(4.18,135.18)}, rotate = 295.95] [fill={rgb, 255:red, 0; green, 0; blue, 0 }  ][line width=0.08]  [draw opacity=0] (10.72,-5.15) -- (0,0) -- (10.72,5.15) -- (7.12,0) -- cycle    ;
\draw  [color={rgb, 255:red, 0; green, 0; blue, 0 }  ,draw opacity=1 ][fill={rgb, 255:red, 0; green, 0; blue, 0 }  ,fill opacity=1 ][line width=0.75]  (514.33,140.67) .. controls (514.33,139.01) and (515.68,137.67) .. (517.33,137.67) .. controls (518.99,137.67) and (520.33,139.01) .. (520.33,140.67) .. controls (520.33,142.32) and (518.99,143.67) .. (517.33,143.67) .. controls (515.68,143.67) and (514.33,142.32) .. (514.33,140.67) -- cycle ;
\draw  [color={rgb, 255:red, 0; green, 0; blue, 0 }  ,draw opacity=1 ][fill={rgb, 255:red, 0; green, 0; blue, 0 }  ,fill opacity=1 ][line width=0.75]  (664.33,140.67) .. controls (664.33,139.01) and (665.68,137.67) .. (667.33,137.67) .. controls (668.99,137.67) and (670.33,139.01) .. (670.33,140.67) .. controls (670.33,142.32) and (668.99,143.67) .. (667.33,143.67) .. controls (665.68,143.67) and (664.33,142.32) .. (664.33,140.67) -- cycle ;
\draw  [color={rgb, 255:red, 0; green, 0; blue, 0 }  ,draw opacity=1 ][fill={rgb, 255:red, 0; green, 0; blue, 0 }  ,fill opacity=1 ][line width=0.75]  (549.83,98.67) .. controls (549.83,97.01) and (551.18,95.67) .. (552.83,95.67) .. controls (554.49,95.67) and (555.83,97.01) .. (555.83,98.67) .. controls (555.83,100.32) and (554.49,101.67) .. (552.83,101.67) .. controls (551.18,101.67) and (549.83,100.32) .. (549.83,98.67) -- cycle ;
\draw  [color={rgb, 255:red, 0; green, 0; blue, 0 }  ,draw opacity=1 ][fill={rgb, 255:red, 0; green, 0; blue, 0 }  ,fill opacity=1 ][line width=0.75]  (628.83,98.67) .. controls (628.83,97.01) and (630.18,95.67) .. (631.83,95.67) .. controls (633.49,95.67) and (634.83,97.01) .. (634.83,98.67) .. controls (634.83,100.32) and (633.49,101.67) .. (631.83,101.67) .. controls (630.18,101.67) and (628.83,100.32) .. (628.83,98.67) -- cycle ;
\draw  [color={rgb, 255:red, 0; green, 0; blue, 0 }  ,draw opacity=1 ][fill={rgb, 255:red, 0; green, 0; blue, 0 }  ,fill opacity=1 ][line width=0.75]  (549.83,183.67) .. controls (549.83,182.01) and (551.18,180.67) .. (552.83,180.67) .. controls (554.49,180.67) and (555.83,182.01) .. (555.83,183.67) .. controls (555.83,185.32) and (554.49,186.67) .. (552.83,186.67) .. controls (551.18,186.67) and (549.83,185.32) .. (549.83,183.67) -- cycle ;
\draw  [color={rgb, 255:red, 0; green, 0; blue, 0 }  ,draw opacity=1 ][fill={rgb, 255:red, 0; green, 0; blue, 0 }  ,fill opacity=1 ][line width=0.75]  (628.83,183.67) .. controls (628.83,182.01) and (630.18,180.67) .. (631.83,180.67) .. controls (633.49,180.67) and (634.83,182.01) .. (634.83,183.67) .. controls (634.83,185.32) and (633.49,186.67) .. (631.83,186.67) .. controls (630.18,186.67) and (628.83,185.32) .. (628.83,183.67) -- cycle ;
\draw [line width=0.75]    (518.28,131.41) .. controls (521.7,112.55) and (540.56,97.41) .. (547.7,102.84) .. controls (554.49,108) and (540.65,124.76) .. (525.77,134.26) ;
\draw [shift={(523.42,135.7)}, rotate = 324.39] [fill={rgb, 255:red, 0; green, 0; blue, 0 }  ][line width=0.08]  [draw opacity=0] (10.72,-5.15) -- (0,0) -- (10.72,5.15) -- (7.12,0) -- cycle    ;

\draw (-66.33,132.4) node [anchor=north west][inner sep=0.75pt]    {$C_{1} \ =$};
\draw (189,131.4) node [anchor=north west][inner sep=0.75pt]    {$C_{2} \ =$};
\draw (447.17,131.07) node [anchor=north west][inner sep=0.75pt]    {$C_{3} \ =$};
\end{tikzpicture}$$
    Any cycle in $\Pi(\widetilde{A}_{n-1})$ is trace-equivalent (see \Cref{defn:invariant tr relation}) to a cycle of the form $C_1^{k_1} C_3^{k_3}$ or $C_2^{k_2} C_3^{k_3}$ for some integers $k_1, k_2, k_3 \in \NN$. 
    Moreover, we have the following relation in $\Pi(\widetilde{A}_{n-1})$:
    \begin{align*}
         \forall i,j \in \{1,2,3\},C_iC_j & = C_jC_i \text{ and } C_1 C_2  = C_3^{n} \in \Pi(\widetilde{A}_{n-1}).
    \end{align*}
\end{prop}
    
\begin{proof}
    Using \Cref{lem: tlp short loop,lem:pqqp}, any cycle is trace-equivalent to a cycle based at vertex $1$.
    The image of the short loop $C_3$ is homotopically trivial in $|\widetilde{A}_{n-1}|$, while the classes of $C_1$ and $C_2$ generate the fundamental group $\pi_1(|\widetilde{A}_{n-1}|)$ (see \Cref{homotopic class}).
    Using \Cref{cycle homotopically trivial}, every homotopically trivial cycle is equal in $\Pi(\widetilde{A}_{n-1})$ to $C_3^{k_3}$ for some $k_3 \in \NN$. In particular, $C_1C_2 = C_3^n$.
    Using \Cref{lem: sliding short loops}, we deduce that $C_1C_3 = C_3C_1$ and $C_2C_3 = C_3C_2$.
    
    Now let $C$ be an arbitrary cycle in $\overline{Q}$ based at vertex $1$.
    Then the image of $C$ in $|\widetilde{A}_{n-1}|$ is homotopic to the image of either $C_1^{k_1}$ or $C_2^{k_2}$ for some $k_1, k_2 \in \NN$.
    
    Consequently, in $\overline{Q}$, it can be decomposed as
    $$
    C = p_1 q_1 p_2 q_2 \cdots p_r q_r,
    $$
    where the cycles $p_1,\ldots,p_r$ are homotopically trivial in $|\widetilde{A}_{n-1}|$, and the cycle $q_1q_2\cdots q_r$ is equal to either $C_1^{k_1}$ or $C_2^{k_2}$.
    
    Applying \Cref{cycle homotopically trivial} to each cycle $p_1,\ldots,p_r$, whose homotopy class is trivial, and using \Cref{lem: sliding short loops} to collect the powers of the short loops, we conclude that $C$ is equal in $\Pi(\widetilde{A}_{n-1})$ to either $C_1^{k_1}C_3^{k_3}$ or $C_2^{k_2}C_3^{k_3}$.
\end{proof}
\begin{rmq}\label{rmq: prepro trace 2}
    We work over the algebraically closed field $\CC$. Since $G^b_{\dd}$ is linearly reductive, \Cref{thm:invariant procesi and zubkov} implies that the invariant ring
    $$\CC[\Mgot_{\dd}]=\CC[\mu^{-1}(0)]^{G^b_{\dd}}$$
    is generated by the cycle invariants defined in \Cref{defn:invariant tr relation}.
    It is therefore sufficient to use the trace-equivalence relations (see \Cref{rmq: prepro trace}) to show that the cycles under consideration generate all cycle invariants.
\end{rmq}

\begin{lem}\label{table Av-a}
    Let $n \geq 2$ and set $(\widetilde{A}_{2n-2},v\text{-}a)$ denote the orthogonal cyclic quiver with $2n-1$ vertices:
    $$
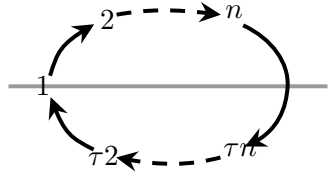
\begin{tikzpicture}[x=0.6pt,y=0.6pt,yscale=-1,xscale=1]
    
    \draw [color={rgb, 255:red, 155; green, 155; blue, 155 }  ,draw opacity=1 ][fill={rgb, 255:red, 155; green, 155; blue, 155 }  ,fill opacity=1 ][line width=1.5]    (250,150) -- (450,150) ;
    \draw [line width=1.5]    (276.03,143.38) .. controls (280.62,128.53) and (282.3,123.84) .. (300.41,113.04) ;
    \draw [shift={(303.74,111.09)}, rotate = 149.01] [fill={rgb, 255:red, 0; green, 0; blue, 0 }  ][line width=0.08]  [draw opacity=0] (13.4,-6.43) -- (0,0) -- (13.4,6.44) -- (8.9,0) -- cycle    ;
    \draw [line width=1.5]  [dash pattern={on 5.63pt off 4.5pt}]  (317.74,105.09) .. controls (336.96,101.79) and (352.11,101.34) .. (379.25,104.6) ;
    \draw [shift={(383.17,105.09)}, rotate = 186.82] [fill={rgb, 255:red, 0; green, 0; blue, 0 }  ][line width=0.08]  [draw opacity=0] (13.4,-6.43) -- (0,0) -- (13.4,6.44) -- (8.9,0) -- cycle    ;
    \draw [line width=1.5]  [dash pattern={on 5.63pt off 4.5pt}]  (382.88,194.63) .. controls (364.72,199.41) and (346.88,200.23) .. (321.15,195.35) ;
    \draw [shift={(317.45,194.63)}, rotate = 10.54] [fill={rgb, 255:red, 0; green, 0; blue, 0 }  ][line width=0.08]  [draw opacity=0] (13.4,-6.43) -- (0,0) -- (13.4,6.44) -- (8.9,0) -- cycle    ;
    \draw [line width=1.5]    (397.13,111.77) .. controls (415.8,120.43) and (425.8,135.77) .. (425,150) .. controls (424.25,163.38) and (419.14,178.2) .. (401.64,186.3) ;
    \draw [shift={(398.13,187.77)}, rotate = 334.45] [fill={rgb, 255:red, 0; green, 0; blue, 0 }  ][line width=0.08]  [draw opacity=0] (13.4,-6.43) -- (0,0) -- (13.4,6.44) -- (8.9,0) -- cycle    ;
    \draw [line width=1.5]    (303.37,189.68) .. controls (289.55,182.58) and (285.22,180.11) .. (277.73,160.4) ;
    \draw [shift={(276.39,156.78)}, rotate = 69.01] [fill={rgb, 255:red, 0; green, 0; blue, 0 }  ][line width=0.08]  [draw opacity=0] (13.4,-6.43) -- (0,0) -- (13.4,6.44) -- (8.9,0) -- cycle    ;
    
    \draw (265.54,142.23) node [anchor=north west][inner sep=0.75pt]  [font=\normalsize]  {$1$};
    \draw (305.79,100.98) node [anchor=north west][inner sep=0.75pt]  [font=\normalsize]  {$2$};
    \draw (384.54,97.98) node [anchor=north west][inner sep=0.75pt]  [font=\normalsize]  {$n$};
    \draw (383.29,184.48) node [anchor=north west][inner sep=0.75pt]  [font=\normalsize]  {$\tau n$};
    \draw (298.29,185.98) node [anchor=north west][inner sep=0.75pt]  [font=\normalsize]  {$\tau 2$};
    \end{tikzpicture}$$
    The fixed arrow $n \to \tau n$ leads to two distinct families of group actions:
    $$\OO(V_1) \times \GL(V_2) \times \cdots \times \GL(V_n) \curvearrowright \Hom(V_1,V_2)\oplus \cdots \oplus \Hom(V_{n-1},V_n) \oplus \left|
\begin{array}{l}
        \Sym(V_n)\\
        \Lambda^2(V_n)
\end{array}
\right.$$
    If the sign satisfies $s(n \to \tau n) = +1$, then there is a closed immersion
    $$\Mgot_{\delta} \inj \VV(xy-z^{2n-1}).$$
    If the sign satisfies $s(n \to \tau n) = -1$, then $\Mgot_{\delta}$ is a reduced point, and there is a closed immersion
    $$\Mgot_{2\delta} \inj \VV(xy-z^{4n-2}).$$
\end{lem}

\begin{proof}
    If the sign satisfies $s(n \to \tau n)=+1$ and the dimension vector is $\delta=(1,\ldots,1)$, then every cycle $C$ satisfies $C=\tr(C)$.
    We consider the ring
    $$\CC[C_1,C_2,C_3]\subset\CC[\mu^{-1}(0)].$$
    \Cref{prop:topo cycle} gives the relation
    $C_1C_2=C_3^{2n-1}$ in $\Pi(\widetilde{A}_{2n-2})$, and hence in
    $\CC[\Mgot_{\delta}]$. Therefore, we obtain the closed immersion
    $$\Mgot_{\delta} \inj \VV(xy-z^{2n-1}) \subset \CC^3.$$
    
    Now assume that $s(n\to\tau n)=-1$.
    In dimension $\delta$, we obtain
    $C_1\sim_{\tr}C_2\sim_{\tr}0$, and moreover
    $C_3\sim_{\tr}0$ by using the short loop between $n$ and $\tau n$ together with \Cref{lem: tlp short loop}.
    
    In dimension $2\delta$, the representations of the cycles
    $C_1,C_2,C_3$ (see \Cref{path morphism}) take values in
    $\End_{\CC}(\CC^2)$. Using the Cayley--Hamilton relation (see \Cref{Hamilton-Cayley scheme morphism}), we study the ring
    $$\CC[\Mgot_{2\delta}] = \CC\Bigg[\begin{matrix}
        \tr(C_1),\tr(C_2),\tr(C_3) \\
        \det(C_1),\det(C_2),\det(C_3) \\
        \tr(C_1C_3),\tr(C_2C_3)
    \end{matrix}\Bigg] \subset \CC[\mu^{-1}(0)].$$
    
    By \Cref{lem:tracetau}, and using that
    \begin{align*}
        & C_1 \sim_{\tr} \tau C_1 \sim_{\tr} -C_1,\quad
        C_1C_3 \sim_{\tr} \tau C_3\tau C_1 \sim_{\tr} -C_1C_3,\\
        & C_2 \sim_{\tr} \tau C_2 \sim_{\tr} -C_2,\quad
        C_2C_3 \sim_{\tr} \tau C_3\tau C_2 \sim_{\tr} -C_2C_3,
    \end{align*}
    we obtain
    $\tr(C_1)=\tr(C_2)=\tr(C_1C_3)=\tr(C_2C_3)=0$.
    Hence, we are reduced to studying
    $$\CC[\Mgot_{2\delta}]
    =\CC[\tr(C_3),\det(C_1),\det(C_2),\det(C_3)].$$
    
    \Cref{lem: tlp short loop} shows that $C_3$ is trace-equivalent to the short loop between $n$ and $\tau n$, which is represented by skew-symmetric matrices. Thus,
    $\tr(C_3)^2=4\det(C_3)$.
    The relation $C_1C_2=C_3^{2n-1}\in\Pi(\widetilde{A}_{2n-2})$ implies that $\det(C_1)\det(C_2)=\det(C_3)^{2n-1}$.
    Therefore,
    $$\Mgot_{2\delta} \inj \VV(xy-z^{4n-2}) \subset \CC^3.$$
\end{proof}

\begin{lem}\label{table Aa-a}
    Let $n \geq 1$, and let $(\widetilde{A}_{2n-1},a\text{-}a)$ denote the orthogonal cyclic quiver with $2n$ vertices:
    $$
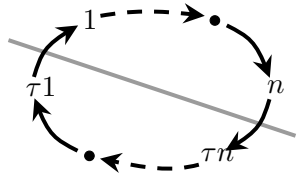
\begin{tikzpicture}[x=0.6pt,y=0.6pt,yscale=-1,xscale=1]
    
    \draw [color={rgb, 255:red, 155; green, 155; blue, 155 }  ,draw opacity=1 ][line width=1.5]    (260.07,120.3) -- (440.07,180.3) ;
    \draw  [color={rgb, 255:red, 0; green, 0; blue, 0 }  ,draw opacity=1 ][fill={rgb, 255:red, 0; green, 0; blue, 0 }  ,fill opacity=1 ] (386.5,108) .. controls (386.5,106.34) and (387.84,105) .. (389.5,105) .. controls (391.16,105) and (392.5,106.34) .. (392.5,108) .. controls (392.5,109.66) and (391.16,111) .. (389.5,111) .. controls (387.84,111) and (386.5,109.66) .. (386.5,108) -- cycle ;
    \draw [line width=1.5]    (276.03,143.38) .. controls (280.62,128.53) and (282.3,123.84) .. (300.41,113.04) ;
    \draw [shift={(303.74,111.09)}, rotate = 149.01] [fill={rgb, 255:red, 0; green, 0; blue, 0 }  ][line width=0.08]  [draw opacity=0] (13.4,-6.43) -- (0,0) -- (13.4,6.44) -- (8.9,0) -- cycle    ;
    \draw [line width=1.5]  [dash pattern={on 5.63pt off 4.5pt}]  (317.74,105.09) .. controls (336.96,101.79) and (352.11,101.34) .. (379.25,104.6) ;
    \draw [shift={(383.17,105.09)}, rotate = 186.82] [fill={rgb, 255:red, 0; green, 0; blue, 0 }  ][line width=0.08]  [draw opacity=0] (13.4,-6.43) -- (0,0) -- (13.4,6.44) -- (8.9,0) -- cycle    ;
    \draw [line width=1.5]  [dash pattern={on 5.63pt off 4.5pt}]  (379.44,198.4) .. controls (361.28,203.18) and (346.58,200.56) .. (321.13,195.38) ;
    \draw [shift={(317.45,194.63)}, rotate = 11.5] [fill={rgb, 255:red, 0; green, 0; blue, 0 }  ][line width=0.08]  [draw opacity=0] (13.4,-6.43) -- (0,0) -- (13.4,6.44) -- (8.9,0) -- cycle    ;
    \draw [line width=1.5]    (423.74,157.38) .. controls (419.15,172.23) and (417.46,176.92) .. (399.36,187.71) ;
    \draw [shift={(396.03,189.66)}, rotate = 329.01] [fill={rgb, 255:red, 0; green, 0; blue, 0 }  ][line width=0.08]  [draw opacity=0] (13.4,-6.43) -- (0,0) -- (13.4,6.44) -- (8.9,0) -- cycle    ;
    \draw [line width=1.5]    (396.96,111.07) .. controls (410.79,118.17) and (415.12,120.64) .. (422.6,140.35) ;
    \draw [shift={(423.95,143.97)}, rotate = 249.01] [fill={rgb, 255:red, 0; green, 0; blue, 0 }  ][line width=0.08]  [draw opacity=0] (13.4,-6.43) -- (0,0) -- (13.4,6.44) -- (8.9,0) -- cycle    ;
    \draw [line width=1.5]    (303.37,189.68) .. controls (289.55,182.58) and (285.22,180.11) .. (277.73,160.4) ;
    \draw [shift={(276.39,156.78)}, rotate = 69.01] [fill={rgb, 255:red, 0; green, 0; blue, 0 }  ][line width=0.08]  [draw opacity=0] (13.4,-6.43) -- (0,0) -- (13.4,6.44) -- (8.9,0) -- cycle    ;
    \draw  [color={rgb, 255:red, 0; green, 0; blue, 0 }  ,draw opacity=1 ][fill={rgb, 255:red, 0; green, 0; blue, 0 }  ,fill opacity=1 ] (307.83,192.89) .. controls (307.83,191.23) and (309.18,189.89) .. (310.83,189.89) .. controls (312.49,189.89) and (313.83,191.23) .. (313.83,192.89) .. controls (313.83,194.55) and (312.49,195.89) .. (310.83,195.89) .. controls (309.18,195.89) and (307.83,194.55) .. (307.83,192.89) -- cycle ;
    
    \draw (305.08,99.87) node [anchor=north west][inner sep=0.75pt]  [font=\normalsize]  {$1$};
    \draw (269.38,141.89) node [anchor=north west][inner sep=0.75pt]  [font=\normalsize]  {$\tau 1$};
    \draw (420.77,144.11) node [anchor=north west][inner sep=0.75pt]  [font=\normalsize]  {$n$};
    \draw (379.17,186.36) node [anchor=north west][inner sep=0.75pt]  [font=\normalsize]  {$\tau n$};
    \end{tikzpicture}$$
    The signs determine three different families of group actions, according to the value of $s(\tau 1 \to 1)+s(n \to \tau n)\in\{-2,0,2\}$.
    $$\GL(V_1) \times \cdots \times \GL(V_n) \curvearrowright \left|
\begin{array}{l}
        \Sym(V_1)\\
        \Lambda^2(V_1)
\end{array}
\right. \oplus \Hom(V_1,V_2)\oplus \cdots \oplus \Hom(V_{n-1},V_n) \oplus \left|
\begin{array}{l}
        \Sym(V_n)\\
        \Lambda^2(V_n)
\end{array}
\right.$$
    
    If $s(\tau 1 \to 1)=s(n \to \tau n)=+1$, then there is a closed immersion
    $$\Mgot_{\delta} \inj \VV(xy-z^{2n}).$$
    
    If $s(\tau 1 \to 1)=+1$ and $s(n \to \tau n)=-1$, hence $\Mgot_{\delta}$ is a reduced point, and there is a closed immersion
    $$\Mgot_{2\delta} \inj \VV(xy-z^{4n}).$$
    
    If $s(\tau 1 \to 1)=s(n \to \tau n)=-1$, hence $\Mgot_{\delta}$ is a reduced point, and there is a closed immersion
    $$\Mgot_{2\delta} \inj \VV(xy-z^{2n}).$$
\end{lem}

\begin{proof}
    In dimension $\delta$, when $s(\tau 1 \to 1)=s(n \to \tau n)=+1$, every cycle satisfies $C=\tr(C)$. Hence
    $$\CC[\Mgot_{\delta}] = \CC[\tr(C_1), \tr(C_2), \tr(C_3)] \subset \CC[\mu^{-1}(0)].$$
    By \Cref{prop:topo cycle}, we have the relation
    $\tr(C_1)\tr(C_2)=\tr(C_3)^{2n}$, and therefore
    $$\Mgot_{\delta} \inj \VV(xy-z^{2n}) \subset \CC^3.$$
    
    For the sign choice $s(\tau 1 \to 1)=+1$ and $s(n \to \tau n)=-1$, in dimension $2\delta$, using the Cayley--Hamilton relations (see \Cref{Hamilton-Cayley scheme morphism}), we are reduced to
    $$\CC[\Mgot_{2\delta}] = \CC\Bigg[\begin{matrix}
        \tr(C_1),\tr(C_2),\tr(C_3) \\
        \det(C_1),\det(C_2),\det(C_3) \\
        \tr(C_1C_3),\tr(C_2C_3)
    \end{matrix}\Bigg] \subset \CC[\mu^{-1}(0)].$$
    By \Cref{lem:tracetau}, applied to $C_1$, $C_2$, $C_1C_3$, and $C_2C_3$, we obtain
    $$\tr(C_1)=\tr(C_2)=\tr(C_1C_3)=\tr(C_2C_3)=0.$$
    Hence, we are reduced to studying
    $$\CC[\Mgot_{2\delta}]
    =\CC[\tr(C_3),\det(C_1),\det(C_2),\det(C_3)].$$
    \Cref{lem: tlp short loop} shows that $C_3$ is trace-equivalent to the short loop between $n$ and $\tau n$, which is represented by a product of skew-symmetric matrices. Hence $C_3$ is a scalar matrix, and therefore
    $$\tr(C_3)^2=4\det(C_3).$$
    Moreover, \Cref{prop:topo cycle} gives the relation
    $C_1C_2=C_3^{2n}$, hence
    $$\det(C_1)\det(C_2)=\det(C_3)^{2n}.$$
    Therefore,
    $$\Mgot_{2\delta} \inj \VV(xy-z^{4n}) \subset \CC^3.$$
    
    In dimension $\delta$, the relation
    $C_1C_2=C_3^{2n}$ still holds, but at the same time
    $C_1=C_2=0$ (in the sense of \Cref{path morphism}) by \Cref{lem:tracetau}, and $C_3=0$ by \Cref{lem: tlp short loop}.
    
    Finally, suppose that
    $s(\tau 1 \to 1)=s(n \to \tau n)=-1$.
    In dimension $2\delta$, \Cref{lem: tlp short loop} shows that $C_3$ is trace-equivalent to the short loop between $n$ and $\tau n$, which is represented by a product of two skew-symmetric matrices. Hence $C_3$ is a scalar matrix, and therefore
    $$\tr(C_3)^2=4\det(C_3).$$
    Similarly, $C_1$ and $C_2$ can be decomposed as the product of a path of length $2n-1$ from $1$ to $\tau1$ and a path of length $1$ from $\tau1$ to $1$. Consequently, $C_1$ and $C_2$ are also scalar matrices.
    Hence, we are reduced to studying
    $$\CC[\Mgot_{2\delta}]
    =\CC[\tr(C_1),\tr(C_2),\tr(C_3)]
    \subset\CC[\mu^{-1}(0)].$$
    The relation
    $C_1C_2=C_3^{2n}$ is obtained from \Cref{prop:topo cycle}. Therefore,
    $$\Mgot_{2\delta} \inj \VV(xy-z^{2n}) \subset \CC^3.$$
    
    In dimension $\delta$, we have
    $C_1=C_2=0$, since there are no nonzero skew-symmetric matrices. Then $C_3=0$ by \Cref{lem: tlp short loop}.
\end{proof}

\begin{lem}\label{table Av-v}
    Let $n \geq 1$, and denote by $(\widetilde{A}_{2n-1},v\text{-}v)$ the following orthogonal quiver.
    $$
    
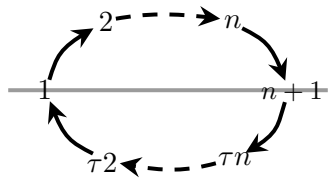
\begin{tikzpicture}[x=0.6pt,y=0.6pt,yscale=-1,xscale=1]
    
    \draw [color={rgb, 255:red, 155; green, 155; blue, 155 }  ,draw opacity=1 ][line width=1.5]    (250.24,149.92) -- (450.64,149.92) ;
    \draw [line width=1.5]    (276.03,143.38) .. controls (280.62,128.53) and (282.3,123.84) .. (300.41,113.04) ;
    \draw [shift={(303.74,111.09)}, rotate = 149.01] [fill={rgb, 255:red, 0; green, 0; blue, 0 }  ][line width=0.08]  [draw opacity=0] (13.4,-6.43) -- (0,0) -- (13.4,6.44) -- (8.9,0) -- cycle    ;
    \draw [line width=1.5]  [dash pattern={on 5.63pt off 4.5pt}]  (317.74,105.09) .. controls (336.96,101.79) and (352.11,101.34) .. (379.25,104.6) ;
    \draw [shift={(383.17,105.09)}, rotate = 186.82] [fill={rgb, 255:red, 0; green, 0; blue, 0 }  ][line width=0.08]  [draw opacity=0] (13.4,-6.43) -- (0,0) -- (13.4,6.44) -- (8.9,0) -- cycle    ;
    \draw [line width=1.5]  [dash pattern={on 5.63pt off 4.5pt}]  (380.45,195.72) .. controls (362.29,200.51) and (349.17,201.33) .. (323.86,196.45) ;
    \draw [shift={(320.2,195.72)}, rotate = 10.75] [fill={rgb, 255:red, 0; green, 0; blue, 0 }  ][line width=0.08]  [draw opacity=0] (13.4,-6.43) -- (0,0) -- (13.4,6.44) -- (8.9,0) -- cycle    ;
    \draw [line width=1.5]    (423.74,157.38) .. controls (419.15,172.23) and (420.07,176.53) .. (402.26,187.28) ;
    \draw [shift={(398.95,189.22)}, rotate = 328.64] [fill={rgb, 255:red, 0; green, 0; blue, 0 }  ][line width=0.08]  [draw opacity=0] (13.4,-6.43) -- (0,0) -- (13.4,6.44) -- (8.9,0) -- cycle    ;
    \draw [line width=1.5]    (396.96,111.07) .. controls (410.79,118.17) and (415.12,120.64) .. (422.6,140.35) ;
    \draw [shift={(423.95,143.97)}, rotate = 249.01] [fill={rgb, 255:red, 0; green, 0; blue, 0 }  ][line width=0.08]  [draw opacity=0] (13.4,-6.43) -- (0,0) -- (13.4,6.44) -- (8.9,0) -- cycle    ;
    \draw [line width=1.5]    (303.37,189.68) .. controls (289.55,182.58) and (285.22,180.11) .. (277.73,160.4) ;
    \draw [shift={(276.39,156.78)}, rotate = 69.01] [fill={rgb, 255:red, 0; green, 0; blue, 0 }  ][line width=0.08]  [draw opacity=0] (13.4,-6.43) -- (0,0) -- (13.4,6.44) -- (8.9,0) -- cycle    ;
    
    \draw (267.04,142.23) node [anchor=north west][inner sep=0.75pt]  [font=\normalsize]  {$1$};
    \draw (305.54,99.73) node [anchor=north west][inner sep=0.75pt]  [font=\normalsize]  {$2$};
    \draw (298.04,187.98) node [anchor=north west][inner sep=0.75pt]  [font=\normalsize]  {$\tau 2$};
    \draw (384.15,101.52) node [anchor=north west][inner sep=0.75pt]  [font=\normalsize]  {$n$};
    \draw (380.36,187.55) node [anchor=north west][inner sep=0.75pt]  [font=\normalsize]  {$\tau n$};
    \draw (407.29,143.23) node [anchor=north west][inner sep=0.75pt]  [font=\small]  {$n+1$};
    \end{tikzpicture}
    $$
    which describes the family of group representations
    $$\OO(V_1)\times\GL(V_2)\times\cdots\times\GL(V_n)\times\OO(V_{n+1})
    \curvearrowright
    \Hom(V_1,V_2)\oplus\cdots\oplus\Hom(V_n,V_{n+1}).$$
    There is a closed immersion
    $$\Mgot_{\delta}\inj\VV(xy-z^{2n}).$$
\end{lem}

\begin{proof}
    The proof follows exactly the same argument as in the previous case $(\widetilde{A}_{2n-2},v\text{-}a)$ with sign $+1$, with $2n-1$ replaced by $2n$.
\end{proof}

\begin{lem}\label{table A c}
    Let $n\geq 1$ and let $(\widetilde{A}_{2n-1},c)$ denote the following quiver:
    $$
    
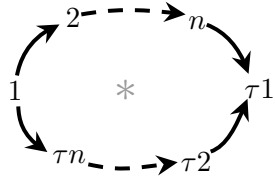
\begin{tikzpicture}[x=0.6pt,y=0.6pt,yscale=-1,xscale=1]
    
    \draw [line width=1.5]  [dash pattern={on 5.63pt off 4.5pt}]  (317.74,105.09) .. controls (336.96,101.79) and (352.11,101.34) .. (379.25,104.6) ;
    \draw [shift={(383.17,105.09)}, rotate = 186.82] [fill={rgb, 255:red, 0; green, 0; blue, 0 }  ][line width=0.08]  [draw opacity=0] (13.4,-6.43) -- (0,0) -- (13.4,6.44) -- (8.9,0) -- cycle    ;
    \draw [line width=1.5]  [dash pattern={on 5.63pt off 4.5pt}]  (322.13,199.1) .. controls (340.05,202.94) and (346.91,202.48) .. (374.63,198.31) ;
    \draw [shift={(378.2,197.77)}, rotate = 171.48] [fill={rgb, 255:red, 0; green, 0; blue, 0 }  ][line width=0.08]  [draw opacity=0] (13.4,-6.43) -- (0,0) -- (13.4,6.44) -- (8.9,0) -- cycle    ;
    \draw [line width=1.5]    (277.4,143.37) .. controls (278.97,135.22) and (283.48,121.18) .. (300.63,112.88) ;
    \draw [shift={(304.07,111.37)}, rotate = 151.07] [fill={rgb, 255:red, 0; green, 0; blue, 0 }  ][line width=0.08]  [draw opacity=0] (13.4,-6.43) -- (0,0) -- (13.4,6.44) -- (8.9,0) -- cycle    ;
    \draw [line width=1.5]    (278.2,162.77) .. controls (278.6,171.22) and (280.48,180.58) .. (293.16,189.32) ;
    \draw [shift={(296.47,191.43)}, rotate = 216.49] [fill={rgb, 255:red, 0; green, 0; blue, 0 }  ][line width=0.08]  [draw opacity=0] (13.4,-6.43) -- (0,0) -- (13.4,6.44) -- (8.9,0) -- cycle    ;
    \draw [line width=1.5]    (398.87,191.1) .. controls (408.88,187.62) and (414.75,175.47) .. (420.46,161.61) ;
    \draw [shift={(421.96,157.96)}, rotate = 112.49] [fill={rgb, 255:red, 0; green, 0; blue, 0 }  ][line width=0.08]  [draw opacity=0] (13.4,-6.43) -- (0,0) -- (13.4,6.44) -- (8.9,0) -- cycle    ;
    \draw [line width=1.5]    (396.4,111.51) .. controls (408.4,116.19) and (416.49,125.89) .. (421.35,138.74) ;
    \draw [shift={(422.62,142.4)}, rotate = 249.65] [fill={rgb, 255:red, 0; green, 0; blue, 0 }  ][line width=0.08]  [draw opacity=0] (13.4,-6.43) -- (0,0) -- (13.4,6.44) -- (8.9,0) -- cycle    ;
    
    \draw (336.76,144.13) node [anchor=north west][inner sep=0.75pt]  [font=\LARGE,color={rgb, 255:red, 155; green, 155; blue, 155 }  ,opacity=1 ]  {$\ast $};
    \draw (270.15,145.55) node [anchor=north west][inner sep=0.75pt]  [font=\normalsize]  {$1$};
    \draw (418.53,143.84) node [anchor=north west][inner sep=0.75pt]  [font=\normalsize]  {$\tau 1$};
    \draw (306.43,99.84) node [anchor=north west][inner sep=0.75pt]  [font=\normalsize]  {$2$};
    \draw (378.53,190.08) node [anchor=north west][inner sep=0.75pt]  [font=\normalsize]  {$\tau 2$};
    \draw (383.29,104.98) node [anchor=north west][inner sep=0.75pt]  [font=\normalsize]  {$n$};
    \draw (298.15,188.7) node [anchor=north west][inner sep=0.75pt]  [font=\normalsize]  {$\tau n$};
\end{tikzpicture}
    $$
    which describes the group representation
    $$\GL(V_1)\times\cdots\times\GL(V_n)
    \curvearrowright
    \Hom(V_1,V_2)\oplus\cdots\oplus\Hom(V_n,V_1^*).$$
    Hence $\Mgot_{\delta}$ is a reduced point, and there is a closed immersion
    $$\Mgot_{2\delta}\inj\VV(x^{n+1}-y^2x-z^2).$$
\end{lem}

\begin{proof}
    In dimension $2\delta$, taking into account the different possible orientations of the quiver, \Cref{prop:topo cycle} gives
    $C_1C_2=(-C_3^2)^n$ in $\Pi(\widetilde{A}_{2n-1})$.
    We are reduced to studying
    $$\CC[\Mgot_{2\delta}] = \CC\Bigg[\begin{matrix}
        \tr(C_1),\tr(C_2),\tr(C_3) \\
        \det(C_1),\det(C_2),\det(C_3) \\
        \tr(C_1C_3),\tr(C_2C_3)
    \end{matrix}\Bigg] \subset \CC[\mu^{-1}(0)].$$
    
    Using \Cref{lem:tracetau}, we obtain
    $C_3 \sim_{\tr} \tau C_3 \sim_{\tr} -C_3$, hence $\tr(C_3)=0$, and
    $C_1\sim_{\tr} C_2,\ C_1C_3\sim_{\tr} C_2C_3$.
    
    \begin{align*}
        \det(C_1) & = \det\left(\cycleCone\right) = \det\left(\cycleConehaut\right)\det\left(\cycleConebas\right) \\
        \intertext{Applying $\tau$ to the right path, we obtain}
        \det(C_1) & = \det\left(\cycleConehaut\right)\det\left(\cycleConebashaut\right) = \det\left(\cycleConehomotopietrivial\right) = \det(C_3)^n.
    \end{align*}
    
    Applying the Cayley--Hamilton theorem three times (see \Cref{Hamilton-Cayley scheme morphism}), we obtain the following relations:
    \begin{align*}
        (C_3C_1)^2 & - \tr(C_3C_1)C_3C_1 + \det(C_3C_1)1_2 = 0, \\
        C_3^2C_1^2 & = (\tr(C_3)C_3 - \det(C_3)1_2)(\tr(C_1)C_1-\det(C_1)1_2) \\
        & = \det(C_3)^{n+1}\cdot 1_2 - \tr(C_1)\det(C_3)C_1, \\
        \det(C_3)^{n+1}\cdot 1_2 & - \tr(C_1)\det(C_3)C_1 - \tr(C_3C_1)C_3C_1 + \det(C_3C_1)1_2 = 0, \\
        & 4\det(C_3)^{n+1} - \tr(C_1)^2\det(C_3) - \tr(C_3C_1)^2 = 0.
    \end{align*}
    Therefore,
    $$\Mgot_{2\delta} \inj \VV(x^{n+1}-y^2x-z^2) \subset \CC^3.$$
    
    In dimension $\delta$, we have $\tr(C_1)=\tr(C_3)^n$ and $\tr(C_3)=-\tr(C_3)=0$ by \Cref{lem:tracetau,lem: tlp short loop}. Hence,
    $\Mgot_\delta$ is a reduced point.
\end{proof}

\subsection{Symmetries of \texorpdfstring{$\widetilde{D}$}{D~tilde} quivers}\label{symmetries Dtilde}

Let $(\widetilde{D}_{2n-2},v)$ (respectively $(\widetilde{D}_{2n-1},a)$), for $n\geq 3$, denote the orthogonal quiver with respectively $2n-1$ (respectively $2n$ if $n\neq\tau n$) vertices:
$$
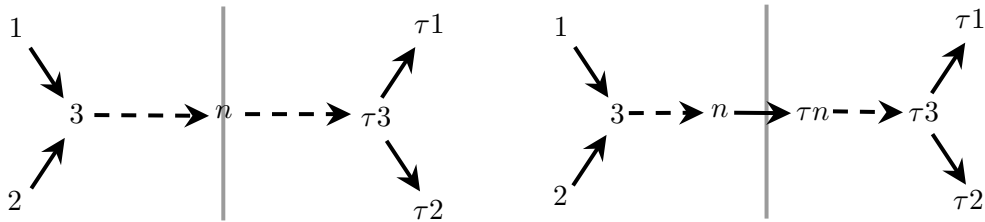
\begin{tikzpicture}[x=0.7pt,y=0.7pt,yscale=-1,xscale=1]

\draw [color={rgb, 255:red, 155; green, 155; blue, 155 }  ,draw opacity=1 ][line width=1.5]    (178,103) -- (178,218) ;
\draw [line width=1.5]    (74.16,125.2) -- (89.95,149.06) ;
\draw [shift={(92.16,152.4)}, rotate = 236.5] [fill={rgb, 255:red, 0; green, 0; blue, 0 }  ][line width=0.08]  [draw opacity=0] (13.4,-6.43) -- (0,0) -- (13.4,6.44) -- (8.9,0) -- cycle    ;
\draw [line width=1.5]  [dash pattern={on 5.63pt off 4.5pt}]  (108.7,161.35) -- (166.33,161.74) ;
\draw [shift={(170.33,161.77)}, rotate = 180.39] [fill={rgb, 255:red, 0; green, 0; blue, 0 }  ][line width=0.08]  [draw opacity=0] (13.4,-6.43) -- (0,0) -- (13.4,6.44) -- (8.9,0) -- cycle    ;
\draw [line width=1.5]  [dash pattern={on 5.63pt off 4.5pt}]  (189.7,160.85) -- (242.7,160.85) ;
\draw [shift={(246.7,160.85)}, rotate = 180] [fill={rgb, 255:red, 0; green, 0; blue, 0 }  ][line width=0.08]  [draw opacity=0] (13.4,-6.43) -- (0,0) -- (13.4,6.44) -- (8.9,0) -- cycle    ;
\draw [color={rgb, 255:red, 155; green, 155; blue, 155 }  ,draw opacity=1 ][line width=1.5]    (470,102) -- (470,217) ;
\draw [line width=1.5]  [dash pattern={on 5.63pt off 4.5pt}]  (395.96,160.3) -- (432.27,160.12) ;
\draw [shift={(436.27,160.1)}, rotate = 179.72] [fill={rgb, 255:red, 0; green, 0; blue, 0 }  ][line width=0.08]  [draw opacity=0] (13.4,-6.43) -- (0,0) -- (13.4,6.44) -- (8.9,0) -- cycle    ;
\draw [line width=1.5]  [dash pattern={on 5.63pt off 4.5pt}]  (505.7,159.35) -- (539.36,159.85) ;
\draw [shift={(543.36,159.91)}, rotate = 180.85] [fill={rgb, 255:red, 0; green, 0; blue, 0 }  ][line width=0.08]  [draw opacity=0] (13.4,-6.43) -- (0,0) -- (13.4,6.44) -- (8.9,0) -- cycle    ;
\draw [line width=1.5]    (452.6,160.1) -- (479.2,160.32) ;
\draw [shift={(483.2,160.35)}, rotate = 180.47] [fill={rgb, 255:red, 0; green, 0; blue, 0 }  ][line width=0.08]  [draw opacity=0] (13.4,-6.43) -- (0,0) -- (13.4,6.44) -- (8.9,0) -- cycle    ;
\draw [line width=1.5]    (265.76,175.2) -- (281.55,199.06) ;
\draw [shift={(283.76,202.4)}, rotate = 236.5] [fill={rgb, 255:red, 0; green, 0; blue, 0 }  ][line width=0.08]  [draw opacity=0] (13.4,-6.43) -- (0,0) -- (13.4,6.44) -- (8.9,0) -- cycle    ;
\draw [line width=1.5]    (366.96,124.8) -- (369.59,128.78) -- (382.75,148.66) ;
\draw [shift={(384.96,152)}, rotate = 236.5] [fill={rgb, 255:red, 0; green, 0; blue, 0 }  ][line width=0.08]  [draw opacity=0] (13.4,-6.43) -- (0,0) -- (13.4,6.44) -- (8.9,0) -- cycle    ;
\draw [line width=1.5]    (559.2,171.6) -- (574.99,195.46) ;
\draw [shift={(577.2,198.8)}, rotate = 236.5] [fill={rgb, 255:red, 0; green, 0; blue, 0 }  ][line width=0.08]  [draw opacity=0] (13.4,-6.43) -- (0,0) -- (13.4,6.44) -- (8.9,0) -- cycle    ;
\draw [line width=1.5]    (558.8,149.6) -- (574.59,125.74) ;
\draw [shift={(576.8,122.4)}, rotate = 123.5] [fill={rgb, 255:red, 0; green, 0; blue, 0 }  ][line width=0.08]  [draw opacity=0] (13.4,-6.43) -- (0,0) -- (13.4,6.44) -- (8.9,0) -- cycle    ;
\draw [line width=1.5]    (366,199.2) -- (381.79,175.34) ;
\draw [shift={(384,172)}, rotate = 123.5] [fill={rgb, 255:red, 0; green, 0; blue, 0 }  ][line width=0.08]  [draw opacity=0] (13.4,-6.43) -- (0,0) -- (13.4,6.44) -- (8.9,0) -- cycle    ;
\draw [line width=1.5]    (263.2,152) -- (278.99,128.14) ;
\draw [shift={(281.2,124.8)}, rotate = 123.5] [fill={rgb, 255:red, 0; green, 0; blue, 0 }  ][line width=0.08]  [draw opacity=0] (13.4,-6.43) -- (0,0) -- (13.4,6.44) -- (8.9,0) -- cycle    ;
\draw [line width=1.5]    (74.8,200.8) -- (90.59,176.94) ;
\draw [shift={(92.8,173.6)}, rotate = 123.5] [fill={rgb, 255:red, 0; green, 0; blue, 0 }  ][line width=0.08]  [draw opacity=0] (13.4,-6.43) -- (0,0) -- (13.4,6.44) -- (8.9,0) -- cycle    ;

\draw (61.29,107.98) node [anchor=north west][inner sep=0.75pt]  [font=\normalsize]  {$1$};
\draw (60.62,200.98) node [anchor=north west][inner sep=0.75pt]  [font=\normalsize]  {$2$};
\draw (93.96,154.32) node [anchor=north west][inner sep=0.75pt]  [font=\normalsize]  {$3$};
\draw (172.29,154.98) node [anchor=north west][inner sep=0.75pt]  [font=\normalsize]  {$n$};
\draw (250.62,155.65) node [anchor=north west][inner sep=0.75pt]  [font=\normalsize]  {$\tau 3$};
\draw (279.29,105.98) node [anchor=north west][inner sep=0.75pt]  [font=\normalsize]  {$\tau 1$};
\draw (278.62,205.65) node [anchor=north west][inner sep=0.75pt]  [font=\normalsize]  {$\tau 2$};
\draw (354.29,107.65) node [anchor=north west][inner sep=0.75pt]  [font=\normalsize]  {$1$};
\draw (353.96,198.32) node [anchor=north west][inner sep=0.75pt]  [font=\normalsize]  {$2$};
\draw (384.62,154.32) node [anchor=north west][inner sep=0.75pt]  [font=\normalsize]  {$3$};
\draw (438.96,154.32) node [anchor=north west][inner sep=0.75pt]  [font=\normalsize]  {$n$};
\draw (484.62,155.65) node [anchor=north west][inner sep=0.75pt]  [font=\normalsize]  {$\tau n$};
\draw (544.96,153.32) node [anchor=north west][inner sep=0.75pt]  [font=\normalsize]  {$\tau 3$};
\draw (569.96,103.32) node [anchor=north west][inner sep=0.75pt]  [font=\normalsize]  {$\tau 1$};
\draw (568.96,200.65) node [anchor=north west][inner sep=0.75pt]  [font=\normalsize]  {$\tau 2$};
\end{tikzpicture}$$
Recall that the fundamental root $\delta$ of a quiver of type $\widetilde{D}$ is the dimension vector
$$\delta=(1,1,2,\ldots,2,1,1),$$
with value $1$ at the vertices of valency one and value $2$ at all other vertices.

We denote $C_1, C_2, C_3$ the following different cycles:
$$
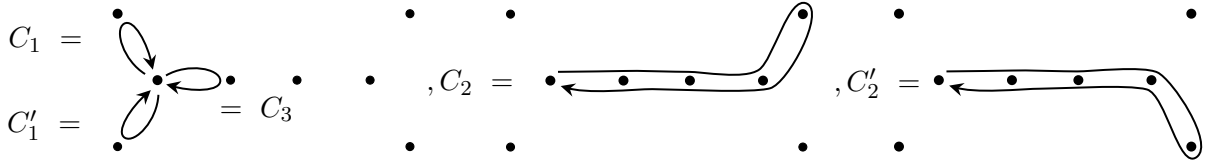
\begin{tikzpicture}[x=0.5pt,y=0.5pt,yscale=-1,xscale=1]

\draw  [color={rgb, 255:red, 0; green, 0; blue, 0 }  ,draw opacity=1 ][fill={rgb, 255:red, 0; green, 0; blue, 0 }  ,fill opacity=1 ][line width=0.75]  (6,205.67) .. controls (6,204.01) and (7.34,202.67) .. (9,202.67) .. controls (10.66,202.67) and (12,204.01) .. (12,205.67) .. controls (12,207.32) and (10.66,208.67) .. (9,208.67) .. controls (7.34,208.67) and (6,207.32) .. (6,205.67) -- cycle ;
\draw  [color={rgb, 255:red, 0; green, 0; blue, 0 }  ,draw opacity=1 ][fill={rgb, 255:red, 0; green, 0; blue, 0 }  ,fill opacity=1 ] (61,205.67) .. controls (61,204.01) and (62.34,202.67) .. (64,202.67) .. controls (65.66,202.67) and (67,204.01) .. (67,205.67) .. controls (67,207.32) and (65.66,208.67) .. (64,208.67) .. controls (62.34,208.67) and (61,207.32) .. (61,205.67) -- cycle ;
\draw  [color={rgb, 255:red, 0; green, 0; blue, 0 }  ,draw opacity=1 ][fill={rgb, 255:red, 0; green, 0; blue, 0 }  ,fill opacity=1 ] (166,205.67) .. controls (166,204.01) and (167.34,202.67) .. (169,202.67) .. controls (170.66,202.67) and (172,204.01) .. (172,205.67) .. controls (172,207.32) and (170.66,208.67) .. (169,208.67) .. controls (167.34,208.67) and (166,207.32) .. (166,205.67) -- cycle ;
\draw  [color={rgb, 255:red, 0; green, 0; blue, 0 }  ,draw opacity=1 ][fill={rgb, 255:red, 0; green, 0; blue, 0 }  ,fill opacity=1 ] (-24,255.67) .. controls (-24,254.01) and (-22.66,252.67) .. (-21,252.67) .. controls (-19.34,252.67) and (-18,254.01) .. (-18,255.67) .. controls (-18,257.32) and (-19.34,258.67) .. (-21,258.67) .. controls (-22.66,258.67) and (-24,257.32) .. (-24,255.67) -- cycle ;
\draw  [color={rgb, 255:red, 0; green, 0; blue, 0 }  ,draw opacity=1 ][fill={rgb, 255:red, 0; green, 0; blue, 0 }  ,fill opacity=1 ][line width=0.75]  (-24,155.67) .. controls (-24,154.01) and (-22.66,152.67) .. (-21,152.67) .. controls (-19.34,152.67) and (-18,154.01) .. (-18,155.67) .. controls (-18,157.32) and (-19.34,158.67) .. (-21,158.67) .. controls (-22.66,158.67) and (-24,157.32) .. (-24,155.67) -- cycle ;
\draw  [color={rgb, 255:red, 0; green, 0; blue, 0 }  ,draw opacity=1 ][fill={rgb, 255:red, 0; green, 0; blue, 0 }  ,fill opacity=1 ] (196,255.67) .. controls (196,254.01) and (197.34,252.67) .. (199,252.67) .. controls (200.66,252.67) and (202,254.01) .. (202,255.67) .. controls (202,257.32) and (200.66,258.67) .. (199,258.67) .. controls (197.34,258.67) and (196,257.32) .. (196,255.67) -- cycle ;
\draw  [color={rgb, 255:red, 0; green, 0; blue, 0 }  ,draw opacity=1 ][fill={rgb, 255:red, 0; green, 0; blue, 0 }  ,fill opacity=1 ] (196,155.67) .. controls (196,154.01) and (197.34,152.67) .. (199,152.67) .. controls (200.66,152.67) and (202,154.01) .. (202,155.67) .. controls (202,157.32) and (200.66,158.67) .. (199,158.67) .. controls (197.34,158.67) and (196,157.32) .. (196,155.67) -- cycle ;
\draw  [color={rgb, 255:red, 0; green, 0; blue, 0 }  ,draw opacity=1 ][fill={rgb, 255:red, 0; green, 0; blue, 0 }  ,fill opacity=1 ] (111,205.67) .. controls (111,204.01) and (112.34,202.67) .. (114,202.67) .. controls (115.66,202.67) and (117,204.01) .. (117,205.67) .. controls (117,207.32) and (115.66,208.67) .. (114,208.67) .. controls (112.34,208.67) and (111,207.32) .. (111,205.67) -- cycle ;
\draw [line width=0.75]    (-0.32,201.07) .. controls (-16.63,191) and (-23.85,167.92) .. (-16.19,163.25) .. controls (-8.92,158.81) and (1.64,177.81) .. (5.07,195.12) ;
\draw [shift={(5.55,197.84)}, rotate = 255.77] [fill={rgb, 255:red, 0; green, 0; blue, 0 }  ][line width=0.08]  [draw opacity=0] (10.72,-5.15) -- (0,0) -- (10.72,5.15) -- (7.12,0) -- cycle    ;
\draw [line width=0.75]    (9.69,216.7) .. controls (8.91,235.85) and (-7.66,253.46) .. (-15.49,249.08) .. controls (-22.93,244.91) and (-11.54,226.39) .. (1.87,214.92) ;
\draw [shift={(4,213.17)}, rotate = 136.41] [fill={rgb, 255:red, 0; green, 0; blue, 0 }  ][line width=0.08]  [draw opacity=0] (10.72,-5.15) -- (0,0) -- (10.72,5.15) -- (7.12,0) -- cycle    ;
\draw [line width=0.75]    (15.41,202.18) .. controls (31.93,192.47) and (55.71,196.87) .. (56.26,205.82) .. controls (56.79,214.33) and (35.05,214.79) .. (18.15,209.72) ;
\draw [shift={(15.52,208.88)}, rotate = 13.63] [fill={rgb, 255:red, 0; green, 0; blue, 0 }  ][line width=0.08]  [draw opacity=0] (10.72,-5.15) -- (0,0) -- (10.72,5.15) -- (7.12,0) -- cycle    ;
\draw  [color={rgb, 255:red, 0; green, 0; blue, 0 }  ,draw opacity=1 ][fill={rgb, 255:red, 0; green, 0; blue, 0 }  ,fill opacity=1 ][line width=0.75]  (301.67,206) .. controls (301.67,204.34) and (303.01,203) .. (304.67,203) .. controls (306.32,203) and (307.67,204.34) .. (307.67,206) .. controls (307.67,207.66) and (306.32,209) .. (304.67,209) .. controls (303.01,209) and (301.67,207.66) .. (301.67,206) -- cycle ;
\draw  [color={rgb, 255:red, 0; green, 0; blue, 0 }  ,draw opacity=1 ][fill={rgb, 255:red, 0; green, 0; blue, 0 }  ,fill opacity=1 ][line width=0.75]  (356.67,206) .. controls (356.67,204.34) and (358.01,203) .. (359.67,203) .. controls (361.32,203) and (362.67,204.34) .. (362.67,206) .. controls (362.67,207.66) and (361.32,209) .. (359.67,209) .. controls (358.01,209) and (356.67,207.66) .. (356.67,206) -- cycle ;
\draw  [color={rgb, 255:red, 0; green, 0; blue, 0 }  ,draw opacity=1 ][fill={rgb, 255:red, 0; green, 0; blue, 0 }  ,fill opacity=1 ][line width=0.75]  (461.67,206) .. controls (461.67,204.34) and (463.01,203) .. (464.67,203) .. controls (466.32,203) and (467.67,204.34) .. (467.67,206) .. controls (467.67,207.66) and (466.32,209) .. (464.67,209) .. controls (463.01,209) and (461.67,207.66) .. (461.67,206) -- cycle ;
\draw  [color={rgb, 255:red, 0; green, 0; blue, 0 }  ,draw opacity=1 ][fill={rgb, 255:red, 0; green, 0; blue, 0 }  ,fill opacity=1 ] (271.67,256) .. controls (271.67,254.34) and (273.01,253) .. (274.67,253) .. controls (276.32,253) and (277.67,254.34) .. (277.67,256) .. controls (277.67,257.66) and (276.32,259) .. (274.67,259) .. controls (273.01,259) and (271.67,257.66) .. (271.67,256) -- cycle ;
\draw  [color={rgb, 255:red, 0; green, 0; blue, 0 }  ,draw opacity=1 ][fill={rgb, 255:red, 0; green, 0; blue, 0 }  ,fill opacity=1 ] (271.67,156) .. controls (271.67,154.34) and (273.01,153) .. (274.67,153) .. controls (276.32,153) and (277.67,154.34) .. (277.67,156) .. controls (277.67,157.66) and (276.32,159) .. (274.67,159) .. controls (273.01,159) and (271.67,157.66) .. (271.67,156) -- cycle ;
\draw  [color={rgb, 255:red, 0; green, 0; blue, 0 }  ,draw opacity=1 ][fill={rgb, 255:red, 0; green, 0; blue, 0 }  ,fill opacity=1 ] (491.67,256) .. controls (491.67,254.34) and (493.01,253) .. (494.67,253) .. controls (496.32,253) and (497.67,254.34) .. (497.67,256) .. controls (497.67,257.66) and (496.32,259) .. (494.67,259) .. controls (493.01,259) and (491.67,257.66) .. (491.67,256) -- cycle ;
\draw  [color={rgb, 255:red, 0; green, 0; blue, 0 }  ,draw opacity=1 ][fill={rgb, 255:red, 0; green, 0; blue, 0 }  ,fill opacity=1 ][line width=0.75]  (491.67,156) .. controls (491.67,154.34) and (493.01,153) .. (494.67,153) .. controls (496.32,153) and (497.67,154.34) .. (497.67,156) .. controls (497.67,157.66) and (496.32,159) .. (494.67,159) .. controls (493.01,159) and (491.67,157.66) .. (491.67,156) -- cycle ;
\draw  [color={rgb, 255:red, 0; green, 0; blue, 0 }  ,draw opacity=1 ][fill={rgb, 255:red, 0; green, 0; blue, 0 }  ,fill opacity=1 ][line width=0.75]  (406.67,206) .. controls (406.67,204.34) and (408.01,203) .. (409.67,203) .. controls (411.32,203) and (412.67,204.34) .. (412.67,206) .. controls (412.67,207.66) and (411.32,209) .. (409.67,209) .. controls (408.01,209) and (406.67,207.66) .. (406.67,206) -- cycle ;
\draw [line width=0.75]    (310.4,199.43) .. controls (328.07,199.1) and (350.73,199.1) .. (360.4,198.77) .. controls (370.07,198.43) and (398.07,199.43) .. (408.73,199.43) .. controls (419.4,199.43) and (450.73,205.43) .. (461.73,197.77) .. controls (472.73,190.1) and (483.73,141.43) .. (497.73,148.1) .. controls (511.73,154.77) and (483.73,207.77) .. (470.07,211.77) .. controls (456.4,215.77) and (396.4,212.43) .. (381.4,212.43) .. controls (367.15,212.43) and (332.44,216.65) .. (315,211.06) ;
\draw [shift={(312.4,210.1)}, rotate = 16.35] [fill={rgb, 255:red, 0; green, 0; blue, 0 }  ][line width=0.08]  [draw opacity=0] (10.72,-5.15) -- (0,0) -- (10.72,5.15) -- (7.12,0) -- cycle    ;
\draw  [color={rgb, 255:red, 0; green, 0; blue, 0 }  ,draw opacity=1 ][fill={rgb, 255:red, 0; green, 0; blue, 0 }  ,fill opacity=1 ][line width=0.75]  (594,205.67) .. controls (594,204.01) and (595.34,202.67) .. (597,202.67) .. controls (598.66,202.67) and (600,204.01) .. (600,205.67) .. controls (600,207.32) and (598.66,208.67) .. (597,208.67) .. controls (595.34,208.67) and (594,207.32) .. (594,205.67) -- cycle ;
\draw  [color={rgb, 255:red, 0; green, 0; blue, 0 }  ,draw opacity=1 ][fill={rgb, 255:red, 0; green, 0; blue, 0 }  ,fill opacity=1 ][line width=0.75]  (649,205.67) .. controls (649,204.01) and (650.34,202.67) .. (652,202.67) .. controls (653.66,202.67) and (655,204.01) .. (655,205.67) .. controls (655,207.32) and (653.66,208.67) .. (652,208.67) .. controls (650.34,208.67) and (649,207.32) .. (649,205.67) -- cycle ;
\draw  [color={rgb, 255:red, 0; green, 0; blue, 0 }  ,draw opacity=1 ][fill={rgb, 255:red, 0; green, 0; blue, 0 }  ,fill opacity=1 ][line width=0.75]  (754,205.67) .. controls (754,204.01) and (755.34,202.67) .. (757,202.67) .. controls (758.66,202.67) and (760,204.01) .. (760,205.67) .. controls (760,207.32) and (758.66,208.67) .. (757,208.67) .. controls (755.34,208.67) and (754,207.32) .. (754,205.67) -- cycle ;
\draw  [color={rgb, 255:red, 0; green, 0; blue, 0 }  ,draw opacity=1 ][fill={rgb, 255:red, 0; green, 0; blue, 0 }  ,fill opacity=1 ][line width=0.75]  (564,255.67) .. controls (564,254.01) and (565.34,252.67) .. (567,252.67) .. controls (568.66,252.67) and (570,254.01) .. (570,255.67) .. controls (570,257.32) and (568.66,258.67) .. (567,258.67) .. controls (565.34,258.67) and (564,257.32) .. (564,255.67) -- cycle ;
\draw  [color={rgb, 255:red, 0; green, 0; blue, 0 }  ,draw opacity=1 ][fill={rgb, 255:red, 0; green, 0; blue, 0 }  ,fill opacity=1 ][line width=0.75]  (564,155.67) .. controls (564,154.01) and (565.34,152.67) .. (567,152.67) .. controls (568.66,152.67) and (570,154.01) .. (570,155.67) .. controls (570,157.32) and (568.66,158.67) .. (567,158.67) .. controls (565.34,158.67) and (564,157.32) .. (564,155.67) -- cycle ;
\draw  [color={rgb, 255:red, 0; green, 0; blue, 0 }  ,draw opacity=1 ][fill={rgb, 255:red, 0; green, 0; blue, 0 }  ,fill opacity=1 ][line width=0.75]  (784,255.67) .. controls (784,254.01) and (785.34,252.67) .. (787,252.67) .. controls (788.66,252.67) and (790,254.01) .. (790,255.67) .. controls (790,257.32) and (788.66,258.67) .. (787,258.67) .. controls (785.34,258.67) and (784,257.32) .. (784,255.67) -- cycle ;
\draw  [color={rgb, 255:red, 0; green, 0; blue, 0 }  ,draw opacity=1 ][fill={rgb, 255:red, 0; green, 0; blue, 0 }  ,fill opacity=1 ][line width=0.75]  (784,155.67) .. controls (784,154.01) and (785.34,152.67) .. (787,152.67) .. controls (788.66,152.67) and (790,154.01) .. (790,155.67) .. controls (790,157.32) and (788.66,158.67) .. (787,158.67) .. controls (785.34,158.67) and (784,157.32) .. (784,155.67) -- cycle ;
\draw  [color={rgb, 255:red, 0; green, 0; blue, 0 }  ,draw opacity=1 ][fill={rgb, 255:red, 0; green, 0; blue, 0 }  ,fill opacity=1 ][line width=0.75]  (699,205.67) .. controls (699,204.01) and (700.34,202.67) .. (702,202.67) .. controls (703.66,202.67) and (705,204.01) .. (705,205.67) .. controls (705,207.32) and (703.66,208.67) .. (702,208.67) .. controls (700.34,208.67) and (699,207.32) .. (699,205.67) -- cycle ;
\draw [line width=0.75]    (602.73,199.1) .. controls (620.4,198.77) and (643.07,198.77) .. (652.73,198.43) .. controls (662.4,198.1) and (690.4,199.1) .. (701.07,199.1) .. controls (711.73,199.1) and (750.32,194.67) .. (761.52,199.87) .. controls (772.72,205.07) and (804.72,253.07) .. (791.12,263.47) .. controls (777.52,273.87) and (767.92,216.67) .. (753.52,212.67) .. controls (739.12,208.67) and (688.73,212.1) .. (673.73,212.1) .. controls (659.48,212.1) and (624.78,216.31) .. (607.34,210.73) ;
\draw [shift={(604.73,209.77)}, rotate = 16.35] [fill={rgb, 255:red, 0; green, 0; blue, 0 }  ][line width=0.08]  [draw opacity=0] (10.72,-5.15) -- (0,0) -- (10.72,5.15) -- (7.12,0) -- cycle    ;

\draw (-105,164.07) node [anchor=north west][inner sep=0.75pt]    {$C_{1} \ =$};
\draw (-106,226.07) node [anchor=north west][inner sep=0.75pt]    {$C'_{1} \ =\ $};
\draw (54,216.57) node [anchor=north west][inner sep=0.75pt]    {$=\ C_{3}$};
\draw (209.67,196.4) node [anchor=north west][inner sep=0.75pt]    {$,C_{2} \ =$};
\draw (516,194.73) node [anchor=north west][inner sep=0.75pt]    {$,C'_{2} \ =$};
\end{tikzpicture}
$$
in the exceptional case where $n = \tau n = 3$, we define $C_3 = C_1 + C_1'$. Note that $\ell(C_1) = \ell(C_1') = \ell(C_3)=2$ and $\ell(C_2) = \ell(C_2')= 2L$ where :
$$ L = \begin{cases}
     2n-5& \text{ if } n = \tau n, \\
    2n-4 & \text{ if } n \ne \tau n.
\end{cases}$$

As in the study of symmetric $\widetilde{A}$ quivers, our goal is to classify cycles in the doubled quiver up to trace-equivalence.
In the present case, the presence of trivalent vertices requires an additional lemma describing the behavior of short loops $C_3$ and $\tau C_3$ at these vertices.

\begin{lem}\label{lem:relation phasme}
    The following relations hold in both $\Pi(\widetilde{D}_{2n-2})$ and $\Pi(\widetilde{D}_{2n-1})$.
    \begin{align*}
    \begin{matrix}
         C_1' = C_3 - C_1, & C_2'= C_3^L - C_2, \\
        C_1C_3 = C_3C_1',& C_2C_3 = C_3C_2', \\
        C_1'C_3 = C_3C_1, & C_2'C_3 = C_3C_2, \\
        \\
    \end{matrix}\quad 
    \begin{matrix}
        C_1^2 = (C_1')^{2}= (\tau C_1)^2 = (\tau C_1')^2= 0, \\
        (C_1C_2C_3)^2 = C_1C_2^2C_3^3-C_1C_2C_1C_3^{L+2} + (C_1C_2)^2C_3^2, \\
        C_2^2 = \begin{cases}
        C_2C_3^L & \text{ if } L \text{ even}, \\
        0 & \text{ if } L \text{ odd}.\\
    \end{cases}\\
    \end{matrix}
    \end{align*}
\end{lem}

\begin{proof}
    The relations $C_1' = C_3 - C_1$, $C_2' = C_3^L - C_2$, and
    $C_1^2 = (C_1')^2 = (\tau C_1)^2 = (\tau C_1')^2 = 0$
    follow from the moment map relations
    $\mu_3,\mu_{\tau3},\mu_1,\mu_2,\mu_{\tau1},\mu_{\tau2}$
    (see \Cref{rmq: relation moment}).
    
    The relation $C_1C_3 = C_3C_1'$ is obtained by combining
    $C_1^2 = (C_1')^2$ with $C_1' = C_3-C_1$.
    Similarly, one obtains the relations
    $C_2C_3=C_3C_2'$, $C_1'C_3=C_3C_1$, and
    $C_2'C_3=C_3C_2$.
    
    The relation
    $$(C_1C_2C_3)^2= C_1C_2^2C_3^3- C_1C_2C_1C_3^{L+2}+ (C_1C_2)^2C_3^2$$
    is obtained by
    \begin{align*}
        (C_1C_2C_3)^2 & = C_1C_2C_1'C_2'C_3^2 = C_1C_2(C_3 - C_1)(C_3^{L} - C_2)C_3^2 \\
        & = C_1C_2C_3^{L+3} - C_1C_2C_3C_2C_3^2 -C_1C_2C_1C_3^{L+2} + (C_1C_2)^2C_3^2 \\
        & = C_1C_2C_3^{L+3} - C_1C_2C_2'C_3^3 -C_1C_2C_1C_3^{L+2} + (C_1C_2)^2C_3^2 \\
        & = C_1C_2C_3^{L+3} - C_1C_2(C_3^L - C_2)C_3^3 -C_1C_2C_1C_3^{L+2} + (C_1C_2)^2C_3^2\\
        & = C_1C_2C_3^{L+3} - C_1C_2C_3^{L+3} + C_1C_2^2C_3^3 -C_1C_2C_1C_3^{L+2} + (C_1C_2)^2C_3^2, \\
        & = C_1C_2^2C_3^3 -C_1C_2C_1C_3^{L+2} + (C_1C_2)^2C_3^2,
    \end{align*}
    The last relation, concerning $C_2^2$, is
    \begin{align*}
        C_2^2 & = \left(\phasmeCtwo\right)^2 = \left(\riri\right) \left(\donald\right) \left(\fifi\right) \\
        & = \left(\riri\right) (\tau C_3)^{L-1} \left(\fifi\right) \\
        & = \left(\riri\right) \begin{cases}
            \left(\fifi\right) C_3^{L-1}& \text{ if } L-1\text{ even } \\
            \left(\loulou\right) C_3^{L-1}& \text{ if } L-1\text{ odd } \\
        \end{cases}\\
        & = \begin{cases}
            0& \text{ if } L\text{ odd } \\
            \left(\picsou\right) C_3^{L-1}& \text{ if } L\text{ even } \\
        \end{cases}\\
        & = \begin{cases}
            0& \text{ if } L\text{ odd } \\
            \left(\zaza\right) C_3^{L-1}& \text{ if } L\text{ even }\\
        \end{cases} = \begin{cases}
            0& \text{ if } L\text{ odd } \\
            C_2 C_3^{L}& \text{ if } L\text{ even. }\\
        \end{cases} \qedhere
    \end{align*}
\end{proof}

\begin{prop}\label{generation phasme}
    Any cycle $C$ contained in the doubled quiver of $\widetilde{D}_{2n-2}$ or $\widetilde{D}_{2n-1}$ $(n\geq 3)$ is trace-equivalent to a linear combination of cycles 
    \begin{align}\label{simple word}
        (C_1C_2)^{k_1}C_3^{k_2},C_2(C_1C_2)^{k_1}C_3^{k_2},C_1(C_2C_1)^{k_1}C_3^{k_2},(C_2C_1)^{k_1}C_3^{k_2} \text{ for pairs } (k_1,k_2)\in\NN^2.
    \end{align}
    Therefore, for any dimension vector $\dd$, we have:
    $$\CC[\Mgot_{\dd}]= \CC\left[
        \begin{matrix}
        \tr((C_1C_2)^{k_1}C_3^{k_2}),&
        \tr(C_2(C_1C_2)^{k_1}C_3^{k_2}),\\
        \tr(C_1(C_2C_1)^{k_1}C_3^{k_2}),&
        \tr((C_2C_1)^{k_1}C_3^{k_2})
        \end{matrix}
        \ \middle|\ k_1,k_2\in\NN
    \right] \subset \CC[\mu^{-1}(0)].$$
\end{prop}

\begin{proof}
    We proceed in two steps. First, we prove that every cycle in the doubled quiver is trace-equivalent to an element of
    $$\CC\langle C_1,C_1',C_2,C_2',C_3\rangle \subset \Pi(Q).$$
    Second, we prove by induction that any such cycle can be reduced to a linear combination of terms described in \eqref{simple word}.

    Using \Cref{cycle homotopically trivial,lem: tlp short loop}, any cycle is trace-equivalent to another cycle passing through $3$ or $\tau 3$. By \Cref{lem:tracetau,lem:pqqp}, we may further assume that it is based at $3$.
    For the first step, it suffices to consider cycles passing through $3$ exactly once. Let $C$ be such a cycle.
    If $C$ passes through $1$ or $2$, then it is equal to $C_1$ or $C_1'$. Otherwise, $C$ can be written as
    $$C = p_1 q_1 p_2 \cdots q_r p_{r+1},$$
    where the $p_k$ are paths between $3$ and $\tau 3$, and the $q_k$ are words in $\tau C_1$ and $\tau C_1'$.
    Since the paths $p_2,\ldots,p_r$ lie entirely on the line between $3$ and $\tau 3$, \Cref{cycle homotopically trivial} implies that each of them is equal to a power of $\tau C_3$.
    Using the relations from \Cref{lem:relation phasme}, we obtain
    $$ C = p_1 q'_1 q'_2 \cdots q'_r p_{r+1} C_3^{k}, $$
    for some integer $k$. Since each $q'_i$ is again either $\tau C_1$ or $\tau C_1'$, and $(\tau C_1)^2 = (\tau C_1')^2 = 0 \in \Pi(\widetilde{D}_n)$ (see \Cref{lem:relation phasme}), the cycle $C$ can be reduced to a linear combination of terms of the form $C_2C_3^{k}$ or $C_2'C_3^{k}$, for various integers $k$.

    We now proceed to the second step. Any cycle $C$ obtained as a product of $C_1,C_1',C_2,C_2',C_3$ can be written as a linear combination of terms described in \eqref{simple word}.
    We prove this by induction on the number $r$ of cycles among $C_1,C_1',C_2,C_2'$ appearing in $C$.
    If $r=0$, then $C$ is a power of $C_3$.
    Assume that a cycle $C$ contains $r+1$ cycles among $C_1,C_1',C_2,C_2'$. By \Cref{lem:relation phasme}, we may write
    $$C = C' C_3^{k},$$
    where $k \in \NN$, $C'$ contains no occurrence of $C_3$ and is a product of $r+1$ sub-cycles among $C_1,C_1',C_2,C_2'$.
    Using the relations of \Cref{lem:relation phasme},
    $$C_1' = C_3 - C_1,\qquad C_2' = C_3^L - C_2,$$
    we can eliminate every occurrence of $C_1'$ and $C_2'$ in $C'$.
    Expanding the resulting expression, we obtain a term $\pm C'' C_3^{k}$, where $C''$ is a product of $C_1$ and $C_2$, together with a linear combination of cycles to which the induction hypothesis applies.
    If $C''$ contains two consecutive occurrences of $C_1$, or two consecutive occurrences of $C_2$ with $L$ even, then $C''=0 \in \Pi(Q)$ by \Cref{lem:relation phasme}.
    If $C''$ contains two consecutive occurrences of $C_2$ with $L$ odd, then \Cref{lem:relation phasme} gives $C_2^2 = C_2 C_3^L$, and we can apply the induction hypothesis to the resulting term.
    
    As explained in \Cref{thm:invariant procesi and zubkov} and \Cref{rmq: prepro trace 2}, the ring
    $$\CC[\Mgot_{\dd}]=\CC[\mu^{-1}(0)]^{G^b_{\dd}}$$
    is generated by the cycle invariants defined in \Cref{defn:invariant tr relation}.
    Moreover, using the relations available in the preprojective algebra (see \Cref{rmq: prepro trace}), we deduce that the cycles in \eqref{simple word} also generate the invariant ring:
    $$\CC[\Mgot_{\dd}]= \CC\left[
        \begin{matrix}
        \tr((C_1C_2)^{k_1}C_3^{k_2}),&
        \tr(C_2(C_1C_2)^{k_1}C_3^{k_2}),\\
        \tr(C_1(C_2C_1)^{k_1}C_3^{k_2}),&
        \tr((C_2C_1)^{k_1}C_3^{k_2})
        \end{matrix}
        \ \middle|\ k_1,k_2\in\NN
    \right] \subset \CC[\mu^{-1}(0)].$$
\end{proof}

\begin{lem}\label{table Da}
    For $n \geq 3$, let $(\widetilde{D}_{2n-1},a)$ denote the following orthogonal quivers:
    $$
    
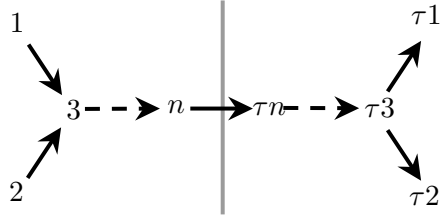
\begin{tikzpicture}[x=0.7pt,y=0.7pt,yscale=-1,xscale=1]
    
    \draw [color={rgb, 255:red, 155; green, 155; blue, 155 }  ,draw opacity=1 ][line width=1.5]    (353,95) -- (353,210) ;
    \draw [line width=1.5]  [dash pattern={on 5.63pt off 4.5pt}]  (278.96,153.3) -- (315.27,153.12) ;
    \draw [shift={(319.27,153.1)}, rotate = 179.72] [fill={rgb, 255:red, 0; green, 0; blue, 0 }  ][line width=0.08]  [draw opacity=0] (13.4,-6.43) -- (0,0) -- (13.4,6.44) -- (8.9,0) -- cycle    ;
    \draw [line width=1.5]  [dash pattern={on 5.63pt off 4.5pt}]  (385.6,152.77) -- (422.36,152.89) ;
    \draw [shift={(426.36,152.91)}, rotate = 180.2] [fill={rgb, 255:red, 0; green, 0; blue, 0 }  ][line width=0.08]  [draw opacity=0] (13.4,-6.43) -- (0,0) -- (13.4,6.44) -- (8.9,0) -- cycle    ;
    \draw [line width=1.5]    (335.6,153.1) -- (365.27,153.1) ;
    \draw [shift={(369.27,153.1)}, rotate = 180] [fill={rgb, 255:red, 0; green, 0; blue, 0 }  ][line width=0.08]  [draw opacity=0] (13.4,-6.43) -- (0,0) -- (13.4,6.44) -- (8.9,0) -- cycle    ;
    \draw [line width=1.5]    (248.66,118.7) -- (264.45,142.56) ;
    \draw [shift={(266.66,145.9)}, rotate = 236.5] [fill={rgb, 255:red, 0; green, 0; blue, 0 }  ][line width=0.08]  [draw opacity=0] (13.4,-6.43) -- (0,0) -- (13.4,6.44) -- (8.9,0) -- cycle    ;
    \draw [line width=1.5]    (441.66,164.7) -- (457.45,188.56) ;
    \draw [shift={(459.66,191.9)}, rotate = 236.5] [fill={rgb, 255:red, 0; green, 0; blue, 0 }  ][line width=0.08]  [draw opacity=0] (13.4,-6.43) -- (0,0) -- (13.4,6.44) -- (8.9,0) -- cycle    ;
    \draw [line width=1.5]    (248.16,190.4) -- (263.95,166.54) ;
    \draw [shift={(266.16,163.2)}, rotate = 123.5] [fill={rgb, 255:red, 0; green, 0; blue, 0 }  ][line width=0.08]  [draw opacity=0] (13.4,-6.43) -- (0,0) -- (13.4,6.44) -- (8.9,0) -- cycle    ;
    \draw [line width=1.5]    (441.66,143.9) -- (457.45,120.04) ;
    \draw [shift={(459.66,116.7)}, rotate = 123.5] [fill={rgb, 255:red, 0; green, 0; blue, 0 }  ][line width=0.08]  [draw opacity=0] (13.4,-6.43) -- (0,0) -- (13.4,6.44) -- (8.9,0) -- cycle    ;
    
    \draw (237.29,100.65) node [anchor=north west][inner sep=0.75pt]  [font=\normalsize]  {$1$};
    \draw (236.96,191.32) node [anchor=north west][inner sep=0.75pt]  [font=\normalsize]  {$2$};
    \draw (267.62,147.32) node [anchor=north west][inner sep=0.75pt]  [font=\normalsize]  {$3$};
    \draw (321.96,147.32) node [anchor=north west][inner sep=0.75pt]  [font=\normalsize]  {$n$};
    \draw (367.62,148.65) node [anchor=north west][inner sep=0.75pt]  [font=\normalsize]  {$\tau n$};
    \draw (427.96,146.32) node [anchor=north west][inner sep=0.75pt]  [font=\normalsize]  {$\tau 3$};
    \draw (452.96,96.32) node [anchor=north west][inner sep=0.75pt]  [font=\normalsize]  {$\tau 1$};
    \draw (451.96,193.65) node [anchor=north west][inner sep=0.75pt]  [font=\normalsize]  {$\tau 2$};

    \end{tikzpicture}
    $$
    The fixed arrow $n \to \tau n$ leads to two distinct families of group actions:
    $$\GL(V_1)\times \GL(V_2) \times \cdots \times \GL(V_n) \action \Hom(V_1,V_2)\oplus \Hom(V_2,V_3)\oplus \Hom(V_3,V_4) \oplus \cdots \oplus \left|
\begin{array}{l}
        \Sym(V_n)\\
        \Lambda^2(V_n)
\end{array}
\right.$$
    If the sign $s(n \to \tau n)=+1$, then there is a closed immersion
    $$\Mgot_{\delta} \inj \VV(z^2-x^{n-1}y+xy^2).$$
    If the sign $s(n\to \tau n)=-1$, then $\Mgot_\delta$ is a reduced point.
\end{lem}

\begin{proof}
    We consider a representation of fundamental dimension $\delta$ with $s(n \to \tau n) = +1$.  
    When two cycles $C$ and $C'$ meet at the same vertex in $\{1,2,\tau1,\tau2\}$, we have $\tr(CC') = \tr(C)\tr(C')$ (since $\delta_1 = \delta_2 = 1$).
    Using the Cayley--Hamilton theorem (see \Cref{Hamilton-Cayley scheme morphism}), we are reduced to studying
    $$\CC[\Mgot_{\delta}] = \CC\begin{bmatrix}
        \tr(C_1),\tr(C_2),\tr(C_3)\\
        \det(C_1),\det(C_2),\det(C_3) \\
        \tr(C_2C_3),\tr(C_1C_3),\tr(C_1C_2), \\
        \tr(C_1C_2C_3),\tr(C_2C_1C_3)
    \end{bmatrix} \subset \CC[\mu^{-1}(0)].$$
    Using \Cref{rmq: relation moment,Hamilton-Cayley scheme morphism}, we have
    \begin{align*}
        C_1^2 & = 0, \quad  \tr(C_1) = \tr(C_3) = \det(C_2) = \det(C_1) = 0, \\
         C_2^2 & = \tr(C_2)C_2, \quad  C_3^2 = -\det(C_3)1_2, \quad C_2 \sim_{\tr} C_1C_3^{L-1}\\
        (C_1C_2)^2 & = \tr(C_1C_2)C_1C_2, \quad (C_1C_2C_3)^2 = C_1C_2C_3\tr(C_1C_2C_3),
    \end{align*}
    and also \Cref{lem:relation phasme}
    \begin{align*}
        \tr(C_2C_1C_3) & = \tr(C_1C_3C_2) = \tr(C_1C_2'C_3) \\
        & = \tr(C_1(C_3^L - C_2)C_3)= \tr(C_1C_3^{L+1}) - \tr(C_1C_2C_3).\\
        \tr(C_1C_3) & = \tr(C_3C_1) = \tr(C_1'C_3) = \tr((C_3-C_1)C_3) = -\tr(C_1C_3) + \tr(C_3^2). \\
        2\tr(C_1C_3) & = \tr(C_3^2).
    \end{align*}
    We are reduced to studying 
    $$\CC[\Mgot_{\delta}] = \CC[\det(C_3),\tr(C_1C_2),\tr(C_1C_2C_3)] \subset \CC[\mu^{-1}(0)].$$
    Applying the relation 
    $$(C_1C_2C_3)^2= C_1C_2^2C_3^3- C_1C_2C_1C_3^{L+2}+ (C_1C_2)^2C_3^2$$
    from \Cref{lem:relation phasme} for $L = 2n - 4$: 
    \begin{align*}
        (C_1C_2C_3)^2 & = C_1C_2^2C_3^3-C_1C_2C_1C_3^{2n-2} + (C_1C_2)^2C_3^2, \\
        \tr(C_1C_2C_3)^2 & = \tr(C_1 C_2)(-\det(C_3))^{n-1}-\tr(C_1C_2C_1)(-\det(C_3))^{n-1} - \tr(C_1C_2)^2\det(C_3)\\
        & = \tr(C_1 C_2)(-\det(C_3))^{n-1} - \tr(C_1C_2)^2\det(C_3).
    \end{align*}
    There is a closed embedding $\Mgot_{\delta} \inj \VV(z^2-x^{n-1}y + xy^2)$.
    
    If we consider a representation of fundamental dimension $\delta$
    with $s(n\to\tau n)=-1$, the same argument as in the previous case applies.
    Applying \Cref{lem: tlp short loop,lem:pqqp}, we can find two cycles, trace-equivalent to $C_3$ and $C_1C_2$, based at the vertex $n$, whose representations are scalar multiples of each other.
    For $C_3$, this implies that $4\det(C_3) = \tr(C_3)^2 = 0$, and since $C_1C_2$ has a representation of rank at most one, we obtain $\tr(C_1C_2) = 0$ and
    $\tr(C_1C_2C_3) = \tr(C_1C_2)\tr(C_3) = 0$, then $\Mgot_\delta$ is a reduced point.
\end{proof}

\begin{lem}\label{table Dv}
    For $n \geq 3$, let $(\widetilde{D}_{2n-2},v)$ denote the following orthogonal quivers:
    $$
    \begin{tikzpicture}[x=0.7pt,y=0.7pt,yscale=-1,xscale=1]
    
    \draw [color={rgb, 255:red, 155; green, 155; blue, 155 }  ,draw opacity=1 ][line width=1.5]    (353,95) -- (353,210) ;
    \draw [line width=1.5]  [dash pattern={on 5.63pt off 4.5pt}]  (288.2,152.6) -- (337.9,152.48) ;
    \draw [shift={(341.9,152.48)}, rotate = 179.87] [fill={rgb, 255:red, 0; green, 0; blue, 0 }  ][line width=0.08]  [draw opacity=0] (13.4,-6.43) -- (0,0) -- (13.4,6.44) -- (8.9,0) -- cycle    ;
    \draw [line width=1.5]    (250.16,191.9) -- (265.95,168.04) ;
    \draw [shift={(268.16,164.7)}, rotate = 123.5] [fill={rgb, 255:red, 0; green, 0; blue, 0 }  ][line width=0.08]  [draw opacity=0] (13.4,-6.43) -- (0,0) -- (13.4,6.44) -- (8.9,0) -- cycle    ;
    \draw [line width=1.5]    (441.16,145.4) -- (456.95,121.54) ;
    \draw [shift={(459.16,118.2)}, rotate = 123.5] [fill={rgb, 255:red, 0; green, 0; blue, 0 }  ][line width=0.08]  [draw opacity=0] (13.4,-6.43) -- (0,0) -- (13.4,6.44) -- (8.9,0) -- cycle    ;
    \draw [line width=1.5]    (439.66,168.7) -- (455.45,192.56) ;
    \draw [shift={(457.66,195.9)}, rotate = 236.5] [fill={rgb, 255:red, 0; green, 0; blue, 0 }  ][line width=0.08]  [draw opacity=0] (13.4,-6.43) -- (0,0) -- (13.4,6.44) -- (8.9,0) -- cycle    ;
    \draw [line width=1.5]    (248.66,118.7) -- (264.45,142.56) ;
    \draw [shift={(266.66,145.9)}, rotate = 236.5] [fill={rgb, 255:red, 0; green, 0; blue, 0 }  ][line width=0.08]  [draw opacity=0] (13.4,-6.43) -- (0,0) -- (13.4,6.44) -- (8.9,0) -- cycle    ;
    \draw [line width=1.5]  [dash pattern={on 5.63pt off 4.5pt}]  (364.95,152.6) -- (414.65,152.48) ;
    \draw [shift={(418.65,152.48)}, rotate = 179.87] [fill={rgb, 255:red, 0; green, 0; blue, 0 }  ][line width=0.08]  [draw opacity=0] (13.4,-6.43) -- (0,0) -- (13.4,6.44) -- (8.9,0) -- cycle    ;
    
    \draw (236.29,99.98) node [anchor=north west][inner sep=0.75pt]  [font=\normalsize]  {$1$};
    \draw (235.62,192.98) node [anchor=north west][inner sep=0.75pt]  [font=\normalsize]  {$2$};
    \draw (268.96,146.32) node [anchor=north west][inner sep=0.75pt]  [font=\normalsize]  {$3$};
    \draw (347.29,146.98) node [anchor=north west][inner sep=0.75pt]  [font=\normalsize]  {$n$};
    \draw (425.62,147.65) node [anchor=north west][inner sep=0.75pt]  [font=\normalsize]  {$\tau 3$};
    \draw (454.29,97.98) node [anchor=north west][inner sep=0.75pt]  [font=\normalsize]  {$\tau 1$};
    \draw (453.62,197.65) node [anchor=north west][inner sep=0.75pt]  [font=\normalsize]  {$\tau 2$};
    \end{tikzpicture}
    $$
    which describes the group representation
    $$\GL(V_1)\times \GL(V_2) \times \cdots \times \OO(V_n) \action \Hom(V_1,V_2)\oplus \Hom(V_2,V_3)\oplus \Hom(V_3,V_4) \oplus \cdots \oplus \Hom(V_{n-1},V_n)$$
    Then we have $\Mgot_{\delta} \inj \VV(z^2 - x^{n-1}y + xy^2)$.
\end{lem}

\begin{proof}
    We consider a representation of fundamental dimension $\delta$.  
    Arguing as in the proof of \Cref{table Da} (with sign $+1$), we are reduced to studying:
    $$\CC[\Mgot_{\delta}] = \CC[\det(C_3),\tr(C_1C_2),\tr(C_1C_2C_3)] \subset \CC[\mu^{-1}(0)].$$
    
    Applying \Cref{lem:relation phasme} with $L=2n-5$, we obtain
    \begin{align*}
        (C_1C_2C_3)^2 & = C_1C_2^2C_3^3-C_1C_2C_1C_3^{2n-3} + (C_1C_2)^2C_3^2, \\
        (C_1C_2C_3)^2 & = -C_1C_2C_1C_3^{2n-3} + (C_1C_2)^2C_3^2, \\
        \tr(C_1C_2C_3)^2 & = -\tr(C_1C_2C_1C_3)(-\det(C_3))^{n-2} - \tr(C_1C_2)^2\det(C_3)\\
        & = -\tr(C_1C_2C_3C_1')(-\det(C_3))^{n-2} - \tr(C_1C_2)^2\det(C_3)\\
        & = -\tr(C_1C_2)(-\det(C_3))^{n-1} - \tr(C_1C_2)^2\det(C_3).\\
    \end{align*}
    There is a closed embedding $$\Mgot_{\delta} \inj\VV(z^2 - x^{n-1}y + xy^2).$$
\end{proof}

\subsection{Two-parameter families}\label{two parameter family}

In the preceding section, we computed embeddings of quiver varieties into singular surfaces.
In order to prove that these quiver varieties have dimension two, we construct in this section explicit two-dimensional families of the form
$${[tF_{\lambda}]}_{(t,\lambda)\in \CC^*\times U} \in \Mgot_{\dd}.$$

For a given orthogonal quiver, \Cref{lambda family} provides an explicit one-parameter family $(F_\lambda)_{\lambda\in U}$ of symmetric representations of the doubled quiver $\overline{Q}$, where $U\subset\CC$ is an open subset.
These representations are given explicitly by their coordinates and satisfy the moment map relations (see \Cref{rmq: relation moment}):
$$F_\lambda=(x_\lambda,y_\lambda)\in \Srep(Q,\dd)\oplus\Srep(Q^{\opp},\dd),
\qquad \text{such that } \forall \lambda\in U,\ F_\lambda\in\mu^{-1}(0).$$

For each family
$[tF_\lambda]_{(t,\lambda)\in\CC^*\times U}\in\Mgot_{\dd}$,
the second parameter $t\in\CC^*$ comes from the conical structure of
$$\Mgot_{\dd}=\mu^{-1}(0)\sslash G_{\dd}\subset
\Srep(\overline{Q},\dd)\sslash G_{\dd}.$$

\subsubsection{Asymptotic directions}\label{asymp direction}

In order to prove that the families of classes $[tF_\lambda]$ remain two-dimensional in the quiver varieties $\Mgot_{\dd}$, we proceed as follows.

First, we construct a $\CC^*$-equivariant embedding
$$\Mgot_{\dd}\hookrightarrow \CC^N$$
using $N$ invariant functions $(f_1,\cdots,f_N)$ associated with cycles, as given by \Cref{thm:invariant procesi and zubkov}.
Since the weight $(w_1,\cdots,w_N)$ of the conical $\CC^*$-action on an invariant of the form $\tr(\Lambda^k C)$ is given by $k\cdot \ell(C)$, we extend this action to $\CC^N$ using these weights, and denote by $w_{\max}$ the maximal weight.

In a second step, we compactify $\CC^N$ into the projective space
$$\PP_{\CC}^N=\CC^N\sqcup \partial_\infty,
\qquad \text{with } \partial_\infty=\{[\ast:\cdots:\ast:0]\}\subset\PP_{\CC}^N.$$

Then, for each fixed $\lambda$, the map in the parameter $t\in\CC^*$ extends to the compactification $\CC^*\subset\PP^1_{\CC}$ as follows:
\begin{align*}
\begin{matrix}
    & \CC^*\times U \longrightarrow \Mgot_{\dd}\\  
    & (t,\lambda) \mapsto t\cdot[F_{\lambda}], \\ 
    & \PP^1_{\CC}\times U \longrightarrow \PP^N_{\CC},\\
    & ([t:u],\lambda) \mapsto [t^{w_1}f_1(F_\lambda)u^{w_{\max}-w_1}:\cdots:t^{w_N}f_N(F_\lambda)u^{w_{\max}-w_N}:u^{w_{\max}}].
\end{matrix}
\end{align*}

When $t$ takes the value $\infty=[1:0]\in\PP^1_{\CC}$, it determines the asymptotic direction as $t\to\infty$, yielding a point on the boundary $\partial_\infty$.
If this limiting point depends nontrivially on the parameter $\lambda$, then the intersection
$$\partial_\infty\cap\overline{[tF_\lambda]}$$
contains a curve parametrized by $\lambda$ in the boundary of the compactification.
Consequently, the family $[tF_\lambda]_{(t,\lambda)}\subset\Mgot_{\dd}$ is two-dimensional.

\subsubsection{List of representations and asymptotic directions}\label{lambda family}

Whenever dotted arrows appear in the quiver, this indicates that the same matrix is repeated along all arrows represented by the dotted arrow.

For the orthogonal quivers $(\widetilde{A}_{2n-1},c)$, we choose the six invariants
$(x,y,z)=(\det(C_3),\tr(C_1),\tr(C_1C_3))$ for the embedding, together with
$(x^{n+1},y^2x,z^2)$ to detect nonzero asymptotic directions.
This yields an embedding
$$\Mgot_{2\delta}\inj \CC^6.$$

Consider now the following family of representations parametrized by $\lambda\in U=\CC^*$, where $J$ denotes the matrix in \eqref{matrice J}.
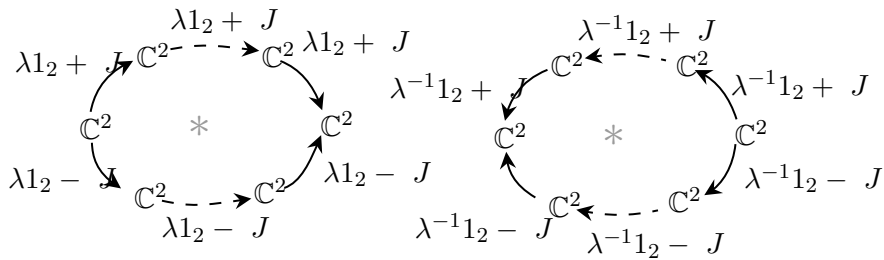
$$
\begin{tikzpicture}[x=0.6pt,y=0.6pt,yscale=-1,xscale=1]

\draw [line width=0.75]  [dash pattern={on 4.5pt off 4.5pt}]  (104.27,29.6) .. controls (123.79,26.25) and (129.96,26.76) .. (157.36,30.26) ;
\draw [shift={(159.97,30.59)}, rotate = 187.23] [fill={rgb, 255:red, 0; green, 0; blue, 0 }  ][line width=0.08]  [draw opacity=0] (10.72,-5.15) -- (0,0) -- (10.72,5.15) -- (7.12,0) -- cycle    ;
\draw [line width=0.75]  [dash pattern={on 4.5pt off 4.5pt}]  (98.93,124.6) .. controls (117.04,128.48) and (123.86,127.97) .. (152.3,123.68) ;
\draw [shift={(155,123.27)}, rotate = 171.48] [fill={rgb, 255:red, 0; green, 0; blue, 0 }  ][line width=0.08]  [draw opacity=0] (10.72,-5.15) -- (0,0) -- (10.72,5.15) -- (7.12,0) -- cycle    ;
\draw [line width=0.75]    (54.2,68.87) .. controls (55.79,60.59) and (60.42,46.23) .. (78.26,37.99) ;
\draw [shift={(80.87,36.87)}, rotate = 151.07] [fill={rgb, 255:red, 0; green, 0; blue, 0 }  ][line width=0.08]  [draw opacity=0] (10.72,-5.15) -- (0,0) -- (10.72,5.15) -- (7.12,0) -- cycle    ;
\draw [line width=0.75]    (55,88.27) .. controls (55.41,96.91) and (57.37,106.48) .. (70.8,115.39) ;
\draw [shift={(73.27,116.93)}, rotate = 216.49] [fill={rgb, 255:red, 0; green, 0; blue, 0 }  ][line width=0.08]  [draw opacity=0] (10.72,-5.15) -- (0,0) -- (10.72,5.15) -- (7.12,0) -- cycle    ;
\draw [line width=0.75]    (175.67,116.6) .. controls (185.9,113.05) and (191.8,100.44) .. (197.64,86.2) ;
\draw [shift={(198.76,83.46)}, rotate = 112.49] [fill={rgb, 255:red, 0; green, 0; blue, 0 }  ][line width=0.08]  [draw opacity=0] (10.72,-5.15) -- (0,0) -- (10.72,5.15) -- (7.12,0) -- cycle    ;
\draw [line width=0.75]    (173.2,37.01) .. controls (185.46,41.79) and (193.64,51.81) .. (198.46,65.09) ;
\draw [shift={(199.42,67.9)}, rotate = 249.65] [fill={rgb, 255:red, 0; green, 0; blue, 0 }  ][line width=0.08]  [draw opacity=0] (10.72,-5.15) -- (0,0) -- (10.72,5.15) -- (7.12,0) -- cycle    ;
\draw [line width=0.75]  [dash pattern={on 4.5pt off 4.5pt}]  (367.32,34.73) .. controls (385.28,31.81) and (394.68,31.93) .. (415,36.65) ;
\draw [shift={(364.09,35.27)}, rotate = 350.77] [fill={rgb, 255:red, 0; green, 0; blue, 0 }  ][line width=0.08]  [draw opacity=0] (10.72,-5.15) -- (0,0) -- (10.72,5.15) -- (7.12,0) -- cycle    ;
\draw [line width=0.75]  [dash pattern={on 4.5pt off 4.5pt}]  (361.79,130.9) .. controls (378.94,134.3) and (392.05,134.43) .. (410,130.65) ;
\draw [shift={(358.76,130.27)}, rotate = 11.22] [fill={rgb, 255:red, 0; green, 0; blue, 0 }  ][line width=0.08]  [draw opacity=0] (10.72,-5.15) -- (0,0) -- (10.72,5.15) -- (7.12,0) -- cycle    ;
\draw [line width=0.75]    (314.67,71.6) .. controls (316.91,62.73) and (322.79,49.69) .. (340.69,42.53) ;
\draw [shift={(314.03,74.53)}, rotate = 284.13] [fill={rgb, 255:red, 0; green, 0; blue, 0 }  ][line width=0.08]  [draw opacity=0] (10.72,-5.15) -- (0,0) -- (10.72,5.15) -- (7.12,0) -- cycle    ;
\draw [line width=0.75]    (315.05,96.99) .. controls (315.9,105.31) and (319.15,114.29) .. (333.09,122.6) ;
\draw [shift={(314.83,93.93)}, rotate = 84.17] [fill={rgb, 255:red, 0; green, 0; blue, 0 }  ][line width=0.08]  [draw opacity=0] (10.72,-5.15) -- (0,0) -- (10.72,5.15) -- (7.12,0) -- cycle    ;
\draw [line width=0.75]    (441.76,119.63) .. controls (450.8,113.97) and (457.14,102.8) .. (458.58,89.12) ;
\draw [shift={(439,121.15)}, rotate = 327.92] [fill={rgb, 255:red, 0; green, 0; blue, 0 }  ][line width=0.08]  [draw opacity=0] (10.72,-5.15) -- (0,0) -- (10.72,5.15) -- (7.12,0) -- cycle    ;
\draw [line width=0.75]    (435.9,43.93) .. controls (447.38,49.48) and (454.93,60) .. (459.25,73.57) ;
\draw [shift={(433.03,42.68)}, rotate = 25.78] [fill={rgb, 255:red, 0; green, 0; blue, 0 }  ][line width=0.08]  [draw opacity=0] (10.72,-5.15) -- (0,0) -- (10.72,5.15) -- (7.12,0) -- cycle    ;

\draw (113.56,69.63) node [anchor=north west][inner sep=0.75pt]  [font=\LARGE,color={rgb, 255:red, 155; green, 155; blue, 155 }  ,opacity=1 ]  {$\ast $};
\draw (45.67,69.07) node [anchor=north west][inner sep=0.75pt]  [font=\normalsize]  {$\mathbb{C}^{2}$};
\draw (81.34,24.07) node [anchor=north west][inner sep=0.75pt]  [font=\normalsize]  {$\mathbb{C}^{2}$};
\draw (160.01,23.73) node [anchor=north west][inner sep=0.75pt]  [font=\normalsize]  {$\mathbb{C}^{2}$};
\draw (196.67,67.07) node [anchor=north west][inner sep=0.75pt]  [font=\normalsize]  {$\mathbb{C}^{2}$};
\draw (155.01,109.07) node [anchor=north west][inner sep=0.75pt]  [font=\normalsize]  {$\mathbb{C}^{2}$};
\draw (80.01,112.07) node [anchor=north west][inner sep=0.75pt]  [font=\normalsize]  {$\mathbb{C}^{2}$};
\draw (373.39,75.3) node [anchor=north west][inner sep=0.75pt]  [font=\LARGE,color={rgb, 255:red, 155; green, 155; blue, 155 }  ,opacity=1 ]  {$\ast $};
\draw (305.5,74.73) node [anchor=north west][inner sep=0.75pt]  [font=\normalsize]  {$\mathbb{C}^{2}$};
\draw (341.17,29.73) node [anchor=north west][inner sep=0.75pt]  [font=\normalsize]  {$\mathbb{C}^{2}$};
\draw (419.83,29.4) node [anchor=north west][inner sep=0.75pt]  [font=\normalsize]  {$\mathbb{C}^{2}$};
\draw (456.5,72.73) node [anchor=north west][inner sep=0.75pt]  [font=\normalsize]  {$\mathbb{C}^{2}$};
\draw (414.83,114.73) node [anchor=north west][inner sep=0.75pt]  [font=\normalsize]  {$\mathbb{C}^{2}$};
\draw (339.83,117.73) node [anchor=north west][inner sep=0.75pt]  [font=\normalsize]  {$\mathbb{C}^{2}$};
\draw (5,29.4) node [anchor=north west][inner sep=0.75pt]    {$\lambda 1_{2} +\ J$};
\draw (102.5,2.4) node [anchor=north west][inner sep=0.75pt]    {$\lambda 1_{2} +\ J$};
\draw (185,16.4) node [anchor=north west][inner sep=0.75pt]    {$\lambda 1_{2} +\ J$};
\draw (1,102.4) node [anchor=north west][inner sep=0.75pt]    {$\lambda 1_{2} -\ J$};
\draw (97,133.4) node [anchor=north west][inner sep=0.75pt]    {$\lambda 1_{2} -\ J$};
\draw (198.5,98.4) node [anchor=north west][inner sep=0.75pt]    {$\lambda 1_{2} -\ J$};
\draw (240,43.4) node [anchor=north west][inner sep=0.75pt]    {$\lambda ^{-1} 1_{2} +\ J$};
\draw (358,6.9) node [anchor=north west][inner sep=0.75pt]    {$\lambda ^{-1} 1_{2} +\ J$};
\draw (454,41.4) node [anchor=north west][inner sep=0.75pt]    {$\lambda ^{-1} 1_{2} +\ J$};
\draw (462.5,99.9) node [anchor=north west][inner sep=0.75pt]    {$\lambda ^{-1} 1_{2} -\ J$};
\draw (363.5,140.9) node [anchor=north west][inner sep=0.75pt]    {$\lambda ^{-1} 1_{2} -\ J$};
\draw (256,130.9) node [anchor=north west][inner sep=0.75pt]    {$\lambda ^{-1} 1_{2} -\ J$};
\end{tikzpicture}
$$

Then the cycle $C_3$ (see \Cref{prop:topo cycle}) is represented by the product
$$
tF_\lambda(C_3)
=t(\lambda^{-1}1_2+J)t(\lambda1_2+J)
=t^2(1_2-1_2+J(\lambda+\lambda^{-1}))
=t^2J(\lambda+\lambda^{-1}),
$$
and hence
$$
\det(tF_\lambda(C_3))
=t^4\det(J)(\lambda+\lambda^{-1})^2
=t^4(\lambda+\lambda^{-1})^2.
$$

The weight of the action of $t\in\CC^*$ on cycle invariants is known (see \Cref{defn:invariant tr relation}); we only need to compute the values of the invariants depending on the parameter $\lambda\in\CC^*$:
\begin{align*}
    F_\lambda(C_1) = (\lambda^{-1}1_2 + J)^n(\lambda1_2-J)^n = (1_2 -J^2 + J(\lambda -\lambda^{-1}))^n = (21_2 + \lambda J - \lambda^{-1}J)^n.
\end{align*}

We want to compute $\tr(F_\lambda(C_1))$.
In the expanded form of the previous expression, the contributions of the $J$-terms vanish.
Therefore, it is sufficient to compute only the first and last terms in
$\CC[\lambda,\lambda^{-1}]$.
\begin{align*}
    y = \tr(F_\lambda(C_1)) & = \begin{cases}
        (-1)^{n/2}\lambda^n + \cdots + (-1)^{n/2}\lambda^{-n}, & n \text { even},\\
        2(-1)^{(n-1)/2}\lambda^{n-1}+ \cdots -2(-1)^{(n-1)/2}\lambda^{-(n-1)},& n \text{ odd},
    \end{cases}\\
    z = F_\lambda(C_1C_3) & = F_\lambda(C_1)F_\lambda(C_3) = (21_2 + \lambda J - \lambda^{-1}J)^nJ(\lambda+ \lambda^{-1}), \\
    \tr(F_\lambda(C_1C_3)) & = \begin{cases}
2( -1)^{n/2} \lambda^{n} +\cdots +2(-1)^{n/2}\lambda^{-n}, & n\ \text{even},\\
(-1)^{( n+1) /2} \lambda^{n+1} +\cdots +(-1)^{(n+1)/2}\lambda^{-(n+1)} , & n\ \text{ odd}.
\end{cases}\\
\end{align*}

As described in \Cref{asymp direction}, we consider the compactification
$$
\Mgot_{\dd}\subset
\CC^6_{\left(x,y,z,x^{n+1},xy^2,z^2\right)}
\subset \PP^6_{\CC},
$$
and extend the map $(t,\lambda)\mapsto [tF_\lambda]$ to the parameter
$t\in\PP^1_{\CC}$. 
This gives the asymptotic direction as
$t\to\infty$ in $\PP^1_{\CC}$.

\begin{align*}
    x^{n+1} = \det(C_3(tF_\lambda))^{n+1} & \underset{t \to \infty}{\to}  (\lambda +\lambda ^{-1})^{2n+2},\\
    y^{2} x = \tr(C_1(tF_\lambda))^2\det(C_3(tF_\lambda))& \underset{t \to \infty}{\to}  \begin{cases}
\lambda ^{2n+2} + \cdots + \lambda^{-(2n+2)} , & n\ \text{even},\\
-4\lambda ^{2n} + \cdots - 4\lambda ^{-2n} , & n\ \text{odd},
\end{cases}\\
z^{2} = \tr(C_1C_3(tF_\lambda))^2 & \underset{t \to \infty}{\to}  \begin{cases}
4\lambda ^{2n} +\cdots +4\lambda ^{-2n} , & n\ \text{even},\\
\lambda ^{2n+2} +\cdots  +\lambda ^{-(2n+2)} , & n\ \text{odd}.
\end{cases}
\end{align*}
For even values of $n$, the asymptotic direction of the family
$[tF_\lambda]\in\PP^6_{\CC}$ is given by the following coordinates in
$\partial_\infty$:
$$
[0,0,0,(\lambda+\lambda^{-1})^{2n+2},
\lambda^{2n+2}+\cdots+\lambda^{-(2n+2)},
4\lambda^{2n}+\cdots+4\lambda^{-2n},0].
$$
Hence, the family $[tF_\lambda]\in\PP^6_{\CC}$ is two-dimensional.

The same argument applies for odd values of $n$ and for the other orthogonal quivers considered in this section. We only state the resulting asymptotic directions.

For orthogonal quivers $(\widetilde{A}_{2n-2},v\text{-}a)$:
$$
$$

For orthogonal quivers $(\widetilde{A}_{2n-1},a\text{-}a)$:
$$%
$$
For orthogonal quivers $(\widetilde{A}_{2n-1},v\text{-}v)$:
$$%
$$
For orthogonal quivers $(\widetilde{D}_{2n-1},a)$
$$%
$$
For orthogonal quivers $(\widetilde{D}_{2n-2},v)$
$$%
$$

\newpage
\subsection{Proof of \Cref{Thm}}\label{proof of Thm}

\begin{proof}[Proof of \Cref{Thm}]\label{proof thm}
    Let $Q$ be an orthogonal quiver whose underlying graph is an affine Dynkin diagram.
    The quiver variety $\Mgot_{\dd}$ is an affine scheme over $\CC$, defined by its invariant ring:
    \begin{align*}
    \Mgot_{\dd}
    =\mu^{-1}(0)\sslash G^b_{\dd}
    =\Spec\bigl(\CC[\mu^{-1}(0)]^{G^b_{\dd}}\bigr).
    \end{align*}
    By \Cref{thm:invariant procesi and zubkov}, the invariant ring is generated by functions associated with cycles contained in the doubled quiver $\overline{Q}$.
    In the following cases:
    \begin{table}[H]
    \centering
        \renewcommand{\arraystretch}{1.2}
        \begin{tabular}{|c|c|c||c|}
        \hline
        Orthogonal quivers & sign of fixed arrow &  dimension $\dd$ & $\Mgot_{\dd}$\\
        \hline
    $(\widetilde{A}_{2n-2} ,v\text{-}a), n\geq 2$ &$-1$&$\delta$& $\{\pt\}$\\
        \hline
    $(\widetilde{A}_{2n-1} ,a\text{-}a),$ & $(+1,-1)$ & $\delta$ & $\{\pt\}$\\
        $n \geq 1$& $(-1,-1)$ & $\delta$ & $\{\pt\}$\\
        \hline
        $(\widetilde{A}_{2n-1} ,c), n \geq 1$& & $\delta$& $\{\pt\}$\\
        \hline
        $(\widetilde{D}_{2n-1} ,a), n \geq 3$& $-1$ & $\delta$ & $\{\pt\}$\\
        \hline 
        \end{tabular}
    \end{table}
    the quiver varieties are reduced points, as established by the computations in \Cref{table A c,table Aa-a,table Av-a,table Da}.
    In the remaining cases:
    \begin{table}[H]
    \centering
        \renewcommand{\arraystretch}{1.2}
        \begin{tabular}{|c|c|c||c|c|}
        \hline
        Orthogonal quivers & sign of fixed arrow &  dimension $\dd$ & $\Mgot_{\dd}$& Kleinian singularity\\
        \hline
    $(\widetilde{A}_{2n-2} ,v\text{-}a),n\geq 2$ &$+1$& $\delta$ & $\VV(xy-z^{2n-1})$& $A_{2n-2}$ \\
        &$-1$& $2\delta$ & $\VV(xy-z^{4n-2})$& $A_{4n-3}$\\
        \hline
          $(\widetilde{A}_{2n-1} ,a\text{-}a),n \geq 1$ & $(+1,+1)$ & $\delta$ & $\VV(xy-z^{2n})$ & $A_{2n-1}$\\
        & $(+1,-1)$ & $2\delta$ & $\VV(xy-z^{4n})$ & $A_{4n-1}$\\
        & $(-1,-1)$ & $2\delta$ & $\VV(xy-z^{2n})$ & $A_{2n-1}$\\
        \hline
    $(\widetilde{A}_{2n-1} ,v\text{-}v), n \geq 1$ && $\delta$ & $\VV(xy - z^{2n})$ & $A_{2n-1}$ \\
        \hline
        $(\widetilde{A}_{2n-1} ,c), n \geq 1$ & & $2\delta$ & $\VV(x^{n+1} - y^2x - z^2)$ & $D_{n+2}$ \\
        \hline 
        $(\widetilde{D}_{2n-2} ,v), n \geq 3$& & $\delta$ & $\VV(z^2 - x^{n-1}y + xy^2)$ & $D_{n+1}$ \\
        \hline
        $(\widetilde{D}_{2n-1} ,a),n\geq3$& $+1$ & $\delta$ & $\VV(z^2-x^{n-1}y + xy^2)$ & $D_{n+1}$ \\
        \hline 
        \end{tabular}
    \end{table}
    applying the graphical calculus developed in \Cref{graphical calculus}, we show in
    \Cref{table A c,table Aa-a,table Av-a,table Da,table Dv,table Av-v} that only three invariants generate the invariant ring $\CC[\Mgot_{\dd}]^{G_{\dd}^b}$.
    The three invariants define a closed embedding of schemes
    \begin{align*}
        \Mgot_{\dd}&\hookrightarrow \VV(P)\subset\CC^3,
    \end{align*}
    for some irreducible polynomial $P\in\CC[x,y,z]$ defining a surface with a Kleinian singularity of type $A$ or $D$.
    It remains to prove that each of these embeddings is an isomorphism.
    In \Cref{two parameter family}, we construct a two-parameter family
    $$(tF_\lambda)_{\lambda\in U,\ t\in\CC^*}\subset\mu^{-1}(0),$$
    whose image in the quotient $\Mgot_{\dd}$ has dimension $2$.
    Hence, the quiver variety has dimension at least two.
    Since $\VV(P)$ is a reduced irreducible surface, the closed embedding is necessarily surjective in each case. Therefore, it is a scheme isomorphism:
    \begin{align*}
    \Mgot_{\dd}&\simeq \VV(P). \qedhere
    \end{align*}
\end{proof}

\section*{Appendix: Notation used in the article}\label{section: notation}

\begin{description}
    \item[$Q$] a quiver; set of arrows,
    \item[$I$] set of vertices,
    \item[$\widetilde{A}_n,\widetilde{D}_n,\widetilde{E}_{6},\widetilde{E}_{7},\widetilde{E}_{8}$] affine Dynkin diagram,
    \item[$\delta$] minimal imaginary root of an affine Dynkin diagram,
  \item[$\tau$] involution on a quiver; see \Cref{defn:ortho quivers}
  \item[$Q^{\opp}$] opposite quiver; see \Cref{defn:opp/double quivers}
  \item[$\overline{Q}$] doubled quiver; see \Cref{defn:opp/double quivers}
  \item[$b$] signed form; see \Cref{defn:srep}
  \item[$J_i$] evaluation in the first variable of a signed form; see \Cref{defn:srep}
  \item[$J$] denotes the $2\times 2$ skew-symmetric matrix
  \begin{align}\label{matrice J}
      \begin{pmatrix}
    0&1\\
    -1&0 \\
\end{pmatrix} \in \mathrm{M}_2(\CC),
  \end{align}
  \item[$\repinvolution$] involution on the space of quiver representations; see \Cref{defn:srep},
  \item[$\Rep(Q,V),\Srep(Q,V)$] spaces of quiver representations and symmetric representations; see \Cref{defn:srep},
  \item[$\End_Q(v)$] set of endomorphisms of quiver representation,
  \item[$G_V$] base-change group; see \Cref{defn:group}
  \item[$G_V^b$] subgroup of $G_V$ that leaves $b$ invariant; see \Cref{defn:group}
  \item[$\ggot_V$,$\ggot_V^b$] Lie algebra of $G_V,G^b_V$; see \Cref{defn:group}
  \item[$\lieinvolution$] Lie algebra involution; see \Cref{defn:group}
  \item[$\kappa,\kappa^b$] perfect pairing on a Lie algebra; see \Cref{lem: pairing kappa}
  \item[$\mugras,\mu$] moment map for quiver and orthogonal quiver; see \Cref{defn:moment map quivers} and \Cref{moment of symmetric quivers}
  \item[$\mu_Q$] formal moment map, element of $\CC \overline{Q}$; see \Cref{defn:formal moment}
  \item[$\Mgotgras_{\zeta,\chi,\dd},\ \Mgot_{\zeta,\chi,\dd},\ \Mgot_{\dd}$] quotient variety $\mugras^{-1}(\zeta)\sslash_{\chi} G_{\dd},\ \mu^{-1}(\zeta)\sslash_{\chi} G_{\dd}$, $\mu^{-1}(0)\sslash G^b_V$; see \Cref{defn:quiv varieties},
  \item[$\CC Q$] path algebra of a quiver; see \Cref{defn:path algebra}
  \item[$\Pi(Q)$] preprojective algebra of a quiver; see \Cref{defn:preprojectiv algebra}
  \item[$\sim_{\tr}$] equivalence relation on cycles; see \Cref{defn:invariant tr relation}
  \item[$|Q|$] underlying topological space of a quiver; see \Cref{homotopic class}
\end{description}
\printbibliography

\end{document}